\documentclass[10pt,reqno]{amsart}

\usepackage{setspace}
\setdisplayskipstretch{1.5}
\newcommand*\centermathcell[1]{\omit\hfil$\displaystyle#1$\hfil\ignorespaces}

\usepackage[T1]{fontenc}
\usepackage{lmodern}

\usepackage{mathtools}
\usepackage{amsmath,amsfonts,amsbsy,amsgen,amscd,mathrsfs,amssymb,amsthm}
\usepackage{enumerate}
\usepackage{bm}
\usepackage{calc}

\usepackage{stmaryrd}
\SetSymbolFont{stmry}{bold}{U}{stmry}{m}{n} % suppress bm warning

\usepackage{footmisc}

\usepackage[usenames,dvipsnames]{xcolor}
\usepackage[colorlinks=true,citecolor=blue,linkcolor=blue]{hyperref}

\usepackage{tikz}
\usetikzlibrary{decorations.markings}
\usetikzlibrary{calc}
\usepackage[font=small, margin=.5cm]{caption}

\usepackage{cleveref}
\crefformat{footnote}{#2\footnotemark[#1]#3}

\newcommand{\hgline}[2]{
\pgfmathsetmacro{\thetaone}{mod(#1,360)}
\pgfmathsetmacro{\thetatwo}{mod(#2,360)}
\pgfmathsetmacro{\theta}{(\thetaone+\thetatwo)/2}
\pgfmathsetmacro{\phi}{abs(\thetaone-\thetatwo)/2}
\pgfmathsetmacro{\close}{less(abs(\phi-90),0.0001)}
\ifdim \close pt = 1pt
    \draw[blue!50] (\thetaone:1) -- (\thetatwo:1);
\else
    \pgfmathsetmacro{\R}{tan(\phi)}
    \ifdim \R pt < 0pt
        \pgfmathsetmacro{\distance}{sqrt(1+\R*\R)}
        \draw[blue!50] (\theta:-\distance) circle (\R);
    \else \ifdim \R pt > 0pt
        \pgfmathsetmacro{\distance}{sqrt(1+\R^2)}
        \draw[blue!50] (\theta:\distance) circle (\R);
        \fi
    \fi
\fi
}

\newcommand{\hglinec}[3]{
\pgfmathsetmacro{\thetaone}{mod(#1,360)}
\pgfmathsetmacro{\thetatwo}{mod(#2,360)}
\pgfmathsetmacro{\theta}{(\thetaone+\thetatwo)/2}
\pgfmathsetmacro{\phi}{abs(\thetaone-\thetatwo)/2}
\pgfmathsetmacro{\close}{less(abs(\phi-90),0.0001)}
\ifdim \close pt = 1pt
    \draw[#3] (\thetaone:1) -- (\thetatwo:1);
\else
    \pgfmathsetmacro{\R}{tan(\phi)}
    \ifdim \R pt < 0pt
        \pgfmathsetmacro{\distance}{sqrt(1+\R*\R)}
        \draw[#3] (\theta:-\distance) circle (\R);
    \else \ifdim \R pt > 0pt
        \pgfmathsetmacro{\distance}{sqrt(1+\R^2)}
        \draw[#3] (\theta:\distance) circle (\R);
        \fi
    \fi
\fi
}

\newcommand*{\superellipse}[3][draw]{% #1 = styles
    \coordinate (ne) at ($(#2.center)!#3!(#2.north east)$);
    \coordinate (nw) at ($(#2.center)!#3!(#2.north west)$);
    \coordinate (sw) at ($(#2.center)!#3!(#2.south west)$);
    \coordinate (se) at ($(#2.center)!#3!(#2.south east)$);
    \pgfmathparse{-9.6*#3*#3 + 14.4*#3 - 4.8}
    \xdef\tension{\pgfmathresult}
    \path[#1]
       plot[smooth cycle, tension=\tension] coordinates{ 
         (#2.east) (ne) (#2.north) (nw) (#2.west) (sw) (#2.south) (se)};
}

\newtheorem{thm}{Theorem}[section]
\newtheorem{lem}[thm]{Lemma}

\newtheorem{prop}[thm]{Proposition}
\newtheorem{cor}[thm]{Corollary}

\theoremstyle{definition}

\newtheorem{defn}[thm]{Definition}

\newtheorem{example}[thm]{Example}
\newtheorem*{qst}{Question}
\newtheorem{rem}[thm]{Remark}
\newtheorem*{rem*}{Remark}

\numberwithin{equation}{section} 
\numberwithin{figure}{section}
\numberwithin{table}{section}

\makeatletter

\newcommand{\supp}{\mathop{\mathrm{supp}}}
\newcommand{\tr}{\mathop{\mathrm{Tr}}}
\newcommand{\ntr}{\mathop{\mathrm{tr}}}
\newcommand{\id}{\mathbf{1}}

\newcommand{\M}{\mathrm{M}}

\newcommand{\spc}{\mathrm{sp}}
\newcommand{\dH}{\mathrm{d_H}}

\begin{document}

\title{The strong convergence phenomenon}

\author{Ramon van Handel} 
\address{Department of Mathematics, Princeton University, Princeton, NJ 
08544, USA} 
\email{rvan@math.princeton.edu}

\begin{abstract}
In a seminal 2005 paper, Haagerup and Thorbj{\o}rnsen discovered that the 
norm of any noncommutative polynomial of independent complex Gaussian 
random matrices converges to that of a limiting family of operators that 
arises from Voiculescu's free probability theory. In recent years, new 
methods have made it possible to establish such strong convergence 
properties in much more general situations, and to obtain even more 
powerful quantitative forms of the strong convergence phenomenon. These, 
in turn, have led to a number of spectacular applications to long-standing 
open problems on random graphs, hyperbolic surfaces, and operator 
algebras, and have provided flexible new tools that enable the study of 
random matrices in unexpected generality. This survey aims to provide an 
introduction to this circle of ideas.
\end{abstract}

\subjclass[2010]{60B20; % RMT: probability
                 15B52; % RMT: algebraic
                 46L53; % noncommutative probability
                 46L54} % free probability

\maketitle

\thispagestyle{empty}
%\iffalse
%{\small
\setcounter{tocdepth}{2}
\tableofcontents
%}
%\fi

\section{Introduction}
\label{sec:intro}

The aim of this survey is to discuss recent developments surrounding
the following phenomenon, which has played a central role in a series of 
breakthroughs in the study of random graphs, hyperbolic surfaces, and 
operator algebras.

\begin{defn}
\label{defn:strong}
Let $\boldsymbol{X}^N=(X_1^N,\ldots,X_r^N)$ be a sequence of $r$-tuples of
random matrices of increasing dimension, and let 
$\boldsymbol{x}=(x_1,\ldots,x_r)$ be an $r$-tuple of bounded operators on 
a Hilbert space.
Then $\boldsymbol{X}^N$ is said to \emph{converge strongly} to 
$\boldsymbol{x}$ if
$$
	\lim_{N\to\infty}
	\| P(\boldsymbol{X}^N,\boldsymbol{X}^{N*}) \|
	= \| P(\boldsymbol{x},\boldsymbol{x}^*)\|
	\quad\text{in probability}
$$
for every $D\in\mathbb{N}$ and 
$P\in\M_D(\mathbb{C})\otimes\mathbb{C}\langle x_1,\ldots,x_{2r}\rangle$.
\end{defn}

Here we recall that a \emph{noncommutative polynomial with matrix 
coefficients}
$P\in\M_D(\mathbb{C})\otimes\mathbb{C}\langle x_1,\ldots,x_r\rangle$
of degree $q$ is a formal expression
$$
	P(x_1,\ldots,x_r) = A_0 \otimes\id +
	\sum_{k=1}^q\sum_{i_1,\ldots,i_k=1}^r
	A_{i_1,\ldots,i_k} \otimes x_{i_1}\cdots x_{i_k},
$$
where $A_{i_1,\ldots,i_k} \in \M_D(\mathbb{C})$ are $D\times D$ complex 
matrices. Such a polynomial defines a bounded operator whenever bounded 
operators are substituted for the free variables $x_1,\ldots,x_r$.
When $D=1$, this reduces to the classical notion of a noncommutative 
polynomial (we will then write $P\in \mathbb{C}\langle 
x_1,\ldots,x_r\rangle$).

The significance of the strong convergence phenomenon may not be 
immediately obvious when it is encountered for the first time.
Let us therefore begin with a very brief discussion of its origins.

The modern study of the spectral theory of random matrices arose from the 
work of physicists, especially that of Wigner and Dyson in the 1950s and 
60s \cite{Wig67}. Random matrices arise here as generic models for real 
physical systems that are too complicated to be understood in detail, such 
as the energy level structure of complex nuclei. It is believed that 
universal features of such systems are captured by random matrix models 
that are chosen essentially uniformly within the appropriate symmetry 
class. Such questions have led to numerous developments in probability and 
mathematical physics, and the spectra of such models are now understood in 
stunning detail down to microscopic scales (see, e.g., \cite{EY17}).

In contrast to these physically motivated models, random matrices that 
arise in other areas of mathematics often possess a much less regular 
structure. One way to build complex models is to consider arbitrary 
noncommutative polynomials of independent random matrices drawn from 
simple random matrix ensembles. It was shown in the seminal work of 
Voiculescu \cite{Voi91} that the limiting empirical eigenvalue 
distribution of such matrices can be described in terms of a family of 
limiting operators obtained by a free product construction. This is a 
fundamentally new perspective: while traditional random matrix methods are 
largely based on asymptotic explicit expressions or self-consistent 
equations satisfied by certain spectral statistics, Voiculescu's theory 
provides us with genuine \emph{limiting objects} whose spectral statistics 
are, in many cases, amenable to explicit computations. The interplay 
between random matrices and their limiting objects has proved to be of 
central importance, and will play a recurring role in the sequel.

While Voiculescu's theory is extremely useful, it yields rather weak 
information in that it can only describe the asympotics of the 
\emph{trace} of polynomials of random matrices. It was a major 
breakthrough when Haagerup and Thorbj{\o}rnsen~\cite{HT05} showed, for 
complex Gaussian (GUE) random matrices, that also the \emph{norm} of 
arbitrary polynomials converges to that of the corresponding limiting 
object. This much more powerful property, which was the first instance of 
strong convergence, opened the door to many subsequent developments.

The works of Voiculescu and Haagerup--Thorbj{\o}rnsen were directly 
motivated by applications to the theory of operator algebras. The fact 
that polynomials of a family of operators can be approximated by matrices 
places strong constraints on the operator (von~Neumann- or $C^*$-)algebras 
generated by this family: roughly speaking, it ensures that such algebras 
are ``approximately finite-dimensional'' in a certain sense. These 
properties have led to the resolution of important open problems in the 
theory of operator algebras which do not \emph{a priori} have anything to 
do with random matrices; see, e.g., \cite{VSW16,MS17,HT05}.

The interplay between operator algebras and random matrices continues to 
be a rich source of problems in both areas; an influential recent example 
is the work of Hayes \cite{Hay22} on the Peterson--Thom conjecture (cf.\ 
section \ref{sec:hayes}). In recent years, however, the notion of strong 
convergence has led to spectacular new applications in several other areas 
of mathematics. Broadly speaking, the importance of the strong convergence 
phenomenon is twofold.
\smallskip
\begin{enumerate}[$\bullet$]
\itemsep\medskipamount
\item Noncommutative polynomials are highly expressive: many
complex structures can be encoded in terms of spectral properties of
noncommutative polynomials.
\item Norm convergence is an extremely strong property, which provides 
access to challenging features of complex models.
\end{enumerate}
\smallskip
Furthermore, new mathematical methods have made it possible to establish 
novel quantitative forms of strong convergence, which enable the treatment 
of even more general random matrix models that were previously
out of reach.

We will presently highlight a number of themes that illustrate 
recent applications and developments surrounding strong convergence.
The remainder of this survey is devoted to a more detailed introduction
to this circle of ideas.

It should be emphasized at the outset that while I have aimed to give a 
general introduction to the strong convergence phenomenon and related 
topics, this survey is selectively focused on recent developments that are 
closest to my own interests, and is by no means comprehensive or complete. 
The interested reader may find complementary perspectives in surveys of 
Magee \cite{Mag24} and Collins \cite{Col23}, and is warmly encouraged to 
further explore the research literature on this subject.

\subsection{Optimal spectral gaps}
\label{sec:introgaps}

Let $G^N$ be a $d$-regular graph with $N$ vertices. By the 
Perron-Frobenius theorem, its adjacency matrix $A^N$ has largest 
eigenvalue $\lambda_1(A)=d$ with eigenvector $1$ (the vector
all of whose entries are one). The remaining eigenvalues are bounded by
$$
	\|A^N|_{1^\perp}\| = \max_{i=2,\ldots,N} |\lambda_i(A^N)| \le d.
$$
The smaller this quantity, the faster does a random walk on 
$G^N$ mix. The following classical lemma yields a lower bound that holds
for \emph{any} sequence of $d$-regular graphs. It provides a speed limit 
on how fast random walks can mix.

\begin{lem}[Alon--Boppana]
\label{lem:alonboppana}
For any $d$-regular graphs $G^N$ with $N$ vertices,
$$
	\|A^N|_{1^\perp}\| \ge 2\sqrt{d-1}- o(1)
	\quad\text{as}\quad N\to\infty.
$$
\end{lem}

\smallskip

Given a universal lower bound on the nontrivial eigenvalues, the obvious 
question is whether there exist graphs that attain this bound. Such 
graphs have the largest possible spectral gap. One may expect that 
such heavenly graphs must be very special, and indeed the first examples 
of such graphs were carefully constructed using deep number-theoretic 
ideas by Lubotzky--Phillips--Sarnak~\cite{LPS88} and 
Margulis~\cite{Mar88}.
It may therefore seem surprising that this property does not turn 
out to be all that special at all: \emph{random} graphs have an optimal 
spectral gap \cite{Fri08}.\footnote{%
	Here we gloss over an important distinction between the explicit 
	and random constructions: the former yields the so-called 
	Ramanujan property $\|A^N|_{1^\perp}\|\le 2\sqrt{d-1}$, 
	while the latter yields only $\|A^N|_{1^\perp}\|\le 2\sqrt{d-1}+o(1)$ 
	which is the natural converse to Lemma \ref{lem:alonboppana}
	(cf.\ section \ref{sec:ramanujan}).}

\begin{thm}[Friedman]
\label{thm:friedman}
For a random $d$-regular graph $G^N$ on $N$ vertices,
$$
	\|A^N|_{1^\perp}\| \le 2\sqrt{d-1} + o(1)
	\quad\text{with probability}\quad 1-o(1)
        \quad\text{as}\quad N\to\infty.
$$
\end{thm}

\smallskip

We now explain that Theorem \ref{thm:friedman} may be viewed as a very 
special instance of strong convergence. This viewpoint will open the door 
to establishing optimal spectral gaps in much more general situations.

Let us begin by recalling that the proof of Lemma \ref{lem:alonboppana} is 
based on the simple observation that for any graph $G$, the number of 
closed walks with a given length and starting vertex is lower bounded by 
the number of such walks in its universal cover $\tilde G$. When $G$ is 
$d$-regular, its universal cover $\tilde G$ is the infinite $d$-regular 
tree. From this, it is not difficult to deduce that the maximum nontrivial 
eigenvalue of a $d$-regular graph is asymptotically lower bounded by the 
spectral radius of the infinite $d$-regular tree, which is 
$2\sqrt{d-1}$ \cite[\S 5.2.2]{HLW06}.

Theorem \ref{thm:friedman} therefore states, in essence, that the support 
of the 
nontrivial spectrum of a random $d$-regular graph behaves as that 
of the infinite $d$-regular tree. To make the connection more explicit, it 
is instructive to construct both the random graph and infinite tree in 
a parallel manner. For simplicity, we assume $d$ is even
(the construction can be modified to the odd case as well).

\begin{figure}
\centering
\begin{tikzpicture}

\path[use as bounding box] (-2.55, 2) rectangle (8, -2);

\begin{scope}

\draw[thick,red] (-2+.25,.25) to[out=45,in=135] (-.25,.25);
\draw[thick,red] (.25,.25) to[out=45,in=135] (2-.25,.25);
\draw[thick,red] (-2-.25,.25) to[out=90+45,in=45] (2+.25,.25);

\draw (0,1.5) node {\color{red} $(123)$};

\draw[thick,blue] (-2+.25,-.25) to[out=-45,in=0] (-2,-.55)
to[out=180,in=-135] (-2-.25,-.25);

\draw[thick,blue] (.25,-.25) to[out=-45,in=-135] (2-.25,-.25);
\draw[thick,blue] (-.25,-.25) to[out=-90-45,in=-45] (2+.25,-.25);

\draw (0,-1.1) node {\color{blue} $(1)(23)$};

\begin{scope}
\draw[fill=black] (0,0) circle (0.05) node[left] {$\scriptstyle 2$};

\draw[thick] (0,0) -- (.25,.25);
\draw[thick] (0,0) -- (.25,-.25);
\draw[thick] (0,0) -- (-.25,.25);
\draw[thick] (0,0) -- (-.25,-.25);

\draw[color=white,fill=white] (.25,.25) circle (0.01);
\draw[color=white,fill=white] (.25,-.25) circle (0.01);
\draw[color=white,fill=white] (-.25,.25) circle (0.01);
\draw[color=white,fill=white] (-.25,-.25) circle (0.01);
\end{scope}

\begin{scope}[xshift=2cm]
\draw[fill=black] (0,0) circle (0.05) node[left] {$\scriptstyle 3$};

\draw[thick] (0,0) -- (.25,.25);
\draw[thick] (0,0) -- (.25,-.25);
\draw[thick] (0,0) -- (-.25,.25);
\draw[thick] (0,0) -- (-.25,-.25);

\draw[color=white,fill=white] (.25,.25) circle (0.01);
\draw[color=white,fill=white] (.25,-.25) circle (0.01);
\draw[color=white,fill=white] (-.25,.25) circle (0.01);
\draw[color=white,fill=white] (-.25,-.25) circle (0.01);
\end{scope}

\begin{scope}[xshift=-2cm]
\draw[fill=black] (0,0) circle (0.05) node[left] {$\scriptstyle 1$};

\draw[thick] (0,0) -- (.25,.25);
\draw[thick] (0,0) -- (.25,-.25);
\draw[thick] (0,0) -- (-.25,.25);
\draw[thick] (0,0) -- (-.25,-.25);

\draw[color=white,fill=white] (.25,.25) circle (0.01);
\draw[color=white,fill=white] (.25,-.25) circle (0.01);
\draw[color=white,fill=white] (-.25,.25) circle (0.01);
\draw[color=white,fill=white] (-.25,-.25) circle (0.01);
\end{scope}

\end{scope}

%%%%%%%%%%%%%% TREE

\begin{scope}[xshift=6cm]

\draw[thick,red] (0,0) -- (1,1);
\draw[thick,red] (0,0) -- (-1,1);
\draw[thick,blue] (0,0) -- (1,-1);
\draw[thick,blue] (0,0) -- (-1,-1);

\draw[thick,red] (1,1) -- (1.5,0.5);
\draw[thick,blue] (1,1) -- (1.5,1.5);
\draw[thick,blue] (1,1) -- (0.5,1.5);

\draw[thick,red] (-1,1) -- (-1.5,0.5);
\draw[thick,blue] (-1,1) -- (-1.5,1.5);
\draw[thick,blue] (-1,1) -- (-0.5,1.5);

\draw[thick,blue] (-1,-1) -- (-1.5,-0.5);
\draw[thick,red] (-1,-1) -- (-1.5,-1.5);
\draw[thick,red] (-1,-1) -- (-0.5,-1.5);

\draw[thick,blue] (1,-1) -- (1.5,-0.5);
\draw[thick,red] (1,-1) -- (1.5,-1.5);
\draw[thick,red] (1,-1) -- (0.5,-1.5);

\begin{scope}[xshift=1.5cm,yshift=1.5cm]
\draw[thick,red,dotted] (0,0) -- (.25,.25);
\draw[thick,red,dotted] (0,0) -- (-.25,.25);
\draw[thick,blue,dotted] (0,0) -- (.25,-.25);
\end{scope}

\begin{scope}[xshift=.5cm,yshift=1.5cm]
\draw[thick,red,dotted] (0,0) -- (.25,.25);
\draw[thick,red,dotted] (0,0) -- (-.25,.25);
\draw[thick,blue,dotted] (0,0) -- (-.25,-.25);
\end{scope}

\begin{scope}[xshift=-.5cm,yshift=1.5cm]
\draw[thick,red,dotted] (0,0) -- (.25,.25);
\draw[thick,red,dotted] (0,0) -- (-.25,.25);
\draw[thick,blue,dotted] (0,0) -- (.25,-.25);
\end{scope}

\begin{scope}[xshift=-1.5cm,yshift=1.5cm]
\draw[thick,red,dotted] (0,0) -- (.25,.25);
\draw[thick,red,dotted] (0,0) -- (-.25,.25);
\draw[thick,blue,dotted] (0,0) -- (-.25,-.25);
\end{scope}

\begin{scope}[xshift=1.5cm,yshift=-1.5cm]
\draw[thick,red,dotted] (0,0) -- (.25,.25);
\draw[thick,blue,dotted] (0,0) -- (.25,-.25);
\draw[thick,blue,dotted] (0,0) -- (-.25,-.25);
\end{scope}

\begin{scope}[xshift=1.5cm,yshift=-.5cm]
\draw[thick,red,dotted] (0,0) -- (.25,.25);
\draw[thick,red,dotted] (0,0) -- (-.25,.25);
\draw[thick,blue,dotted] (0,0) -- (.25,-.25);
\end{scope}

\begin{scope}[xshift=-1.5cm,yshift=-.5cm]
\draw[thick,red,dotted] (0,0) -- (.25,.25);
\draw[thick,red,dotted] (0,0) -- (-.25,.25);
\draw[thick,blue,dotted] (0,0) -- (-.25,-.25);
\end{scope}

\begin{scope}[xshift=-1.5cm,yshift=.5cm]
\draw[thick,red,dotted] (0,0) -- (-.25,.25);
\draw[thick,blue,dotted] (0,0) -- (.25,-.25);
\draw[thick,blue,dotted] (0,0) -- (-.25,-.25);
\end{scope}

\begin{scope}[xshift=1.5cm,yshift=.5cm]
\draw[thick,red,dotted] (0,0) -- (.25,.25);
\draw[thick,blue,dotted] (0,0) -- (.25,-.25);
\draw[thick,blue,dotted] (0,0) -- (-.25,-.25);
\end{scope}

\begin{scope}[xshift=0.5cm,yshift=-1.5cm]
\draw[thick,red,dotted] (0,0) -- (-.25,.25);
\draw[thick,blue,dotted] (0,0) -- (.25,-.25);
\draw[thick,blue,dotted] (0,0) -- (-.25,-.25);
\end{scope}

\begin{scope}[xshift=-0.5cm,yshift=-1.5cm]
\draw[thick,red,dotted] (0,0) -- (.25,.25);
\draw[thick,blue,dotted] (0,0) -- (.25,-.25);
\draw[thick,blue,dotted] (0,0) -- (-.25,-.25);
\end{scope}

\begin{scope}[xshift=-1.5cm,yshift=-1.5cm]
\draw[thick,red,dotted] (0,0) -- (-.25,.25);
\draw[thick,blue,dotted] (0,0) -- (.25,-.25);
\draw[thick,blue,dotted] (0,0) -- (-.25,-.25);
\end{scope}

\draw[fill=black] (0,0) circle (0.05) node[right] {$\scriptstyle e$};
\draw[fill=black] (1,1) circle (0.05) node[right] {$\scriptstyle a$};
\draw[fill=black] (-1,1) circle (0.05);
\draw[fill=black] (-1,-1) circle (0.05);
\draw[fill=black] (-1,1.06) node[right] {$\scriptstyle a^{\text{-}1}$};
\draw[fill=black] (-1,-0.96) node[right] {$\scriptstyle b^{\text{-}1}$};
\draw[fill=black] (1,-1) circle (0.05) node[right] {$\scriptstyle b$};

\draw[fill=black] (1.5,0.5) circle (0.05);
\draw[fill=black] (1.5,0.55) node[right] {$\scriptstyle a^2$};
\draw[fill=black] (1.5,1.5) circle (0.05);
\draw[fill=black] (1.5,1.52) node[right] {$\scriptstyle ba$};
\draw[fill=black] (0.5,1.5) circle (0.05);
\draw[fill=black] (0.5,1.55) node[right] {$\scriptstyle b^{\text{-}1}a$};

\draw[fill=black] (-1.5,0.5) circle (0.05);
\draw[fill=black] (-1.5,0.56) node[right] {$\scriptstyle a^{\text{-}2}$};
\draw[fill=black] (-1.5,1.5) circle (0.05);
\draw[fill=black] (-1.5,1.55) node[right] {$\scriptstyle 
b^{\text{-}1}a^{\text{-}1}$};

\draw[fill=black] (-0.5,1.5) circle (0.05);
\draw[fill=black] (-0.5,1.55) node[right] {$\scriptstyle ba^{\text{-}1}$};

\draw[fill=black] (-1.5,-0.5) circle (0.05);
\draw[fill=black] (-1.5,-0.46) node[right] {$\scriptstyle b^{\text{-}2}$};

\draw[fill=black] (-1.5,-1.5) circle (0.05);
\draw[fill=black] (-1.5,-1.46) node[right] {$\scriptstyle 
a^{\text{-}1}b^{\text{-}1}$};

\draw[fill=black] (-0.5,-1.5) circle (0.05);
\draw[fill=black] (-0.5,-1.46) node[right] {$\scriptstyle ab^{\text{-}1}$};

\draw[fill=black] (1.5,-0.5) circle (0.05);
\draw[fill=black] (1.5,-0.46) node[right] {$\scriptstyle b^2$};

\draw[fill=black] (1.5,-1.5) circle (0.05);
\draw[fill=black] (1.5,-1.5) node[right] {$\scriptstyle ab$};

\draw[fill=black] (0.5,-1.5) circle (0.05);
\draw[fill=black] (0.5,-1.46) node[right] {$\scriptstyle a^{\text{-}1}b$};

\end{scope}
\end{tikzpicture}
\caption{Left figure: $4$-regular graph generated by two permutations.
Right figure: $4$-regular tree generated by two free generators $a,b$ of
the free group $\mathbf{F}_2$. The edges defined by the two generators
are colored red and blue, respectively.\label{fig:friedman}}
\end{figure}
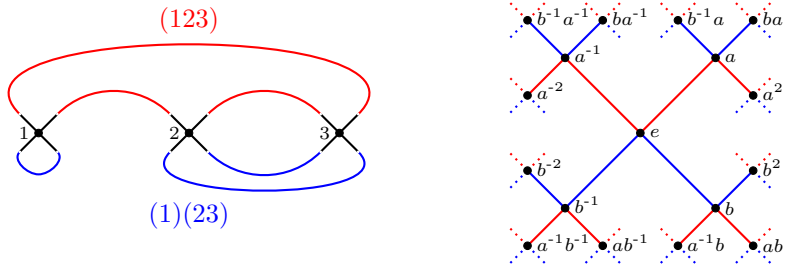

Given a permutation $\sigma\in\mathbf{S}_N$, we can define edges between
$N$ vertices by connecting each vertex $k\in[N]$ to its neighbors 
$\sigma(k)$ 
and $\sigma^{-1}(k)$. This defines a $2$-regular graph. To define a 
$d$-regular graph, we repeat this process with $r=\frac{d}{2}$ 
permutations. If the permutations are chosen independently and uniformly 
at random from $\mathbf{S}_N$, we obtain a random $d$-regular 
graph with adjacency matrix 
$$
	A^N = U_1^N + U_1^{N*} + \cdots + U_{r}^N +
	U_{r}^{N*},
$$
where $U_i^N$ are i.i.d.\ random permutation matrices of dimension 
$N$.\footnote{%
This is the \emph{permutation model} of random graphs;
see \cite[p.\,3]{Fri08} for its relation to other models.}

To construct the infinite $d$-regular tree in a parallel manner, we 
identify the vertices of the tree with the free group $\mathbf{F}_r$ with 
$r=\frac{d}{2}$ free generators $g_1,\ldots,g_r$. Each vertex 
$w\in\mathbf{F}_r$ is then connected to its neighbors $g_iw$ and 
$g_i^{-1}w$ for $i=1,\ldots,r$. This defines a $d$-regular tree with 
adjacency matrix
$$
	a = u_1 + u_1^* + \cdots + u_r + u_r^*,
$$
where $u_i=\lambda(g_i)$ is defined by the
left-regular representation 
$\lambda:\mathbf{F}_r \to B(l^2(\mathbf{F}_r))$,
i.e., $\lambda(g)\delta_w = \delta_{gw}$ where $\delta_w\in
l^2(\mathbf{F}_r)$ is
the coordinate vector of $w\in\mathbf{F}_r$.

These models are illustrated in Figure \ref{fig:friedman} for 
$r=2$, where the edges are colored according to their generator;
e.g., $U_1+U_1^*$ and $u_1+u_1^*$ are the 
adjacency matrices of the red edges in the left and right figures, 
respectively.

In these terms, Theorem \ref{thm:friedman} states that
$$
	\lim_{N\to\infty}
	\|(U_1^N + U_1^{N*} + \cdots + U_{r}^N +
        U_{r}^{N*})|_{1^\perp}\|
	=
	\|u_1 + u_1^* + \cdots + u_r + u_r^*\|
$$
in probability. This is precisely the kind of convergence described by
Definition \ref{defn:strong}, but only for one very special 
polynomial
$$
	P(\boldsymbol{x},\boldsymbol{x}^*) =
	x_1+x_1^*+\cdots + x_r+x_r^*.
$$
Making a leap of faith, we can now ask whether the same conclusion might 
even hold for \emph{every} noncommutative polynomial $P$. That this is 
indeed the case is a remarkable result of Bordenave and Collins \cite{BC19}.

\begin{thm}[Bordenave--Collins]
\label{thm:bc}
Let $\boldsymbol{U}^N=(U_1^N,\ldots,U_r^N)$ and
$\boldsymbol{u}=(u_1,\ldots,u_r)$ be defined as above.
Then $\boldsymbol{U}^N|_{1^\perp}$ converges strongly 
to $\boldsymbol{u}$, that is,
$$
	\lim_{N\to\infty}
	\|P(\boldsymbol{U}^N,\boldsymbol{U}^{N*})|_{1^\perp}\|
	=
	\|P(\boldsymbol{u},\boldsymbol{u}^*)\|
	\quad\text{in probability}
$$
for every $D\in\mathbb{N}$ and 
$P\in\M_D(\mathbb{C})\otimes\mathbb{C}\langle 
x_1,\ldots,x_{2r}\rangle$.\footnote{%
The reason that we must restrict to $1^\perp$ is elementary: the matrices 
$U_i^N$ have a Perron-Frobenius eigenvector $1$, but the operators $u_i$ 
do not as $1\not\in l^2(\mathbf{F}_r)$ (an infinite vector of ones is not 
in $l^2$). Thus we must remove the Perron-Frobenius 
eigenvector to achieve strong convergence.} 
\end{thm}

Theorem \ref{thm:bc} is a powerful tool because many problems can be 
encoded as special cases of this theorem by a suitable choice of $P$. To 
illustrate this, let us revisit the optimal spectral gap phenomenon in a 
broader context.

In general terms, the optimal spectral gap phenomenon is the following. 
The spectrum of various kinds of geometric objects admits a universal bound in 
terms of that of their universal covering space. The question is then 
whether there exist such objects which meet this bound. In particular, we 
may ask whether that is the case for random constructions. Lemma 
\ref{lem:alonboppana} and Theorem \ref{thm:friedman} establish this for 
regular graphs. An analogous picture in other situations is much more recent:
\smallskip
\begin{enumerate}[$\bullet$]
\itemsep\medskipamount
\item
It was shown by Greenberg \cite[Theorem 6.6]{HLW06} that for \emph{any} 
sequence of finite graphs $G^N$ with diverging number of vertices that 
have the same universal cover $\tilde G$, the maximal nontrivial 
eigenvalue of $G^N$ is asymptotically lower bounded by the spectral radius 
of $\tilde G$. On the other hand, given any (not necessarily regular) 
finite graph $G$, there is a natural construction of \emph{random lifts 
$G^N$} with the same universal cover $\tilde G$ \cite[\S 6.1]{HLW06}. It 
was shown by Bordenave and Collins \cite{BC19} that an optimal spectral 
gap phenomenon holds for random lifts of \emph{any} graph $G$.
\item Any hyperbolic surface $X$ has the hyperbolic plane $\mathbb{H}$ 
as its universal cover. Huber~\cite{Hub74} and
Cheng \cite{Che75} showed that for \emph{any} sequence $X^N$ of closed 
hyperbolic surfaces with diverging diameter, the first 
nontrivial eigenvalue of the Laplacian $\Delta_{X^N}$ is upper bounded by 
the bottom of the spectrum of $\Delta_{\mathbb{H}}$.
Whether this bound can be attained was an old question of 
Buser \cite{Bus84}. An affirmative answer was obtained by Hide and Magee 
\cite{HM23} by showing that an optimal spectral gap phenomenon holds for 
\emph{random} covering spaces of hyperbolic surfaces.
\end{enumerate}
\smallskip
The key ingredient in these breakthroughs is that the relevant 
spectral properties can be encoded as special instances of Theorem 
\ref{thm:bc}. How this is accomplished will be sketched in sections 
\ref{sec:lifts} and \ref{sec:buser}. 
In section \ref{sec:minsurf}, we will sketch another remarkable 
application due to Song \cite{Son24} to minimal surfaces.

Another series of developments surrrounding optimal spectral gaps
arises from a different perspective on Theorem \ref{thm:bc}.
The map $\mathrm{std}_N:\mathbf{S}_N\to 
\M_{N-1}(\mathbb{C})$ that assigns to each permutation 
$\sigma\in\mathbf{S}_N$ the restriction of the associated $N\times N$ 
permutation matrix to $1^\perp$ defines an irreducible representation of 
the symmetric group $\mathbf{S}_N$ of dimension $N-1$ called the 
\emph{standard representation}. Thus
$$
	U_i^N|_{1^\perp} = \mathrm{std}_N(\sigma_i^N),
$$
where $\sigma_1^N,\ldots,\sigma_r^N$ are independent uniformly distributed 
elements of $\mathbf{S}_N$. One may ask whether strong convergence remains 
valid if $\mathrm{std}_N$ is replaced by other representations $\pi_N$ of 
$\mathbf{S}_N$. This and related optimal spectral gap phenomena 
in representation-theoretic settings are the subject of long-standing 
questions and conjectures; see, e.g., \cite{RS19} and the references therein.

Recent advances in the study of strong convergence have led to major 
progress in the understanding of such questions 
\cite{BC20,CGTV25,CGV25,MdlS24,Cas24}. One of the most striking results in 
this direction to date is the recent work of Cassidy \cite{Cas24}, who 
shows that strong convergence for the symmetric group holds 
uniformly for all nontrivial irreducible representations of 
$\mathbf{S}_N$ of dimension up to 
$\exp(N^{\frac{1}{20}-\delta})$.\footnote{%
For comparison, all irreducible representations of $\mathbf{S}_N$ have dimension $\exp(O(N\log N))$.}
This makes it possible, for example, to study natural models of random 
regular graphs that achieve optimal spectral gaps using far less randomness 
than is required by Theorem~\ref{thm:bc}. We will discuss these results in
more detail in section \ref{sec:cassidy}.

\subsection{Intrinsic freeness}
\label{sec:intriintro}

We now turn to an entirely different development surrounding strong 
convergence that has enabled a sharp understanding of a very large class 
of random matrices in unexpected generality.

To set the stage for this development, let us begin by recalling 
the original strong convergence result of Haagerup and Thorbj{\o}rnsen
\cite{HT05}.
Let $G_1^N,\ldots,G_r^N$ be independent GUE matrices, that
is, $N\times N$ self-adjoint complex Gaussian random matrices whose
off-diagonal elements have variance $\frac{1}{N}$ and whose distribution
is invariant under unitary conjugation. The associated limiting object 
is a \emph{free semicircular family} $s_1,\ldots,s_r$
(cf.\ section \ref{sec:semicircle}). Define the random matrix
$$
	X^N = A_0\otimes\id + \sum_{i=1}^r A_i\otimes G_i^N
$$
and the limiting operator
$$
	X_{\rm free} = 
	A_0\otimes\id + \sum_{i=1}^r A_i\otimes s_i,
$$
where $A_0,\ldots,A_r\in\M_D(\mathbb{C})$ are self-adjoint matrix coefficients.

\begin{thm}[Haagerup--Thorbj{\o}rnsen]
\label{thm:ht}
For $X^N$ and $X_{\rm free}$ defined as above,\footnote{%
Here $\spc(X)$ denotes the spectrum of $X$, and
we recall that the Hausdorff distance
between sets $A,B\subseteq\mathbb{R}$ is defined as
$\dH(A,B) = \inf\{ \varepsilon>0:
A\subseteq B+[-\varepsilon,\varepsilon]\text{ and }
B\subseteq A+[-\varepsilon,\varepsilon]\}$.}
$$
	\lim_{N\to\infty}
	\dH\big(\spc(X^N),\spc(X_{\rm free})\big) = 0
	\quad\text{a.s.}
$$
\end{thm}

\smallskip

It is a nontrivial fact, known as the \emph{linearization trick}, that 
Theorem \ref{thm:ht} implies that $\boldsymbol{G}^N=(G_1^N,\ldots,G_r^N)$ 
converges strongly to $\boldsymbol{s}=(s_1,\ldots,s_r)$; see section 
\ref{sec:lin}. This conclusion was used by Haagerup and Thorbj{\o}rnsen to 
prove an old conjecture that the $\mathrm{Ext}$ invariant of the reduced 
$C^*$-algebra of any countable free group with at least two generators is 
not a group. For our present purposes, however, the above formulation of 
Theorem \ref{thm:ht} will be the most natural.

The result of Haagerup--Thorbj{\o}rnsen may be viewed as a strong 
incarnation of Voiculescu's asymptotic freeness principle \cite{Voi91}. 
The limiting operators $s_1,\ldots,s_r$ arise from a free product 
construction and are thus algebraically free (in fact, they are freely 
independent in the sense of Voiculescu). This makes it possible to 
compute the spectral statistics of $X_{\rm free}$ by means of 
closed form equations, cf.\ section \ref{sec:semicircle}. The explanation 
for Theorem \ref{thm:ht} provided by the asymptotic freeness principle is 
that since the random matrices $G_i^N$ have independent uniformly random 
eigenbases (due to their unitary invariance), they become increasingly 
noncommutative as $N\to\infty$ which leads them to behave freely in the 
limit.

From the perspective of applications, however, the most interesting case 
of this model is the special case $N=1$, that is, the random matrix $X=X^1$
defined by 
$$
	X = A_0 + \sum_{i=1}^r A_i g_i
$$
where $g_i$ are independent standard Gaussians. Indeed, \emph{any} 
$D\times D$ self-adjoint random matrix with jointly Gaussian entries (with 
arbitrary mean and covariance) can be expressed in this form. This model 
therefore captures almost arbitrarily structured random matrices: if one 
could understand random matrices at this level of generality, one would 
capture in one fell swoop a huge class of models that arise in applications. 
However, since the $1\times 1$ matrices $g_i=G_i^1$ commute, the 
asymptotic freeness principle has no bearing on such matrices, and there 
is no reason to expect that $X_{\rm free}$ has any significance for the 
behavior of $X$.

It is therefore rather surprising that the spectral properties of 
arbitrarily structured Gaussian random matrices $X$ are nonetheless 
captured by those of $X_{\rm free}$ in great generality. This phenomenon, 
developed in joint works of the author with Bandeira, 
Boedihardjo, Cipolloni, and Schr\"oder \cite{BBV23,BCSV23}, is captured by 
the following counterpart of Theorem \ref{thm:ht} 
(stated here in simplified form).

\begin{thm}[Intrinsic freeness]
\label{thm:intrfree}
For $X$ and $X_{\rm free}$ be defined as above, we have
$$
	\mathbf{P}\big[
	\dH\big(\spc(X),\spc(X_{\rm free})\big) >
	C\tilde v(X)\big((\log D)^{\frac{3}{4}} + t\big)
	\big] \le e^{-t^2}
$$
for all $t\ge 0$. Here $C$ is a universal constant,
$$
	\tilde v(X) = \|\mathbf{E}[(X-\mathbf{E}X)^2]\|^{\frac{1}{4}}
	\|\mathrm{Cov}(X)\|^{\frac{1}{4}},
$$
and $\mathrm{Cov}(X)$ is the
$D^2\times D^2$ covariance matrix of the entries of $X$.
\end{thm}

\begin{rem}
While Theorem \ref{thm:intrfree} captures the edges of the spectrum of 
$X$, analogous results for other spectral parameters (such as the spectral 
distribution) may be found in~\cite{BBV23,BCSV23}. These results are 
further extended to a large class of non-Gaussian random matrices in joint 
work of the author with Brailovskaya \cite{BvH24}.
\end{rem}

Theorem \ref{thm:intrfree} states that the spectrum of $X$ behaves like 
that of $X_{\rm free}$ as soon at the parameter $\tilde v(X)$ is small. 
Unlike for the model $X^N$, where noncommutativity is overtly introduced 
by means of unitarily invariant matrices, the mechanism for $X$ to behave 
as $X_{\rm free}$ can only arise from the structure of the matrix 
coefficients $A_0,\ldots,A_r$. We call this phenomenon \emph{intrinsic 
freeness}. It should not be obvious at this point why the parameter 
$\tilde v(X)$ captures intrinsic freeness; the origin of this phenomenon 
(which was inspired in part by \cite{HT05,Tro18}) and the mechanism that 
gives rise to it will be discussed in section \ref{sec:mconc}.

In practice, Theorem \ref{thm:intrfree} proves to be a powerful result as 
$\tilde v(X)$ is indeed small in numerous applications, while the result 
applies without any structural assumptions on the random matrix $X$. This 
is especially useful in questions of applied mathematics, where messy 
random matrices are par for the course. Several such applications are 
illustrated, for example, in \cite[\S 3]{BCSV23}.

\begin{figure}
\includegraphics[width=\textwidth]{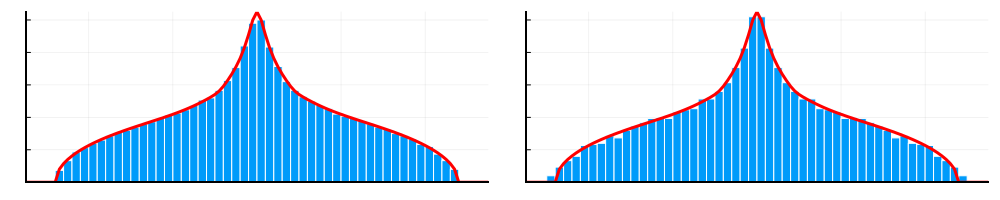}
\caption{Example of $P(H_1^N,H_2^N,H_3^N)$ where $H_i^N$ are independent 
$N\times N$ GUE (left plot) or $w$-sparse band matrices (right 
plot) with $N=1000$ and $w=27$. The histograms show the eigenvalues of a 
single realization of the random matrix, the solid line is the 
spectral density of $P(s_1,s_2,s_3)$ (computed using 
\texttt{NCDist.jl}).\label{fig:bandmat}} 
\end{figure}

Another consequence of Theorem \ref{thm:intrfree} is that the 
Haagerup--Thorbj{\o}rnsen strong convergence result extends to far more 
general situations. We only give one example for sake of illustration. A 
$w$-sparse Wigner matrix $H$ is a self-adjoint real random matrix so that 
each row has exactly $w$ nonzero entries, each of which is an independent 
(modulo symmetry $H_{ij}=H_{ji}$) centered Gaussian with variance 
$w^{-1}$. In this case $\tilde v(H)\sim w^{-\frac{1}{4}}$. Theorem 
\ref{thm:intrfree} shows that if $H_1^N,\ldots,H_r^N$ are independent 
$w$-sparse Wigner matrices of dimension $N$, then 
$\boldsymbol{H}^N=(H_1^N,\ldots,H_r^N)$ converges strongly to 
$\boldsymbol{s}=(s_1,\ldots,s_r)$ as soon as $w\gg (\log N)^3$ 
\emph{regardless} of the choice of sparsity pattern. Unlike GUE matrices, 
such models need not possess any invariance and can have 
localized eigenbases. Even though this appears to dramatically violate the 
classical intuition behind asymptotic freeness, this model exhibits 
precisely the same strong convergence property as GUE (see Figure 
\ref{fig:bandmat}).

\subsection{New methods in random matrix theory}
\label{sec:newmethods}

The development of strong convergence has gone hand in hand with the 
discovery of new methods in random matrix theory. For example, Haagerup 
and Thorbj{\o}rnsen \cite{HT05} pioneered the use of self-adjoint 
linearization (section \ref{sec:lin}), which enabled them to make 
effective use of Schwinger-Dyson equations to capture general polynomials
(their work was extended to various classical random matrix models in
\cite{Sch05,And13,CM14}); 
while Bordenave and Collins \cite{BC19,BC20,BC24} developed 
operator-valued nonbacktracking methods in order to efficiently apply the 
moment method to strong convergence.

More recently, however, two new methods for proving strong convergence 
have proved to be especially powerful, and have opened the door both to 
obtaining strong quantitative results and to achieving strong convergence 
in new situations that were previously out of reach. In contrast to 
previous approaches, these methods are rather different in spirit to those 
used in classical random matrix theory.

\subsubsection{The interpolation method}
\label{sec:introinterpol}

The general principle captured by strong convergence (and by the earlier 
work of Voiculescu) is that the spectral statistics of families of random 
matrices behave of those of a family of limiting operators. In classical 
approaches to random matrix theory, however, the limiting operators do not 
appear directly: rather, one shows that the spectral statistics of these 
operators admit explicit expressions or closed-form equations, and that 
the random matrices obey these same expressions or equations 
approximately.

In contrast, the existence of limiting operators suggests that these may 
be exploited explicitly as a method of proof in random matrix theory. This 
is the basic idea behind the \emph{interpolation method}, which was 
developed independently by Collins, Guionnet, and Parraud \cite{CGP22} to 
obtain a quantitative form of the Haagerup--Thorbj{\o}rnsen theorem, and 
by Bandeira, Boedihardjo, and the author~\cite{BBV23} to prove the intrinsic 
freeness principle (Theorem \ref{thm:intrfree}).

Roughly speaking, the method works as follows. We aim to 
show that the spectral statistics of a random matrix $X$ behave as those 
of a limiting operator $X_{\rm free}$. To this end, one introduces a certain 
continuous interpolation $(X(t))_{t\in[0,1]}$ between these objects, so 
that $X(1)=X$ and $X(0)=X_{\rm free}$. To bound the discrepancy between
the spectral statistics of $X$ and $X_{\rm free}$, one can then estimate
$$
	|\mathbf{E}[\ntr h(X)] - \tau(h(X_{\rm free}))|
	\le
	\int_0^1 \bigg| \frac{d}{dt} \mathbf{E}[\ntr h(X(t))]\bigg|\,dt,
$$
where $\tau$ denotes the limiting trace (see 
section \ref{sec:cstar}). If a good bound can be obtained
for sufficiently general $h$ (we will choose $h(x)=|z-x|^{-2p}$ for
$p\in\mathbb{N}$ and $z\in\mathbb{C}\backslash\mathbb{R}$), convergence
of the norm will follow as a consequence.

As stated above, this procedure does not make much sense. Indeed $X$ (a 
random matrix) and $X_{\rm free}$ (a deterministic operator) do not even 
live in the same space, so it is unclear what it means to interpolate 
between them. Moreover, the general approach outlined above does not in 
itself explain why the derivative along the interpolation should be small: 
the latter is the key part of the argument that requires one to understand 
the mechanism that gives rise to free behavior. Both these issues will be 
explained in more detail in section \ref{sec:mconc}, where we will sketch 
the main ideas behind the proof of Theorem \ref{thm:intrfree}.

\subsubsection{The polynomial method}
\label{sec:intropoly}

We now describe an entirely different method, developed in the recent work 
of Chen, Garza-Vargas, Tropp, and the author \cite{CGTV25}, which has led 
to a series of new developments.

Consider a sequence of self-adjoint random matrices $X^N$ with limiting operator 
$X_{\rm F}$; one may keep in mind the example 
$X^N = P(\boldsymbol{U}^N,\boldsymbol{U}^{N*})|_{1^\perp}$ and 
$X_{\rm F}=P(\boldsymbol{u},\boldsymbol{u}^*)$ in the context of Theorem 
\ref{thm:bc}. In many natural models, it turns out to be the case that
spectral statistics of \emph{polynomial} test functions $h$ can be expressed
as
$$
	\mathbf{E}[\ntr h(X^N)] = \Phi_h(\tfrac{1}{N}),
$$
where $\Phi_h$ is a rational function whose degree is controlled by the
degree $q$ of the polynomial $h$. Whenever this is the case, the fact that
$$
	\tau(h(X_{\rm F})) = \Phi_h(0)
$$
is generally an immediate consequence. However, such soft information 
does not in itself suffice to reason about the norm.

The key observation behind the \emph{polynomial method} is that 
classical results in the analytic theory of polynomials (due to Chebyshev, 
Markov, Bernstein, $\ldots$) can be exploited to ``upgrade'' the above 
soft information to strong quantitative bounds, merely by virtue of the 
fact that $\Phi_h$ is rational. The basic idea is to write
$$
	|\mathbf{E}[\ntr h(X^N)]-\tau(h(X_{\rm F}))|
	=
	|\Phi_h(\tfrac{1}{N})-\Phi_h(0)| \le
	\frac{1}{N} \|\Phi_h'\|_{L^\infty[0,\frac{1}{N}]}.
$$
This is reminiscent of the interpolation method, 
where now instead of an interpolation parameter 
we ``differentiate with 
respect to $\frac{1}{N}$''. In contrast to the interpolation method, 
however, the surprising feature of the present approach is that the
derivative of $\Phi_h$ can be controlled by means of completely general
tools that do not require any understanding of the random matrix model.
In particular, the analysis makes use of the following two
classical facts about polynomials \cite{Che98}.
\smallskip
\begin{enumerate}[$\bullet$]
\itemsep\medskipamount
\item An inequality of A.\ Markov states that
$\|f'\|_{L^\infty[-1,1]} \le q^2 \|f\|_{L^\infty[-1,1]}$
for every real polynomial $f$ of degree at most $q$.
\item A corollary of the Markov inequality states that
$\|f\|_{L^\infty[-1,1]} \le 2 \max_{x\in I}|f(x)|$
for any discretization $I$ of $[-1,1]$ with spacing at most $\frac{1}{q^2}$.
\end{enumerate}
\smallskip
Applying these results to the numerator and denominator of the rational 
function $\Phi_h$ yields with minimal effort a bound of the form 
$$
	\|\Phi_h'\|_{L^\infty[0,\frac{1}{N}]}\lesssim q^4\max_{N}
	|\Phi_h(\tfrac{1}{N})|=q^4 \max_{N} |\mathbf{E}[\ntr h(X^N)]|
$$
(the additional factor $q^2$ arises since we must restrict to the part of the
interval where the spacing between the points $\{\frac{1}{N}\}$ is
sufficiently small). Thus we obtain a strong quantitative bound in a 
completely soft manner.

In this form, the above method does not suffice to achieve strong convergence.
To this end, two additional ingredients must be added.
\smallskip
\begin{enumerate}[1.]
\itemsep\medskipamount
\item The above analysis requires the test function $h$
to be a polynomial. However, since the bound depends only
polynomially on the degree of $h$, one can use a Fourier-analytic argument
to extend the bound to arbitrary smooth $h$.
\item 
The $\frac{1}{N}$ rate obtained above does not suffice to
deduce convergence of the norm, since it can only ensure that $X^N$ has
a bounded (rather than vanishing) number of eigenvalues outside the 
support of $X_{\rm F}$. To achieve strong convergence, we must expand
$\Phi_h$ to second (or higher) order and control the additional term(s).
\end{enumerate}
\smallskip
Nonetheless, all these ingredients are essentially elementary and can be 
implemented with minimal problem-specific inputs. 

The polynomial method will be discussed in detail in section 
\ref{sec:poly}, where we will illustrate it by giving an essentially 
complete proof of Theorem \ref{thm:bc}. That an elementary proof is 
possible at all is surprising in itself, given that previous methods 
required delicate and lengthy computations.

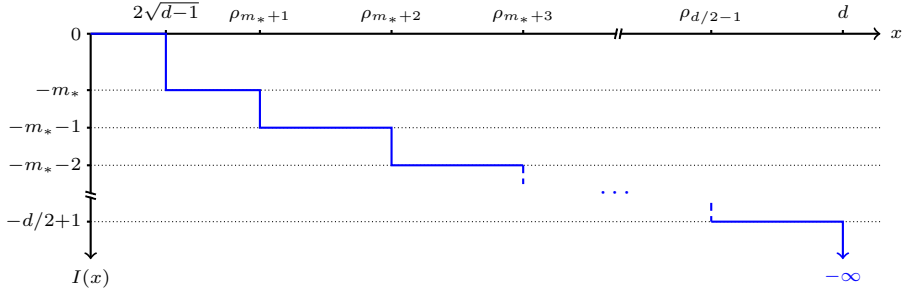
\begin{figure} 
\centering 
\begin{tikzpicture}

\draw[->,thick] (0,0) -- (10.5,0) node[right] {$\scriptstyle x$};
\draw[->,thick] (0,0) -- (0,-3) node[below] {$\scriptstyle I(x)$};

\draw[thick] (1,0) -- (1,.05);
\draw (1,.3) node {$\scriptstyle 2\sqrt{d-1}$};

\draw[thick] (2.25,0) -- (2.25,.05);
\draw (2.25,.23) node {$\scriptstyle \rho_{m_*+1}$};

\draw[thick] (4,0) -- (4,.05);
\draw (4,.23) node {$\scriptstyle \rho_{m_*+2}$};

\draw[thick] (5.75,0) -- (5.75,.05);
\draw (5.75,.23) node {$\scriptstyle \rho_{m_*+3}$};

\draw[thick] (8.25,0) -- (8.25,.05);
\draw (8.25,.23) node {$\scriptstyle \rho_{d/2-1}$};

\draw[thick] (10,0) -- (10,.05);
\draw (10,.3) node {$\scriptstyle d$};

\draw[thick] (0,0) -- (-.05,0);
\draw (0,0) node[left] {$\scriptstyle 0$};

\draw[thick] (0,-.75) -- (-.05,-.75);
\draw (0,-.75) node[left] {$\scriptstyle -m_*$};
\draw[color=black!100,densely dotted] (0,-.75) -- (10.5,-.75);

\draw[thick] (0,-1.25) -- (-.05,-1.25);
\draw (0,-1.25) node[left] {$\scriptstyle -m_*-1$};
\draw[color=black!100,densely dotted] (0,-1.25) -- (10.5,-1.25);

\draw[thick] (0,-1.75) -- (-.05,-1.75);
\draw (0,-1.75) node[left] {$\scriptstyle -m_*-2$};
\draw[color=black!100,densely dotted] (0,-1.75) -- (10.5,-1.75);

\draw[thick] (0,-2.5) -- (-.05,-2.5);
\draw (0,-2.5) node[left] {$\scriptstyle -d/2+1$};
\draw[color=black!100,densely dotted] (0,-2.5) -- (10.5,-2.5);

\draw[thick,blue] (0,0) -- (1,0) -- (1,-.75) -- (2.25,-.75) --
(2.25,-1.25) -- (4,-1.25) -- (4,-1.75) -- (5.75,-1.75);
\draw[thick,densely dashed,blue] (5.75,-1.75) -- (5.75,-2);

\draw[thick,densely dashed,blue] (8.25,-2.25) -- (8.25,-2.5);

\draw[->,thick,blue] (8.25,-2.5) -- (10,-2.5) -- (10,-3)
node[below] {$\scriptstyle -\infty$};

\draw[blue] (7,-2.125) node {$\cdots$};

\draw[thick] (-0.07,-2.1) to (0.07,-2.14);
\draw[thick] (-0.07,-2.16) to (0.07,-2.2);
\draw[thick,white] (-0.07,-2.13) to (0.07,-2.17);

\draw[thick] (6.97,-0.07) to (7,0.07);
\draw[thick] (7.03,-0.07) to (7.06,0.07);
\draw[thick,white] (7,-0.07) to (7.03,0.07);

\end{tikzpicture}
\caption{Staircase pattern of the large deviation probabilities
in Friedman's Theorem~\ref{thm:friedman} for the permutation model of random regular
graphs.
Here we define
$I(x)=\lim_{N\to\infty} \frac{\log\mathbf{P}[\|A^N\|\,\ge\, x]}{\log N}$,
$m_* = \lfloor \frac{1}{2}(\sqrt{d-1}+1)\rfloor$, and
$\rho_m = 2m-1 + \frac{d-1}{2m-1}$.
\label{fig:staircase}}
\end{figure}

When it is applicable, the polynomial method has typically provided the 
best known quantitative results and has made it possible to address 
previously inaccessible questions. To date, this includes works of Magee 
and de la Salle~\cite{MdlS24} and of Cassidy~\cite{Cas24} on strong 
convergence of high dimensional representations of the unitary and 
symmetric groups (see also \cite{CGV25}); strong convergence for 
polynomials with coefficients in subexponential operator spaces 
\cite{CGV25}; strong convergence of the tensor GUE model of graph products 
of semicircular variables (\emph{ibid}.); a characterization of the 
unusual large deviations in Friedman's theorem \cite{CGTV25} as 
illustrated in Figure~\ref{fig:staircase}; work of Magee, Puder and 
the author on strong convergence of uniformly random permutation 
representations of surface groups \cite{MPV25}; and work of 
Hide--Macera--Thomas on spectral gaps of Weil--Petersson random
surfaces \cite{HMT25b}.

\subsection{Organization of this survey}

The rest of this survey is organized as follows.

Section \ref{sec:strong} collects a number of basic but very useful 
properties of strong (and weak) convergence that are scattered throughout 
the literature. These properties also illustrate the fundamdental 
interplay between strong convergence and the operator algebraic properties 
of the limiting objects.

Section \ref{sec:poly} provides a detailed illustration of the polynomial 
method: we will give an essentially self-contained proof of Theorem 
\ref{thm:bc}.

Section \ref{sec:mconc} is devoted to further discussion of the intrinsic 
freeness phenomenon. In particular, we aim to explain the mechanism that 
gives rise to it.

Section \ref{sec:appl} discusses in more detail various applications of 
the strong convergence phenomenon that we introduced above. In particular, 
we aim to explain how the strong convergence property is used in these 
applications.

Finally, section \ref{sec:open} is devoted to a brief exposition of 
various open problems surrounding the strong convergence phenomenon.

\section{Strong convergence}
\label{sec:strong}

The aim of this section is to collect various general properties of strong 
convergence that are often useful. Many of these 
properties rely on operator algebraic properties of the limiting objects. 
We have aimed to make the presentation accessible to
readers without any prior background in operator algebras.

\subsection{$C^*$-probability spaces}
\label{sec:cstar}

Let $X$ be an $N\times N$ self-adjoint (usually random) matrix. We will 
be interested in understanding the
\emph{empirical spectral distribution}
$$
	\mu_X = \frac{1}{N}\sum_{k=1}^N \delta_{\lambda_k(X)}
$$
(that is, $\mu_X(I)$ is the fraction of the eigenvalues
of $X$ that lies in the set $I\subseteq\mathbb{R}$); and the 
\emph{spectral edges}, that is, the extreme eigenvalues 
$$
	\|X\| = \max_{1\le k\le N} |\lambda_k(X)|
$$
or, more generally, the full spectrum $\spc(X)$ as a set. In the models
that we will consider, both these spectral features are well described 
by the corresponding features of a limiting operator $X_{\rm F}$ as 
$N\to\infty$: convergence of the spectral distribution is \emph{weak 
convergence}, and convergence of the spectral edges is \emph{strong 
convergence}. These notions will be formally defined in the next 
section.\footnote{%
	These notions capture the macroscopic features of the 
	spectrum. A large part of modern random matrix theory is concerned 
	with understanding the spectrum at the \emph{microscopic} (or 
	local) scale, that is, understanding the limit of the eigenvalues 
	viewed as a point process. Such questions are rather different in 
	spirit, as the behavior of the local statistics is expected to be 
	universal and is not described by the spectral properties of 
	limiting operators.}

To do so, we must first give meaning to the spectral distribution and 
edges of the limiting operator $X_{\rm F}$. For the spectral edges, we may 
simply consider the norm or spectrum of $X_{\rm F}$ which are well defined 
for bounded operators on any Hilbert space $H$. However, the meaning of 
the spectral distribution of $X_{\rm F}$ is not clear a priori. Indeed, 
since the empirical spectral distribution
$$
	\int f \,d\mu_X =
	\frac{1}{N}\sum_{k=1}^N f(\lambda_k(X))
	= \ntr(f(X))
$$
is defined by the normalized trace,\footnote{%
	We denote by $\ntr X=\frac{1}{N}\tr X$ the normalized trace
	of an $N\times N$ matrix $X$, and define $f(X)$ 
	by functional calculus (i.e., apply $f$
	to the eigenvalues of $X$ while keeping the eigenvectors fixed).}
defining the spectral distribution of $X_{\rm F}$ requires us to make 
sense of the normalized trace of infinite-dimensional operators. This is 
impossible in general, as any linear functional $\tau:B(H)\to\mathbb{C}$ 
with the trace property $\tau(xy)=\tau(yx)$ for all $x,y\in B(H)$ must be 
trivial $\tau\equiv 0$ (this follows immediately by noting that when $H$ 
is infinite-dimensional, every element of $B(H)$ can be written as the sum 
of two commutators \cite{Hal54}).

This situation is somewhat reminiscent of a basic issue of
probability theory: one cannot define a probability measure on 
\emph{arbitrary} subsets of an uncountable set, but must rather work with 
a suitable $\sigma$-algebra of sets for which the notion of measure makes 
sense. In the present setting, we cannot define a normalized 
trace for \emph{all} bounded operators on an infinite-dimensional Hilbert 
space $H$, but must rather work with a suitable algebra 
$\mathcal{A}\subset B(H)$ of operators on which the trace 
$\tau:\mathcal{A}\to\mathbb{C}$ can be defined.
These objects must satisfy some basic axioms.\footnote{%
	We are a bit informal in our terminology:
	$C^*$-algebras are usually defined as
	Banach algebras rather than as subalgebras of $B(H)$.
	However, as any $C^*$-algebra may be represented in the
	latter form, our definition entails no loss of generality.
	What we call a trace should more precisely be called
	a tracial state. A crash course on the basic
	notions may be found in \cite{NS06}.}

\begin{defn}[$C^*$-algebra]
A \emph{unital $C^*$-algebra} is an algebra $\mathcal{A}$ of bounded 
operators on a complex Hilbert space that is self-adjoint 
($a\in\mathcal{A}$ implies $a^*\in\mathcal{A}$), contains the
identity $\id\in\mathcal{A}$, and is closed in the operator norm
topology.
\end{defn}

\begin{defn}[Trace]
\label{defn:tr}
A \emph{trace} on a unital $C^*$-algebra $\mathcal{A}$ is a linear 
functional $\tau:\mathcal{A}\to\mathbb{C}$ that is positive $\tau(a^*a)\ge 
0$, unital $\tau(\id)=1$, and tracial $\tau(ab)=\tau(ba)$.
A trace is called \emph{faithful} if $\tau(a^*a)=0$ implies $a=0$.
\end{defn}

\begin{defn}[$C^*$-probability space]
A \emph{$C^*$-probability space} is a pair $(\mathcal{A},\tau)$, where
$\mathcal{A}$ is a unital $C^*$-algebra and 
$\tau:\mathcal{A}\to\mathbb{C}$ is a faithful trace.
\end{defn} 

The simplest example of a $C^*$-probability space is the algebra of 
$N\times N$ matrices with its normalized trace $(\M_N(\mathbb{C}),\ntr)$. 
One may conceptually think of general $C^*$-probability spaces as 
generalizations of this example.

\begin{rem}
Most of the axioms in the above definitions are obvious analogues
of the properties of $(\M_N(\mathbb{C}),\ntr)$. What may not
be obvious at first sight is why we require $\mathcal{A}$ to be closed
in the norm topology. The reason is that it ensures that
$f(a)\in\mathcal{A}$ for any self-adjoint $a\in\mathcal{A}$
not only when $f$ is a polynomial (which follows merely from the fact 
that
$\mathcal{A}$ is an algebra), but also when $f$ is a continuous function, 
since the latter can be approximated in norm by polynomials. This 
property will presently be needed to define the spectral distribution.
\end{rem}

\begin{rem}
\label{rem:vna}
If we make the stronger assumption that $\mathcal{A}$ is closed in the 
strong operator topology, $\mathcal{A}$ is called a \emph{von~Neumann 
algebra}. This ensures that $f(a)\in\mathcal{A}$ even when $f$ is a 
bounded measurable function. Von~Neumann algebras form a major research
area in their own right, but
appear in this survey only in section~\ref{sec:hayes}.
\end{rem}

Given a $C^*$-probability space $(\mathcal{A},\tau)$, we can now 
associate to each self-adjoint element $a\in\mathcal{A}$, $a=a^*$ a
spectral distribution $\mu_a$ by defining
$$
	\int f\,d\mu_a = \tau(f(a))
$$
for every continuous function $f:\mathbb{R}\to\mathbb{C}$. Indeed,
that $\tau$ is positive and unital implies that $f\mapsto\tau(f(a))$ is
a positive and normalized linear functional on $C_0(\mathbb{R})$, so
$\mu_a$ exists by the Riesz representation theorem.

This survey is primarily concerned with strong convergence, that is, with 
norms and not with spectral distributions. Nonetheless, it is generally 
the case that the only spectral statistics of random matrices that are 
directly computable are trace statistics (such as the moments 
$\mathbf{E}[\ntr X^p]$), so that a good understanding of weak convergence 
is prerequisite for proving strong convergence. In particular, we must
understand how to recover the spectrum from the spectral distribution.
It is here that the faithfulness of the trace $\tau$ plays a key role.

\begin{lem}[Spectral distribution and spectrum]
\label{lem:faith}
Let $(\mathcal{A},\tau)$ be a $C^*$-probability space. Then for any 
self-adjoint $a\in\mathcal{A}$, $a=a^*$, we have
$\supp\mu_a = \spc(a)$ and thus
$$
	\|a\| = \lim_{p\to\infty} \tau(a^{2p})^{\frac{1}{2p}}.
$$
\end{lem}

\begin{proof}
By the definition of support, $x\not\in\supp\mu_a$ if and only if
there is a continuous nonnegative function $f$ so that
$f(x)>0$ and $\int f\,d\mu_a=0$. On the other hand, by the spectral
theorem, $x\not\in\spc(a)$ if and only if there is a continuous 
nonnegative function $f$ so that $f(x)>0$ and $f(a)=0$.
That $\supp\mu_a=\spc(a)$ therefore follows as
$\tau(f(a))=0$ if and only if $f(a)=0$, since $\tau$ is faithful and
$f\ge 0$.
\end{proof}

We now introduce one of the most important examples of a $C^*$-probability 
space. Another important example will appear in section 
\ref{sec:semicircle}.

\begin{example}[Reduced group $C^*$-algebras]
Let $\mathbf{G}$ be a finitely generated group with generators
$g_1,\ldots,g_r$, and $\lambda:\mathbf{G}\to 
B(l^2(\mathbf{G}))$ be its left-regular representation, i.e.,
$\lambda(g)\delta_w = \delta_{gw}$ where $\delta_w\in 
l^2(\mathbf{G})$ is the coordinate vector of $w\in\mathbf{G}$.
Then
$$
	C^*_{\rm red}(\mathbf{G}) = 
	\mathrm{cl}_{\|\cdot\|}\big(\mathop{\mathrm{span}}
	\{\lambda(g):g\in\mathbf{G}\}\big) =
	C^*(\lambda(g_1),\ldots,\lambda(g_r))
$$
is called the \emph{reduced $C^*$-algebra of $\mathbf{G}$}.
Here
and in the sequel, $C^*(\boldsymbol{u})$ denotes the $C^*$-algebra
generated by a family of operators $\boldsymbol{u}=(u_1,\ldots,u_r)$ (that 
is, the norm-closure of the set of all noncommutative polynomials in 
$\boldsymbol{u},\boldsymbol{u}^*$).

The reduced $C^*$-algebra of any group always comes equipped with a 
\emph{canonical trace} $\tau:C^*_{\rm red}(\mathbf{G})\to\mathbb{C}$ that 
is defined for all $a\in C^*_{\rm red}(\mathbf{G})$ by
$$
	\tau(a) = \langle \delta_e, a\,\delta_e\rangle,
$$
where $e\in\mathbf{G}$ is the identity element. Note that:
\smallskip
\begin{enumerate}[$\bullet$]
\itemsep\medskipamount
\item 
It is straightforward to check that $\tau$ is indeed tracial: by
linearity, it suffices to show that $\tau(\lambda(g)\lambda(h))=
1_{gh=e}=\tau(\lambda(h)\lambda(g))$ for all $g,h\in\mathbf{G}$.
\item
$\tau$ is also faithful: if 
$\tau(a^*a)=0$,
then $\|a\delta_g\|^2=\tau(\lambda(g)^*a^*a\lambda(g))=\tau(a^*a)=0$ for 
all $g\in\mathbf{G}$ by the trace property (since 
$\lambda(g)\lambda(g)^*=\id$), and thus $a=0$.
\end{enumerate}
\smallskip
Thus $(C^*_{\rm red}(\mathbf{G}),\tau)$ defines a
$C^*$-probability space.
\end{example}

\begin{example}[Free group]
In the case that $\mathbf{G}=\mathbf{F}_r$ is a free group, we 
implicitly encountered the above
construction in section \ref{sec:introgaps}. We argued there
that the adjacency matrix of a random $2r$-regular graph is 
modelled by the operator
$$
	a = \lambda(g_1)+\lambda(g_1)^* + \cdots +
	\lambda(g_r)+\lambda(g_r)^* \in C^*_{\rm red}(\mathbf{F}_r).
$$
It follows immediately from the
definition that the moments of the spectral distribution $\mu_a$ 
(defined by the canonical trace $\tau$) are given by
$$
	\int x^p \,d\mu_a =
	\tau(a^p) =
	\#\{ \text{words of length $p$ in }g_1,g_1^{-1},\ldots,g_r,g_r^{-1}
	\text{ that reduce to }e\}.
$$
As the moments grow at most exponentially in $p$,
this uniquely determines $\mu_a$. The density of $\mu_a$ was
computed in a classic paper of Kesten \cite[proof of Theorem 3]{Kes59},
and is known as the \emph{Kesten distribution}. Since the explicit formula 
for the 
density shows that $\supp \mu_a=[-2\sqrt{2r-1},2\sqrt{2r-1}]$, 
Lemma \ref{lem:faith} yields
$$
	\|a\| = 2\sqrt{2r-1}.
$$
This explains the value of the norm that appears in Theorem 
\ref{thm:friedman}.
\end{example}

\subsection{Strong and weak convergence}

We can now formally define the notions of weak and strong convergence of
families of random matrices.

\begin{defn}[Weak and strong convergence]
\label{defn:wkstrcv}
Let $\boldsymbol{X}^N=(X_1^N,\ldots,X_r^N)$ be a sequence of
$r$-tuples of random matrices, and let 
$\boldsymbol{x}=(x_1,\ldots,x_r)$ be an $r$-tuple of elements
of a $C^*$-probability space $(\mathcal{A},\tau)$.
\smallskip
\begin{enumerate}[$\bullet$]
\itemsep\medskipamount
\item 
$\boldsymbol{X}^N$ is said to \emph{converge weakly} to 
$\boldsymbol{x}$ if
for every $P\in\mathbb{C}\langle x_1,\ldots,x_{2r}\rangle$
\begin{equation}
\label{eq:wkcv}
	\lim_{N\to\infty} \ntr(P(\boldsymbol{X}^N,\boldsymbol{X}^{N*}))
	= \tau(P(\boldsymbol{x},\boldsymbol{x}^{*}))
	\quad\text{in probability}.
\end{equation}
\item 
$\boldsymbol{X}^N$ is said to \emph{converge strongly} to 
$\boldsymbol{x}$ if
for every $P\in\mathbb{C}\langle x_1,\ldots,x_{2r}\rangle$
\begin{equation}
\label{eq:strcv}
	\lim_{N\to\infty} \|P(\boldsymbol{X}^N,\boldsymbol{X}^{N*})\|
	= \|P(\boldsymbol{x},\boldsymbol{x}^{*})\|
	\quad\text{in probability}.
\end{equation}
\end{enumerate}
\end{defn}
\smallskip
This definition appears to be slightly weaker than our initial
definition of strong convergence in Definition \ref{defn:strong}, where
we allowed for polynomials $P$ with matrix rather than scalar 
coefficients. We will show in section \ref{sec:scalarmtx} that the 
apparently weaker definition in fact already implies the stronger one.

We begin by spelling out some basic properties, see for example
\cite[\S 2.1]{CM14}.

\begin{lem}[Equivalent formulations of weak convergence]
\label{lem:equivweak}
The following are equivalent.
\smallskip
\begin{enumerate}[a.]
\itemsep\medskipamount
\item $\boldsymbol{X}^N$ converges weakly to $\boldsymbol{x}$.
\item Eq.\ \eqref{eq:wkcv} holds for every
\emph{self-adjoint} $P\in\mathbb{C}\langle x_1,\ldots,x_{2r}\rangle$.
\item For every self-adjoint $P\in\mathbb{C}\langle 
x_1,\ldots,x_{2r}\rangle$, the empirical spectral distribution
$\mu_{P(\boldsymbol{X}^N,\boldsymbol{X}^{N*})}$ converges weakly to 
$\mu_{P(\boldsymbol{x},\boldsymbol{x}^*)}$ in probability.
\end{enumerate}
\end{lem}

\begin{proof}
Since every polynomial $P\in\mathbb{C}\langle x_1,\ldots,x_{2r}\rangle$ 
can be written as $P=P_1+iP_2$ for self-adjoint polynomials $P_1,P_2$,
the equivalence $a\Leftrightarrow b$ is immediate by linearity of the 
trace. Moreover, the implication $c\Rightarrow b$ is trivial since
$\tau(a)=\int x\,\mu_a(dx)$ by the definition of the spectral distribution
(and as $\mu_a$ is compactly supported).

On the other hand, since $P^p\in \mathbb{C}\langle x_1,\ldots,x_{2r}\rangle$ 
for every $p\in\mathbb{N}$, \eqref{eq:wkcv} implies 
\begin{multline*}
	\int x^p\,\mu_{P(\boldsymbol{X}^N,\boldsymbol{X}^{N*})}(dx) = 
	\ntr(P(\boldsymbol{X}^N,\boldsymbol{X}^{N*})^p)
	\xrightarrow{N\to\infty} \\
	\tau(P(\boldsymbol{x},\boldsymbol{x}^*)^p)=\int 
	x^p\,\mu_{P(\boldsymbol{x},\boldsymbol{x}^*)}(dx)
\end{multline*}
in probability. As $\mu_{P(\boldsymbol{x},\boldsymbol{x}^*)}$ is compactly
supported, convergence of moments implies weak convergence, and
the implication $b\Rightarrow c$ follows.
\end{proof}

A parallel result holds for strong convergence.

\begin{lem}[Equivalent formulations of strong convergence]
\label{lem:equivstrong}
The following are equivalent.
\smallskip
\begin{enumerate}[a.]
\itemsep\medskipamount
\item $\boldsymbol{X}^N$ converges strongly to $\boldsymbol{x}$.
\item Eq.\ \eqref{eq:strcv} holds for every
\emph{self-adjoint} $P\in\mathbb{C}\langle x_1,\ldots,x_{2r}\rangle$.
\item For every self-adjoint $P\in\mathbb{C}\langle 
x_1,\ldots,x_{2r}\rangle$ and $f\in C(\mathbb{R})$, we have
$$
        \lim_{N\to\infty} \|f(P(\boldsymbol{X}^N,\boldsymbol{X}^{N*}))\|
        = \|f(P(\boldsymbol{x},\boldsymbol{x}^{*}))\|
	\quad\text{in probability}.
$$
\item For every self-adjoint $P\in\mathbb{C}\langle
x_1,\ldots,x_{2r}\rangle$, we have
$$
	\lim_{N\to\infty} 
	\dH\big(\spc(P(\boldsymbol{X}^N,\boldsymbol{X}^{N*})),
	\spc(P(\boldsymbol{x},\boldsymbol{x}^{*}))\big)
	= 0
	\quad\text{in probability}.		
$$
\end{enumerate}
\end{lem}

\begin{proof}
Since $\|X\|^2=\|X^*X\|$ for any operator $X$ and as $P^*P\in 
\mathbb{C}\langle x_1,\ldots,x_{2r}\rangle$ is self-adjoint, it is
immediate that $a\Leftrightarrow b$. That $d\Rightarrow 
b$ is immediate as $|\|X\|-\|Y\|| \le \dH(\spc(X),\spc(Y))$ for
any bounded self-adjoint operators $X,Y$.

\smallskip

To show that $b\Rightarrow c$, we may choose for any $\varepsilon>0$
a univariate real polynomial $h$ so that $\|f-h\|_{L^\infty[-K,K]}\le
\varepsilon$, where $K=2\|P(\boldsymbol{x},\boldsymbol{x}^{*})\|$.
Since \eqref{eq:strcv} implies that
$\|P(\boldsymbol{X}^N,\boldsymbol{X}^{N*})\|\le K$ with probability 
$1-o(1)$ as $N\to\infty$, we obtain
$$
	\|f(P(\boldsymbol{x},\boldsymbol{x}^{*}))\| - 4\varepsilon	
	\le
	\|f(P(\boldsymbol{X}^N,\boldsymbol{X}^{N*}))\|
	\le
	\|f(P(\boldsymbol{x},\boldsymbol{x}^{*}))\| + 4\varepsilon
$$
with probability $1-o(1)$ as $N\to\infty$ by applying \eqref{eq:strcv}
to $h\circ P\in \mathbb{C}\langle 
x_1,\ldots,x_{2r}\rangle$. That
$b\Rightarrow c$ follows by taking $\varepsilon\downarrow 0$.

\smallskip

To show that $c\Rightarrow d$, choose $f\in C(\mathbb{R})$ so that
$f(x)=0$ for $x\in\spc(P(\boldsymbol{x},\boldsymbol{x}^{*}))$ and
$f(x)>0$ otherwise. Since $c$ implies that
$\|f(P(\boldsymbol{X}^N,\boldsymbol{X}^{N*}))\|\to 0$ in probability,
$$
	\spc(P(\boldsymbol{X}^N,\boldsymbol{X}^{N*}))\subseteq
	\spc(P(\boldsymbol{x},\boldsymbol{x}^{*}))+[-\varepsilon,
	\varepsilon]
$$
with probability $1-o(1)$ as $N\to\infty$ for every $\varepsilon>0$.
On the other hand, for any $y\in 
\spc(P(\boldsymbol{x},\boldsymbol{x}^{*}))$, we may choose $f\in 
C(\mathbb{R})$ so that $f(y)=1$ and $f(x)<1$ for $x\ne y$.
Since $c$ implies that
$\|f(P(\boldsymbol{X}^N,\boldsymbol{X}^{N*}))\|\to 1$ in probability,
$$
	y+[-\tfrac{\varepsilon}{2},\tfrac{\varepsilon}{2}]
	\subseteq
	\spc(P(\boldsymbol{X}^N,\boldsymbol{X}^{N*}))+[-\varepsilon,
	\varepsilon]
$$
with probability $1-o(1)$ as $N\to\infty$ for every $\varepsilon>0$.
As $\spc(P(\boldsymbol{x},\boldsymbol{x}^{*}))$ can be covered
by a finite number of such sets  
$y+[-\tfrac{\varepsilon}{2},\tfrac{\varepsilon}{2}]$, the implication
$c\Rightarrow d$ follows.
\end{proof}

The elementary equivalent formulations of weak and strong convergence 
discussed above are all concerned with the (real) eigenvalues of 
self-adjoint polynomials. In contrast, what implications weak or strong 
convergence may have for the empirical distributions of the complex 
eigenvalues of non-self-adjoint (or non-normal) polynomials is poorly 
understood; see section \ref{sec:complex}. We nonetheless record one 
easy observation in this direction \cite[Remark 3.6]{Miy23}.

\begin{lem}[Spectral radius]
\label{lem:specradius}
Suppose that $\boldsymbol{X}^N$ converges strongly to $\boldsymbol{x}$.
Then
$$
	\varrho\big(P(\boldsymbol{X}^N,\boldsymbol{X}^{N*})\big) 
	\le
	\varrho\big(P(\boldsymbol{x},\boldsymbol{x}^*)\big)+o(1)
	\quad
	\text{with probability}\quad 1-o(1)
$$
for every $P\in\mathbb{C}\langle x_1,\ldots,x_{2r}\rangle$,
where $\varrho(a)=\sup\{|\lambda|:\lambda\in\spc(a)\}$ denotes the
spectral radius of any (not necessarily normal) operator $a$.
\end{lem}

\begin{proof}
This follows immediately from the fact that the spectral radius is
upper semicontinuous with respect to the operator norm
\cite[\S 104]{Hal82}.
\end{proof}

\subsection{Strong implies weak}

While we have formulated weak and strong convergence as distinct 
phenomena, it turns out that strong convergence---or even merely a 
one-sided form of it---often automatically implies weak convergence. Such 
a statement should be viewed with suspicion, since the definition of weak 
convergence requires us to specify a trace while the definition of strong 
convergence is independent of the trace. However, it turns out that many 
$C^*$-algebras have a \emph{unique} trace, and this is precisely the 
setting we will consider.

\begin{lem}[Strong implies weak]
\label{lem:upperlower}
Let $\boldsymbol{X}^N=(X_1^N,\ldots,X_r^N)$ be a sequence of
$r$-tuples of random matrices, and let 
$\boldsymbol{x}=(x_1,\ldots,x_r)$ be an $r$-tuple of 
elements
of a $C^*$-probability space $(\mathcal{A},\tau)$. Consider the following 
conditions.
\smallskip
\begin{enumerate}[a.]
\itemsep\medskipamount
\item 
For every $P\in\mathbb{C}\langle x_1,\ldots,x_{2r}\rangle$
$$
	\|P(\boldsymbol{X}^N,\boldsymbol{X}^{N*})\|
        \le \|P(\boldsymbol{x},\boldsymbol{x}^{*})\| + o(1)
	\quad\text{with probability}\quad 1-o(1).
$$
\item $\boldsymbol{X}^N$ converges weakly to $\boldsymbol{x}$.
\item
For every $P\in\mathbb{C}\langle x_1,\ldots,x_{2r}\rangle$
$$
	\|P(\boldsymbol{X}^N,\boldsymbol{X}^{N*})\|
        \ge \|P(\boldsymbol{x},\boldsymbol{x}^{*})\| - o(1)
	\quad\text{with probability}\quad 1-o(1).
$$
\end{enumerate}
Then $b\Rightarrow c$, and $a\Rightarrow b$ if in addition 
$C^*(\boldsymbol{x})$ has a unique trace.
\end{lem}

\begin{proof}
To prove $b\Rightarrow c$, note that weak convergence implies
$$
	\|P(\boldsymbol{X}^N,\boldsymbol{X}^{N*})\|
	\ge
	\ntr\big(|P(\boldsymbol{X}^N,\boldsymbol{X}^{N*})|^{2p}
	\big)^{\frac{1}{2p}}
	=
	\tau\big(|P(\boldsymbol{x},\boldsymbol{x}^{*})|^{2p}\big)^{\frac{1}{2p}}
	-o(1)
$$
with probability $1-o(1)$ for every $p\in\mathbb{N}$, as
$|P|^{2p}= (P^*P)^p\in \mathbb{C}\langle x_1,\ldots,x_{2r}\rangle$.
The conclusion follows by letting $p\to\infty$ and applying
Lemma \ref{lem:faith}.

To prove $a\Rightarrow b$, let us first consider the special case that 
$\boldsymbol{X}^N$ are nonrandom.
Define a linear functional
$\ell_N : \mathbb{C}\langle x_1,\ldots,x_{2r}\rangle \to\mathbb{C}$ by
$$
	\ell_N(P) = \ntr P(\boldsymbol{X}^N,\boldsymbol{X}^{N*}).
$$
This is called the \emph{law} of the family $\boldsymbol{X}^N$; it has
the same properties as the trace in Definition \ref{defn:tr}, but 
restricted only to polynomials. Note that by linearity, $\ell_N$ is
fully determined by its value on all monomials.

Since $|\ell_N(Q)|\le \max_{N,i}\|X_i^N\|^{\deg(Q)}$ for every 
monomial $Q$, the sequence $\ell_N$ is precompact in the 
weak$^*$-topology.
Thus for every subsequence of the indices $N$, there is a further
subsequence so that $\ell_N\to\ell$ pointwise for some law $\ell$ that
satisfies the properties of a trace. On the other hand, condition $a$ 
ensures that
$$
	|\ell(P)| = \lim |\ell_N(P)| \le
	\limsup \|P(\boldsymbol{X}^N,\boldsymbol{X}^{N*})\|
	\le \|P(\boldsymbol{x},\boldsymbol{x}^{*})\|
$$
where the limits are taken along the subsequence. Thus $\ell$ extends
by continuity to a trace on $C^*(\boldsymbol{x})$. Since the latter
has the unique trace property, we must have
$\ell(P)=\tau(P(\boldsymbol{x},\boldsymbol{x}^{*}))$, and thus we have
proved weak convergence.

When $\boldsymbol{X}^N$ are random, we note that
condition $a$ implies (by Borel-Cantelli and as $\mathbb{C}\langle 
x_1,\ldots,x_{2r}\rangle$ is separable) that for every subequence of 
indices $N$, we can find a further subsequence along which
$\|P(\boldsymbol{X}^N,\boldsymbol{X}^{N*})\|
\le \|P(\boldsymbol{x},\boldsymbol{x}^{*})\| + o(1)$ for every $P
\in \mathbb{C}\langle x_1,\ldots,x_{2r}\rangle$ a.s.
The proof now proceeds as in the nonrandom case.
\end{proof}

The unique trace property turns out to arise frequently in practice. In 
particular, that $C^*_{\rm red}(\mathbf{F}_r)$ has a unique trace for 
$r\ge 2$ is a classical result of Powers~\cite{Pow75}, and a general 
characterization of countable groups $\mathbf{G}$ so that $C^*_{\rm 
red}(\mathbf{G})$ has a unique trace is given by 
Breuillard--Kalantar--Kennedy--Ozawa \cite{BKKO17}. In such 
situations, Lemma~\ref{lem:upperlower} shows that a strong convergence 
\emph{upper} bound (condition $a$) already suffices to establish both 
strong and weak convergence in full. Establishing such an upper bound
is the main difficulty in proofs of strong convergence.

\begin{rem}
The implication $a\Rightarrow c$ of Lemma~\ref{lem:upperlower} also holds 
under the alternative hypothesis that $C^*(\boldsymbol{x})$ is a simple 
$C^*$-algebra; see \cite[pp.\ 16–19]{LM25}.
\end{rem}

\subsection{Scalar, matrix, and operator coefficients}
\label{sec:scalarmtx}

In Definition \ref{defn:wkstrcv}, we have defined the weak and strong 
convergence properties for polynomials $P$ with scalar coefficients. 
However, applications often require polynomials with matrix or even 
operator coefficients to encode the models of interest.\footnote{%
Let $(\mathcal{A},\tau)$ and $(\mathcal{B},\sigma)$ be $C^*$-probability
spaces. If $\boldsymbol{x}=(x_1,\ldots,x_r)$ are  
elements of $\mathcal{A}$
and $P\in\mathcal{B}\otimes\mathbb{C}\langle x_1,\ldots,x_{2r}\rangle$
is a polynomial with coefficients in $\mathcal{B}$, then
$P(\boldsymbol{x},\boldsymbol{x}^*)$ lies in the algebraic tensor product 
$\mathcal{A}\otimes_{\rm alg}\mathcal{B}$. This viewpoint suffices
for weak convergence. To make sense of strong convergence, however,
we must define a norm on the tensor product.
We will do so in the obvious way:
Given $\mathcal{A}\subseteq B(H_1)$ and $\mathcal{B}\subseteq B(H_2)$,
we define the $C^*$-algebra $\mathcal{A}\otimes\mathcal{B}\subseteq
B(H_1\otimes H_2)$ by
$$
	\mathcal{A}\otimes\mathcal{B} = 
	\mathrm{cl}_{\|\cdot\|}\big(\mathop{\mathrm{span}}
	\{a\otimes b: a\in\mathcal{A},b\in\mathcal{B}\}\big),
$$
and extend the trace $\tau\otimes\sigma:\mathcal{A}\otimes\mathcal{B}
\to\mathbb{C}$ accordingly.
This construction is called the \emph{minimal} tensor product
of $C^*$-probability spaces, and is often denoted $\otimes_{\rm min}$.
For simplicity, we fix the following convention:
\emph{in this survey, the notation $\otimes$ will always denote
the minimal tensor product.}}
We now show that such properties are already implied by their 
counterparts for scalar polynomials.

For weak convergence, this situation is easy.

\begin{lem}[Operator-valued weak convergence]
The following are equivalent.
%\smallskip
\begin{enumerate}[a.]
\itemsep\medskipamount
\item
$\boldsymbol{X}^N$ converges weakly to $\boldsymbol{x}$, i.e.,
for all $P\in\mathbb{C}\langle x_1,\ldots,x_{2r}\rangle$
$$
	\lim_{N\to\infty} \ntr(P(\boldsymbol{X}^N,\boldsymbol{X}^{N*}))
	= \tau(P(\boldsymbol{x},\boldsymbol{x}^{*}))
	\quad\text{in probability}.
$$
\item
For any $C^*$-probability space $(\mathcal{B},\sigma)$ and
$P\in\mathcal{B}\otimes\mathbb{C}\langle x_1,\ldots,x_{2r}\rangle$
$$
	\lim_{N\to\infty} 
	(\sigma\otimes\ntr)(P(\boldsymbol{X}^N,\boldsymbol{X}^{N*}))
	= (\sigma\otimes\tau)(P(\boldsymbol{x},\boldsymbol{x}^{*}))
	\quad\text{in probability}.
$$
\end{enumerate}
\end{lem}

\begin{proof}
That $b\Rightarrow a$ is obvious. To prove $a\Rightarrow b$,
let us express $P\in\mathcal{B}\otimes\mathbb{C}\langle 
x_1,\ldots,x_{2r}\rangle$ concretely as
$P(x_1,\ldots,x_{2r}) = b_0 \otimes\id +
\sum_{k=1}^q\sum_{i_1,\ldots,i_k=1}^{2r}
b_{i_1,\ldots,i_k} \otimes x_{i_1}\cdots x_{i_k}$
with operator coefficients $b_{i_1,\ldots,i_k}\in\mathcal{B}$. Then
clearly
$$
	(\sigma\otimes\ntr)(P(\boldsymbol{X}^N,\boldsymbol{X}^{N*})) =
	\sigma(b_0) +
	\sum_{k=1}^q\sum_{i_1,\ldots,i_k=1}^{2r}
	\sigma(b_{i_1,\ldots,i_k})\,\ntr(X^N_{i_1}\cdots X^N_{i_k})
$$
where we denote $X_{r+i}^N = X_i^{N*}$ for $i=1,\ldots,r$.
Since $a$ yields $\ntr(X^N_{i_1}\cdots X^N_{i_k})\to
\tau(x_{i_1}\cdots x_{i_k})$ for all $k,i_1,\ldots,i_k$, the
conclusion follows.
\end{proof}

Unfortunately, the analogous equivalence for strong convergence is simply 
false at this level of generality; a counterexample can be constructed 
as in \cite[Appendix~A]{CGV25}. Nonetheless, strong 
convergence extends in complete generality to polynomials with 
\emph{matrix} (as opposed to 
operator) coefficients. This justifies the apparently more general 
Definition \ref{defn:strong} given in the introduction.

\begin{lem}[Matrix-valued strong convergence]
\label{lem:mtxstr}
The following are equivalent.
%\smallskip
\begin{enumerate}[a.]
\itemsep\medskipamount
\item
$\boldsymbol{X}^N$ converges strongly to $\boldsymbol{x}$, i.e.,
for all $P\in\mathbb{C}\langle x_1,\ldots,x_{2r}\rangle$
$$
	\lim_{N\to\infty} \|P(\boldsymbol{X}^N,\boldsymbol{X}^{N*})\|
	= \|P(\boldsymbol{x},\boldsymbol{x}^{*})\|
	\quad\text{in probability}.
$$
\item
For every $D\in\mathbb{N}$ and
$P\in\M_D(\mathbb{C})\otimes\mathbb{C}\langle x_1,\ldots,x_{2r}\rangle$
$$
	\lim_{N\to\infty} 
	\|P(\boldsymbol{X}^N,\boldsymbol{X}^{N*})\|
	= \|P(\boldsymbol{x},\boldsymbol{x}^{*})\|
	\quad\text{in probability}.
$$
\end{enumerate}
\end{lem}

\begin{proof}
That $b\Rightarrow a$ is obvious. To prove $a\Rightarrow b$, 
express
$P\in\M_D(\mathbb{C})\otimes\mathbb{C}\langle x_1,\ldots,x_{2r}\rangle$
as $P = \sum_{i,j=1}^D e_i e_j^* \otimes P_{ij}$ with
$P_{ij}\in \mathbb{C}\langle x_1,\ldots,x_{2r}\rangle$, where 
$e_1,\ldots,e_D$ denotes the standard basis of $\mathbb{C}^D$.
We can therefore estimate
$$
	\max_{i,j} \|P_{ij}(\boldsymbol{x},\boldsymbol{x}^{*})\|
	\le
	\|P(\boldsymbol{x},\boldsymbol{x}^{*})\| \le
	D^2 \max_{i,j}\|P_{ij}(\boldsymbol{x},\boldsymbol{x}^{*})\|,
$$
and analogously for $P(\boldsymbol{X}^N,\boldsymbol{X}^{N*})$.
Here we used $\|P_{ij}\| = 
\|(e_ie_i^*\otimes\id)P(e_je_j^*\otimes\id)\|$
for the first inequality
and the triangle inequality for the second.
Thus $a$ yields
$$
	D^{-2}\|P(\boldsymbol{x},\boldsymbol{x}^{*})\| -o(1)
	\le
	\|P(\boldsymbol{X}^N,\boldsymbol{X}^{N*})\|
	\le D^2 \|P(\boldsymbol{x},\boldsymbol{x}^{*})\| + o(1)
$$
for probability $1-o(1)$ as $N\to\infty$ for every
$P\in\M_D(\mathbb{C})\otimes\mathbb{C}\langle x_1,\ldots,x_{2r}\rangle$.
Now note that since $\|P\|^{2p} = \|(P^*P)^p\|$ and 
$(P^*P)^p\in \M_D(\mathbb{C})\otimes\mathbb{C}\langle 
x_1,\ldots,x_{2r}\rangle$ for every $p\in\mathbb{N}$, applying the above 
inequality to $(P^*P)^p$ implies \emph{a fortiori} that
$$
	D^{-1/p}\|P(\boldsymbol{x},\boldsymbol{x}^{*})\| -o(1)
	\le
	\|P(\boldsymbol{X}^N,\boldsymbol{X}^{N*})\|
	\le D^{1/p} \|P(\boldsymbol{x},\boldsymbol{x}^{*})\| + o(1)
$$
for probability $1-o(1)$ as $N\to\infty$. Taking $p\to\infty$
completes the proof.
\end{proof}

Strong convergence of polynomials with operator coefficients requires 
additional assumptions. For example, if the coefficients are 
\emph{compact} operators, strong convergence follows easily from Lemma 
\ref{lem:mtxstr} since compact operators can be approximated in norm by 
finite rank operators (i.e., by matrices).

A much weaker requirement is provided by the following 
property of $C^*$-algebras. We give the definition in the form that is 
most relevant for our purposes; its equivalence to the original more 
algebraic definition (in terms of short exact sequences) is a nontrivial 
fact due to Kirchberg, see \cite[Chapter 17]{Pis03} or \cite{BO08}.

\begin{defn}[Exact $C^*$-algebra]
\label{defn:exact}
A $C^*$-algebra $\mathcal{B}$ is called \emph{exact} if for every
finite-dimensional subspace $\mathcal{S}\subseteq\mathcal{B}$ and
$\varepsilon>0$, there exists $D\in\mathbb{N}$ and a linear embedding
$u:\mathcal{S}\to \M_D(\mathbb{C})$ such that 
$$
	\|(u\otimes\mathrm{id})(x)\| \le \|x\| \le
	(1+\varepsilon)\|(u\otimes\mathrm{id})(x)\|
$$
for every $C^*$-algebra $\mathcal{A}$ and
$x\in\mathcal{S}\otimes\mathcal{A}$.
\end{defn}

We can now prove the following.

\begin{lem}[Operator-valued strong convergence]
\label{lem:exact}
Suppose that $\boldsymbol{X}^N$ converges strongly to $\boldsymbol{x}$.
Then we have
$$
	\lim_{N\to\infty} \|P(\boldsymbol{X}^N,\boldsymbol{X}^{N*})\|
	= \|P(\boldsymbol{x},\boldsymbol{x}^{*})\|
	\quad\text{in probability}
$$
for every 
$P\in\mathcal{B}\otimes\mathbb{C}\langle x_1,\ldots,x_{2r}\rangle$
with coefficients in an exact $C^*$-algebra $\mathcal{B}$.
\end{lem}

\begin{proof}
Fix $P\in\mathcal{B}\otimes\mathbb{C}\langle x_1,\ldots,x_{2r}\rangle$,
let $\mathcal{S}\subseteq\mathcal{B}$ be the linear span of the 
operator coefficients of $P$, and let $\varepsilon>0$. 
Let $u:\mathcal{S}\to\M_D(\mathbb{C})$ be the embedding
provided
by Definition \ref{defn:exact}. Since $Q=(u\otimes\mathrm{id})(P)\in
\M_D(\mathbb{C})\otimes \mathbb{C}\langle x_1,\ldots,x_{2r}\rangle$,
we obtain
$$
	\|Q(\boldsymbol{x},\boldsymbol{x}^{*})\| - o(1)
	\le
	\|P(\boldsymbol{X}^N,\boldsymbol{X}^{N*})\| \le
	(1+\varepsilon)\|Q(\boldsymbol{x},\boldsymbol{x}^{*})\| + o(1)
$$
with probability $1-o(1)$ as $N\to\infty$ by Lemma \ref{lem:mtxstr},
while
$$
	\|Q(\boldsymbol{x},\boldsymbol{x}^{*})\|\le 
	\|P(\boldsymbol{x},\boldsymbol{x}^{*})\| \le 
	(1+\varepsilon)\|Q(\boldsymbol{x},\boldsymbol{x}^{*})\|.
$$
The conclusion follows by letting $\varepsilon\downarrow 0$.
\end{proof}

The exactness property turns out to arise frequently in practice. In 
particular, $C^*_{\rm red}(\mathbf{F}_r)$ is exact \cite[Corollary 
17.10]{Pis03}, as is $C^*_{\rm red}(\mathbf{G})$ for many other groups 
$\mathbf{G}$. For an extensive discussion, see 
\cite[Chapter 5]{BO08} or \cite{Ana07}.

One reason that exactness is very useful in a strong convergence context 
is that it enables us construct complex strong convergence models by 
combining simpler building blocks, as will be explained briefly in section 
\ref{sec:hayes}. Another useful application of exactness is that it 
enables an improved form of Lemma \ref{lem:upperlower} with uniform bounds 
over polynomials with matrix coefficients of any dimension
\cite[\S 5.3]{MdlS24}.

\subsection{Linearization}
\label{sec:lin}

In the previous section, we showed that strong convergence of polynomials 
with scalar coefficients implies strong convergence of polynomials with 
matrix coefficients. If we allow for matrix coefficients, however, we can 
achieve a different kind of simplification: to establish strong 
convergence, it suffices to consider only polynomials with matrix 
coefficients \emph{of degree one}. This nontrivial fact is often referred 
to as the \emph{linearization trick}.

We first develop a version of the linearization trick for unitary 
families.

\begin{thm}[Unitary linearization]
\label{thm:ulin}
Let $\boldsymbol{U}^N=(U_1^N,\ldots,U_r^N)$ be a sequence of $r$-tuples
of unitary random matrices, and let $\boldsymbol{u}=(u_1,\ldots,u_r)$ be
an $r$-tuple of unitaries in a $C^*$-algebra $\mathcal{A}$. Then
the following are equivalent.
%\smallskip
\begin{enumerate}[a.]
\itemsep\medskipamount  
\item For every $D\in\mathbb{N}$ and self-adjoint
$P\in\M_D(\mathbb{C})\otimes\mathbb{C}\langle x_1,\dots,x_{2r}\rangle$
of degree one,
$$
	\lim_{n\to\infty} 
	\|P(\boldsymbol{U}^N,\boldsymbol{U}^{N*})\| =
	\|P(\boldsymbol{u},\boldsymbol{u}^*)\|
	\quad\text{in probability}.
$$
\item $\boldsymbol{U}^N$ converges strongly to $\boldsymbol{u}$.
\end{enumerate}
\end{thm}

Theorem \ref{thm:ulin} is due to Pisier \cite{Pis96,Pis18}, but the 
elementary proof we present here is due to Lehner \cite[\S 5.1]{Leh99}. 
We will need a classical lemma.

\begin{lem}
\label{lem:dilation}
For any operator $X$ in a $C^*$-algebra $\mathcal{A}$, define its
self-adjoint dilation
$\tilde X= e_1e_2^*\otimes X + e_2e_1^*\otimes X^*$ in
$\M_2(\mathbb{C})\otimes\mathcal{A}$.
Then $\|X\|=\|\tilde X\|$ and $\spc(\tilde X)=-\spc(\tilde 
X)$.
\end{lem}

\begin{proof}
We first note that
$
	\|\tilde X\|^2 =
	\|\tilde X^2\| = \|e_1e_1^*\otimes XX^* + e_2e_2^*\otimes X^*X\|
	=\|X\|^2
$.
To show that the spectrum is symmetric, it suffices to note that
$\tilde X$ is unitarily conjugate to $-\tilde X$ since
$U\tilde X U^*=-\tilde X$ with
$U=(e_1e_1^*-e_2e_2^*)\otimes\id$.
\end{proof}

The main step in the proof of Theorem \ref{thm:ulin} is as follows.

\begin{lem}
\label{lem:linstep}
Fix $D,r\in\mathbb{N}$ and $A_{ij}\in\M_D(\mathbb{C})$ for 
$i,j\in[r]$. Then there exist $C\ge 0$, $D'\in\mathbb{N}$ and
$A_i'\in \M_{D'}(\mathbb{C})$ for $i\in[r]$ such that
$$
	\Bigg\|\sum_{i,j=1}^{r} A_{ij}\otimes U_i^*U_j\Bigg\| =
	\Bigg\|\sum_{i=1}^{r} A_i'\otimes U_i\Bigg\|^2 - C
$$
for any family of unitaries $U_1,\ldots,U_r$ 
in any $C^*$-algebra 
$\mathcal{A}$.
\end{lem} 

\begin{proof}
Let $X=\sum_{i,j=1}^{r} A_{ij}\otimes U_i^*U_j$. Lemma 
\ref{lem:dilation} yields $\tilde X=\sum_{i,j=1}^{r} \tilde A_{ij}\otimes 
U_i^*U_j$ with $\tilde A_{ij}=e_1e_2^*\otimes A_{ij} + 
e_2e_1^*\otimes A_{ji}^*$. We therefore obtain for any $c>0$
$$
	\|X\| + rc =
	\|\tilde X\| + rc =
	\|\tilde X + rc\id\| =
	\Bigg\|\sum_{i,j=1}^{r} (\tilde A_{ij}+c1_{i=j}\id)\otimes 
	U_i^*U_j\Bigg\|,
$$
where the second equality used
that $\tilde X$ has a symmetric spectrum. 

Now note that the $r\times r$ block matrix $\tilde A=(\tilde 
A_{ij}+c1_{i=j}\id)_{i,j\in[r]}\in \M_{2Dr}(\mathbb{C})$
is self-adjoint, and we can choose $c$ sufficiently large so that it is 
positive definite. Then we may write $\tilde A=B^*B$ for
$B\in\M_{2Dr}(\mathbb{C})$. Now view $B$ as an $1\times r$ block matrix
with $2Dr\times 2D$ blocks $B_1,\ldots,B_r$, so that
$\tilde A_{ij} + c1_{i=j}\id = B_i^*B_j$. Therefore
$
	\|X\|+rc =
	\|Y^*Y\| = \|Y\|^2
$
with $Y=\sum_{i=1}^r B_i\otimes U_i$. To conclude we let
$C=rc$, $D'=2Dr$, and define $A_i'$ by padding
$B_i$ with $2D(r-1)$ zero columns.
\end{proof}

We can now conclude the proof of Theorem \ref{thm:ulin}.

\begin{proof}[Proof of Theorem \ref{thm:ulin}]
Fix any $P\in\M_D(\mathbb{C})\otimes\mathbb{C}
\langle x_1,\ldots,x_r\rangle$ of degree at most $2^q$, and let
$\boldsymbol{u}=(u_1,\ldots,u_r)$ be unitaries in any $C^*$-algebra 
$\mathcal{A}$. Denote by $U_1,\ldots,U_R$ all monomials of 
degree at most $2^{q-1}$ in the variables 
$\boldsymbol{u},\boldsymbol{u}^*$. Then we may clearly express 
$P(\boldsymbol{u},\boldsymbol{u}^*)=
\sum_{i,j=1}^R A_{ij}\otimes U_i^*U_j$ for some matrix coefficients
$A_{ij}\in\M_D(\mathbb{C})$. Lemma \ref{lem:linstep} yields
$P'\in\M_{D'}(\mathbb{C})\otimes\mathbb{C}
\langle x_1,\ldots,x_r\rangle$ of degree at most $2^{q-1}$ so that
$$
	\|P(\boldsymbol{u},\boldsymbol{u}^*)\| =
	\|P'(\boldsymbol{u},\boldsymbol{u}^*)\|^2 - C.
$$
Iterating this procedure $q$ times and using Lemma \ref{lem:dilation}, we 
obtain a self-adjoint $Q\in\M_{D''}(\mathbb{C})\otimes\mathbb{C}
\langle x_1,\ldots,x_r\rangle$ of degree at most one and a real
polynomial $h$ so that
$$
	\|P(\boldsymbol{u},\boldsymbol{u}^*)\| = 
	h(\|Q(\boldsymbol{u},\boldsymbol{u}^*)\|)
$$
for any $r$-tuple of unitaries $\boldsymbol{u}=(u_1,\ldots,u_r)$ 
in any $C^*$-algebra $\mathcal{A}$.
As this identity
therefore applies also to $\boldsymbol{U}^N$, the implication 
$a\Rightarrow b$ follows immediately. The converse implication
$b\Rightarrow a$ follows from Lemma \ref{lem:mtxstr}.
\end{proof}

We have included a full proof of Theorem \ref{thm:ulin} to give a flavor 
of how the linearization trick comes about. In the rest of this section, 
we briefly discuss two additional linearization results without proof.

The proof of Theorem \ref{thm:ulin} relied crucially on the unitary 
assumption. It is tempting to conjecture that its conclusion extends to 
the non-unitary case. Unfortuntately, a simple example shows that this 
cannot be true.

\begin{example}
Consider any $D\in\mathbb{N}$ and
$P\in\M_D(\mathbb{C})\otimes\mathbb{C}\langle x_1\rangle$ of degree one, 
that is,
$P(x_1)=A_0\otimes\id + A_1\otimes x_1$. Then the spectral theorem yields
$$
	\|P(x)\| = \sup_{\lambda\in\spc(x)} \|A_0+\lambda A_1\|
$$
for every self-adjoint operator $x$.
Now let $x,y$ be self-adjoint operators with
$\spc(x)=[-1,1]$ and $\spc(y)=\{-1,1\}$. Since the right-hand side of the 
above identity is
the supremum of a convex function of $\lambda$, it is clear that
$\|P(x)\|=\|P(y)\|$ for every 
$P\in\M_D(\mathbb{C})\otimes\mathbb{C}\langle x_1\rangle$ of degree one. 
But clearly $\|1-x^2\|=1$ while $\|1-y^2\|=0$. 
\end{example}

This example shows that the norms of polynomials of degree one cannot 
detect gaps in the spectrum of a self-adjoint operator, while higher 
degree polynomials can. Thus the norm of degree one polynomials does not 
suffice for strong convergence in the self-adjoint setting. However, it 
was realized by Haagerup and Thorbj{\o}rnsen~\cite{HT05} that this issue 
can be surmounted by requiring convergence not just of the norm, but 
rather of the full spectrum, of degree one polynomials.

\begin{thm}[Self-adjoint linearization]
\label{thm:salin}
Let $\boldsymbol{X}^N=(X_1^N,\ldots,X_r^N)$ be a sequence of $r$-tuples
of self-adjoint random matrices, and let $\boldsymbol{x}=(x_1,\ldots,x_r)$ 
be an $r$-tuple of self-adjoint elements of a $C^*$-algebra $\mathcal{A}$. 
The following are equivalent.
\smallskip
\begin{enumerate}[a.]
\itemsep\medskipamount  
\item For every $D\in\mathbb{N}$ and self-adjoint
$P\in\M_D(\mathbb{C})\otimes\mathbb{C}\langle x_1,\dots,x_{r}\rangle$
of degree one,
$$
	\spc\big(P(\boldsymbol{X}^N)\big)
	\subseteq
	\spc\big(P(\boldsymbol{x})\big)
	+ o(1) [-1,1]
	\quad\text{with probability}\quad 1-o(1).
$$
\item For every $D\in\mathbb{N}$ and self-adjoint
$P\in\M_D(\mathbb{C})\otimes\mathbb{C}\langle x_1,\dots,x_{r}\rangle$
$$
	\|P(\boldsymbol{X}^N)\| \le 
	\|P(\boldsymbol{x})\|+o(1)
	\quad\text{with probability}\quad 1-o(1).
$$
\end{enumerate}
\end{thm}

We omit the proof, which may be found in \cite{HT05} or in \cite{HST06} 
(see also \cite[\S 10.3]{MS17}). Let us note that while this theorem only 
gives an upper bound, the corresponding lower bound will often follow from 
Lemma \ref{lem:upperlower}.

Finally, while we have focused on strong convergence, 
linearization tricks for weak convergence can be found in the
paper \cite{dlS10} of de la Salle. For example, we state the following 
result which follows readily from the proof of \cite[Lemma 1.1]{dlS10}.

\begin{lem}[Linearization and weak convergence]
Let $(\mathcal{A},\tau)$ be a $C^*$-probability space.
Then in the setting of Theorem \ref{thm:salin}, the following are 
equivalent.
\smallskip
\begin{enumerate}[a.]
\itemsep\medskipamount
\item For every $p,D\in\mathbb{N}$ and self-adjoint
$P\in\M_D(\mathbb{C})\otimes\mathbb{C}\langle x_1,\dots,x_{r}\rangle$
of degree one,
$$
	\lim_{N\to\infty} \ntr\big(P(\boldsymbol{X}^N)^{2p}\big)
	= ({\ntr}\otimes\tau)\big(P(\boldsymbol{x})^{2p}\big)
	\quad\text{in probability}.
$$
\item $\boldsymbol{X}^N$ converges weakly to $\boldsymbol{x}$.
\end{enumerate}
\end{lem}

Why is linearization useful? It is often the case that one can perform 
computations more easily for polynomials of degree one than for general 
polynomials. For example, linearization played a key role in the 
Haagerup--Thorbj{\o}rnsen proof of strong convergence of GUE matrices 
\cite{HT05} because the matrix Cauchy transform of polynomials of degree 
one can be computed by means of quadratic equations. Similarly, 
polynomials of degree one make the moment computations in the works of 
Bordenave and Collins \cite{BC19,BC20,BC24} tractable. However, the 
interpolation and polynomial methods discussed in section 
\ref{sec:newmethods} do not rely on linearization.

\subsection{Positivization}

The linearization trick of the previous section states that if we work 
with general matrix coefficients, it suffices to consider only polynomials 
of degree one. We now introduce (in the setting of group $C^*$-algebras) a 
complementary principle: if we admit polynomials of any degree, it 
suffices to consider only polynomials with \emph{positive} scalar 
coefficients. This \emph{positivization trick} was introduced 
in the work of Magee and de la Salle \cite[\S 6.2]{MdlS24}.\footnote{% 
This idea appears in \cite{MdlS24} in a slightly different context, cf.\ 
Remark \ref{rem:posnonfree}. The form of the positivization trick that is 
presented here was explained to the author by Mikael de la 
Salle.}

The positivization trick will rely on another nontrivial operator 
algebraic property that we introduce presently. Let us
fix a finitely generated group $\mathbf{G}$ with generators
$g_1,\ldots,g_r$, let $\lambda:\mathbf{G}\to B(l^2(\mathbf{G}))$ be
its left-regular representation, and let $\tau$ be the canonical
trace on $C^*_{\rm red}(\mathbf{G})$.
For simplicity, we will denote $u_i = \lambda(g_i)$.
Then for any $P\in\mathbb{C}\langle x_1,\ldots,x_{2r}\rangle$, we can
uniquely express
\begin{equation}
\label{eq:polygroup}
	P(\boldsymbol{u},\boldsymbol{u}^*) =
	\sum_{g\in\mathbf{G}}
	a_g \,\lambda(g)
\end{equation}
for some coefficients $a_g\in\mathbb{C}$ that vanish for all but a finite
number of $g\in\mathbf{G}$. Moreover, it is readily verified
using the definition of the trace that
$$
	\|P(\boldsymbol{u},\boldsymbol{u}^*)\|_2 =
	\tau(|P(\boldsymbol{u},\boldsymbol{u}^*)|^2)^{\frac{1}{2}}
	=
	\Bigg(
	\sum_{g\in\mathbf{G}} |a_g|^2
	\Bigg)^{\frac{1}{2}}.
$$
We can now introduce the following property.

\begin{defn}[Rapid decay property]
The group $\mathbf{G}$ is said to have the \emph{rapid decay property}
if there exists constants $C,c>0$ so that
$$
	\|P(\boldsymbol{u},\boldsymbol{u}^*)\| \le Cq^c
	\|P(\boldsymbol{u},\boldsymbol{u}^*)\|_2
$$
for all $q\in\mathbb{N}$ and $P\in\mathbb{C}\langle 
x_1,\ldots,x_{2r}\rangle$ of degree $q$.
\end{defn}

The key feature of this property is the polynomial dependence on 
degree $q$. This is a major improvement over the trivial bound obtained
by applying the triangle inequality and Cauchy--Schwarz, which would yield
such an inequality with an exponential constant $|\{g\in\mathbf{G}:a_g\ne 
0\}|^{1/2}\le (2r+1)^{q/2}$.

While the rapid decay property appears to be very strong, it is 
widespread. It was first proved by Haagerup \cite{Haa78} for the
free group $\mathbf{G}=\mathbf{F}_r$, for which rapid decay property
is known as the \emph{Haagerup inequality}. The rapid decay property is
now known to hold for many other groups, cf.\ \cite{Cha17}.

We are now ready to introduce the positivization trick. For simplicity, we 
formulate the result for the case of the free group 
$\mathbf{G}=\mathbf{F}_r$ (see Remark \ref{rem:posnonfree}).

\begin{lem}[Positivization]
\label{lem:pos}
Let $\boldsymbol{U}^N=(U_1^N,\ldots,U_r^N)$ be a sequence of $r$-tuples
of unitary random matrices, and let $\boldsymbol{u}=(u_1,\ldots,u_r)$
be defined as above for $\mathbf{G}=\mathbf{F}_r$
(that is, $u_i=\lambda(g_i)\in C^*_{\rm red}(\mathbf{F}_r)$).
Then the following are equivalent.
\smallskip
\begin{enumerate}[a.]
\itemsep\medskipamount
\item For every self-adjoint
$P\in\mathbb{R}_+\langle x_1,\dots,x_{2r}\rangle$
$$
	\|P(\boldsymbol{U}^N,\boldsymbol{U}^{N*})\| \le
	\|P(\boldsymbol{u},\boldsymbol{u}^*)\| + o(1)
	\quad\text{with probability}\quad 1-o(1).
$$
\item $\boldsymbol{U}^N$ converges strongly to $\boldsymbol{u}$.
\end{enumerate}
\end{lem}

\begin{proof}
The implication $b\Rightarrow a$ is trivial. To prove $a\Rightarrow b$,
fix any $P\in\mathbb{C}\langle x_1,\dots,x_{2r}\rangle$.
We may clearly assume without loss of generality that all the monomials
of $P$ are reduced (i.e., do not contain consecutive letters
$x_ix_i^*$ or $x_i^*x_i$), so that the coefficients of $P$ are precisely
those that appear in the representation \eqref{eq:polygroup}.

Let us write $P=P_1+iP_2$ for $P_1,P_2\in\mathbb{R}\langle 
x_1,\dots,x_{2r}\rangle$
defined by taking the real (imaginary) parts of the coefficients of 
$P$. Since the polynomials $P_j^*P_j$
are self-adjoint with real coefficients, we can write
$P_j^*P_j = Q_j-R_j$ for self-adjoint $Q_j,R_j\in \mathbb{R}_+\langle 
x_1,\dots,x_{2r}\rangle$ defined by keeping only the positive (negative)
coefficients of $P_j^*P_j$. Then we can estimate by the triangle 
inequality
\begin{align*}
	&\|P(\boldsymbol{U}^N,\boldsymbol{U}^{N*})\|^2 \le
	2(\|P_1(\boldsymbol{U}^N,\boldsymbol{U}^{N*})\|^2+
	\|P_2(\boldsymbol{U}^N,\boldsymbol{U}^{N*})\|^2) \\
	&\le
	2(\|Q_1(\boldsymbol{U}^N,\boldsymbol{U}^{N*})\|+\|R_1(\boldsymbol{U}^N,\boldsymbol{U}^{N*})\|+\|Q_2(\boldsymbol{U}^N,\boldsymbol{U}^{N*})\|+\|R_2(\boldsymbol{U}^N,\boldsymbol{U}^{N*})\|).
\end{align*}
On the other hand, note that
\begin{align*}
	\|Q_j(\boldsymbol{u},\boldsymbol{u}^*)\|_2^2 +
	\|R_j(\boldsymbol{u},\boldsymbol{u}^*)\|_2^2 &= 
	\|
	P_j(\boldsymbol{u},\boldsymbol{u}^*)^*
	P_j(\boldsymbol{u},\boldsymbol{u}^*)
	\|_2^2,
\\
	\|P_1(\boldsymbol{u},\boldsymbol{u}^*)\|_2^2 +
	\|P_2(\boldsymbol{u},\boldsymbol{u}^*)\|_2^2 &=
	\|P(\boldsymbol{u},\boldsymbol{u}^*)\|_2^2.
\end{align*}
We can therefore estimate
\begin{align*}
	&\|Q_1(\boldsymbol{u},\boldsymbol{u}^*)\|+
	\|R_1(\boldsymbol{u},\boldsymbol{u}^*)\|+
	\|Q_2(\boldsymbol{u},\boldsymbol{u}^*)\|+
	\|R_2(\boldsymbol{u},\boldsymbol{u}^*)\|
\\
	&\le Cq^c(
	\|P_1(\boldsymbol{u},\boldsymbol{u}^*)^*
        P_1(\boldsymbol{u},\boldsymbol{u}^*)\|_2 +
	\|P_2(\boldsymbol{u},\boldsymbol{u}^*)^*
        P_2(\boldsymbol{u},\boldsymbol{u}^*)\|_2)	
\\
	&\le Cq^c(
	\|P_1(\boldsymbol{u},\boldsymbol{u}^*)\|^2
	+
	\|P_2(\boldsymbol{u},\boldsymbol{u}^*)\|^2)	
\\
	&\le 
	C'q^{c'}\|P(\boldsymbol{u},\boldsymbol{u}^*)\|^2
\end{align*}
for some $C,C',c,c'>0$, where $q$ is the degree of $P$ and we have
applied the rapid decay property of $\mathbf{F}_r$ in the
first and last inequality. Thus $a$ implies that
$$
	\|P(\boldsymbol{U}^N,\boldsymbol{U}^{N*})\|
	\le
	Cq^{c}\|P(\boldsymbol{u},\boldsymbol{u}^*)\| + o(1)
	\quad\text{with probability}\quad 1-o(1)
$$
for every $P\in\mathbb{C}\langle x_1,\ldots,x_{2r}\rangle$ of degree
at most $q$ and some constants $C,c'>0$.

Now note that, for every $p\in\mathbb{N}$, applying the above
to $(P^*P)^p$ yields
$$
	\|P(\boldsymbol{U}^N,\boldsymbol{U}^{N*})\|
	\le
	C(2pq)^{\frac{c}{2p}}\|P(\boldsymbol{u},\boldsymbol{u}^*)\| + o(1)
	\quad\text{with probability}\quad 1-o(1).
$$
Taking $p\to\infty$ yields the strong convergence upper bound, and the 
lower bound now follows from Lemma \ref{lem:upperlower} since 
$C^*_{\rm red}(\mathbf{F}_r)$ has the unique trace property.
\end{proof}

The positivization trick is very useful in the context of the polynomial 
method, as we will see in section \ref{sec:poly}. Let us however give a 
hint as to its significance.

For a self-adjoint polynomial $P$ with positive coefficients, we may 
interpret \eqref{eq:polygroup} as defining the adjacency matrix of a 
weighted graph with vertex set $\mathbf{G}$, where we place an edge 
with weight $a_g$ between every pair of vertices $(w,gw)$ with 
$w\in\mathbf{G}$ and $a_g>0$. Thus, for example, computing the 
moments of $P(\boldsymbol{u},\boldsymbol{u}^*)$ is in essence a 
combinatorial problem of counting the number of closed walks in this 
graph. This greatly facilitates the analysis of such quantities; for 
example, we can obtain upper bounds by overcounting some of the walks.

For a general choice of $P$, we may still view 
$P(\boldsymbol{u},\boldsymbol{u}^*)$ as a kind of adjacency matrix of a 
graph with complex edge weights. This is a much more complicated object, 
however, since the moments of this operator may exhibit cancellations 
between different walks and can therefore no longer by treated as a 
counting problem. The surprising consequence of the positivization trick 
is that for the purposes of proving strong convergence, we can completely 
ignore these cancellations and restrict attention only to the 
combinatorial situation.

\begin{rem}
\label{rem:posnonfree}
The only part of the proof of Lemma \ref{lem:pos} where we used
$\mathbf{G}=\mathbf{F}_r$ is in the very first step,
where we argued that we may assume that the coefficients of $P$ agree
with those in the representation \eqref{eq:polygroup}. 
For other groups $\mathbf{G}$, it is not clear that this is the case 
unless
we assume that the matrices $\boldsymbol{U}^N$ also satisfy the group 
relations, i.e., that $U_i^N=\pi_N(g_i)$ where
$\pi_N:\mathbf{G}\to\M_N(\mathbb{C})$ is a (random) 
unitary representation of $\mathbf{G}$. Under the latter assumption,
Lemma \ref{lem:pos} extends directly to any 
$\mathbf{G}$ with the rapid decay and unique trace properties.

Alternatively, when the positivization trick is applied to the polynomial 
method, it is possible to apply a variant of the argument directly to the 
limiting object that appears in the proof, avoiding the need to invoke 
properties of the random matrices. This form of the positivization trick 
is developed in \cite[\S 6.2]{MdlS24} (cf.\ Remark \ref{rem:mdlspos}).
\end{rem}

\section{The polynomial method}
\label{sec:poly}

The polynomial method, which was introduced in the recent work of Chen, 
Garza-Vargas, Tropp, and the author \cite{CGTV25}, has enabled 
significantly simpler proofs of strong convergence and has opened the 
door to various new developments. The method was briefly introduced in 
section \ref{sec:intropoly} above. In this section, we aim to provide a 
detailed 
illustration of this method by using it to prove strong convergence of 
random permutation matrices (Theorem \ref{thm:bc}).

We will follow a simplified form of the treatment in 
\cite{CGTV25}. The simplifications arise for two reasons: we will make no 
attempt to get good quantitative bounds, enabling us to to use crude 
estimates in various places; and we will take advantage of the idea of 
\cite{MdlS24} to significantly simplify one part of the argument by 
exploiting positivization. Aside from the use of standard results on 
polynomials and Schwartz distributions, the proof given here
is essentially self-contained.

Despite its simplicity, what makes the polynomial method work appears 
rather mysterious at first sight. We will conclude this section with a 
discussion of the new phenomenon that is captured by this method 
(section \ref{sec:discpoly}).

Significant refinements of the polynomial method may be found in 
\cite{CGV25,MPV25}.

\subsection{Outline}

In the following, we fix independent random permutation matrices 
$\boldsymbol{U}^N=(U_1^N,\ldots,U_r^N)$ and the limiting model 
$\boldsymbol{u}=(u_1,\ldots,u_r)$ as in Theorem~\ref{thm:bc}. 
More precisely, recall that 
$u_i=\lambda(g_i)$, where $g_1,\ldots,g_r$ and $\lambda$ are the free 
generators and left-regular representation of $\mathbf{F}_r$. We will view 
$\boldsymbol{u}$ as living in the $C^*$-probability space $(C^*_{\rm 
red}(\mathbf{F}_r),\tau)$ where $\tau$ denotes the canonical trace.

For notational purposes, it will be convenient to define $g_0=e$ and 
$g_{r+i}=g_i^{-1}$ for $i=1,\ldots,r$. We analogously define $u_0=\id$ and 
$u_{r+i}=u_i^*$, and similarly $U_0^N=\id$ and $U_{r+i}^N = U_i^{N*}$, 
for $i=1,\ldots,r$. We will think of $r$ as fixed, and all constants that
appear in this section may depend on $r$.

We begin by outlining the key ingredients that are 
needed to conclude the proof. These ingredients will then be developed
in the remainder of this section.

\subsubsection{Polynomial encoding}
\label{sec:ipolyenc}

The first step of the analysis is to show that 
the expected traces of monomials of
$\boldsymbol{U}^N|_{1^\perp}$ are
rational expressions of $\frac{1}{N}$.

\begin{lem}
\label{lem:polyenc}
For every $q\in\mathbb{N}$ and $\boldsymbol{w}=(w_1,\ldots,w_q)\in
\{0,\ldots,2r\}^q$, there exist
real polynomials $f_{\boldsymbol{w}}$ and $g_q$ of degree at most $Cq$ so 
that for all $N\ge q$
$$
	\mathbf{E}\big[\ntr U_{w_1}^N\cdots U_{w_q}^N|_{1^\perp}\big]
	= \frac{f_{\boldsymbol{w}}(\frac{1}{N})}{g_q(\frac{1}{N})} =
	\Phi_{\boldsymbol{w}}(\tfrac{1}{N}).
$$
\end{lem}

Lemma \ref{lem:polyenc} immediately implies that
$$
	\mathbf{E}\big[\ntr U_{w_1}^N\cdots U_{w_q}^N|_{1^\perp}\big] =
	\mu_0(\boldsymbol{w}) +
	\frac{\mu_1(\boldsymbol{w})}{N} + O\bigg(\frac{1}{N^2}\bigg)
$$
as $N\to\infty$. The values of $\mu_0(\boldsymbol{w})$ and
$\mu_1(\boldsymbol{w})$ can be easily read off from the proof of Lemma 
\ref{lem:polyenc}. In particular, it will follow that
\begin{equation}
\label{eq:polyweak}
	\mu_0(\boldsymbol{w}) = 1_{g_{w_1}\cdots g_{w_q}=e} = 
	\tau(u_{w_1}\cdots u_{w_q}),
\end{equation}
which essentially establishes weak convergence of 
$\boldsymbol{U}^N|_{1^\perp}$ to $\boldsymbol{u}$ (albeit in expectation 
rather than in probability; this will not be important in what follows).

\subsubsection{Asymptotic expansion}

Now fix a self-adjoint noncommutative polynomial $P\in\mathbb{C}\langle 
x_1,\ldots,x_{2r}\rangle$. Then for every univariate real 
polynomial $h$, since $h\circ P$ is again a noncommutative polynomial, we 
immediately obtain
\begin{equation}
\label{eq:expansion}
	\mathbf{E}\big[ \ntr 
	h(P(\boldsymbol{U}^N,\boldsymbol{U}^{N*}))|_{1^\perp}\big]
	= \nu_0(h) + \frac{\nu_1(h)}{N} + O\bigg(\frac{1}{N^2}\bigg).
\end{equation}
Here $\nu_0$ and $\nu_1$ are defined, \emph{a priori}, as linear 
functionals on the space $\mathcal{P}$ of all univariate real polynomials
(of course, $\nu_0,\nu_1$ also depend on the choice of $P$, but we will 
view $P$ as fixed throughout the argument).

The core of the proof is now to show that the expansion 
\eqref{eq:expansion} is valid not only for polynomial test functions 
$h\in\mathcal{P}$, but even for arbitrary smooth test functions $h\in 
C^\infty(\mathbb{R})$. It is far from obvious why this should be the case; 
for example, it is conceivable that there could exist smooth test 
functions $h$ for which weak convergence takes place at a rate slower than 
the $\frac{1}{N}$ rate for polynomial $h$. If that were to be the case,
then $\nu_1(h)$ would not even make sense for smooth $h$. We will show, 
however, that this hypothetical scenario is not realized.

Recall that a linear functional $\nu$ on $C^\infty(\mathbb{R})$ is
called a \emph{compactly supported (Schwartz) distribution} (see 
\cite[Chapter II]{Hor03}) if
$$
	|\nu(h)| \le C\|h\|_{C^m[-K,K]}
	\quad\text{for all }h\in C^\infty(\mathbb{R})
$$
holds for some constants $C,K\in\mathbb{R}_+$ and $m\in\mathbb{Z}_+$.

\begin{prop}
\label{prop:smex}
For every self-adjoint $P\in\mathbb{C}\langle x_1,\ldots,x_{2r}\rangle$,
the corresponding linear functionals $\nu_0,\nu_1$ in \eqref{eq:expansion} 
extend to compactly supported Schwartz distributions, and the expansion
\eqref{eq:expansion} remains valid for any $h\in C^\infty(\mathbb{R})$.
\end{prop}

Note that it is immediate from \eqref{eq:polyweak} that
$$
	\nu_0(h) = \tau\big(h(P(\boldsymbol{u},\boldsymbol{u}^*))\big)
$$
for all $h\in C^\infty(\mathbb{R})$. In other words, 
$\nu_0=\mu_{P(\boldsymbol{u},\boldsymbol{u}^*)}$ is nothing other
than the spectral distribution of $P(\boldsymbol{u},\boldsymbol{u}^*)$.
The nontrivial aspect of Proposition \ref{prop:smex} is that $\nu_1$
and the expansion \eqref{eq:expansion} make sense for smooth $h$ as well.

The proof of Proposition \ref{prop:smex} is the key point of the 
polynomial method. We will exploit the Markov inequality to 
achieve a quantitative form of \eqref{eq:expansion} for $h\in\mathcal{P}$. 
The resulting bound is so strong that it can be extended to any
$h\in C^\infty(\mathbb{R})$ by means of a simple Fourier-analytic 
argument.

\subsubsection{The infinitesimal distribution}
\label{sec:infdist}

As $\nu_0=\mu_{P(\boldsymbol{u},\boldsymbol{u}^*)}$,
Lemma \ref{lem:faith} yields
$$
	\supp\nu_0 \subseteq [-\|P(\boldsymbol{u},\boldsymbol{u}^*)\|,
	\|P(\boldsymbol{u},\boldsymbol{u}^*)\|].
$$
The final ingredient of the proof is to show that $\nu_1$ satisfies
the same bound. By the positivization trick, it suffices 
to assume that $P$ has positive coefficients.

\begin{lem}
\label{lem:supp}
For every choice of self-adjoint $P\in\mathbb{R}_+\langle 
x_1,\ldots,x_{2r}\rangle$, we have
$$
	\supp\nu_1 \subseteq [-\|P(\boldsymbol{u},\boldsymbol{u}^*)\|,
	\|P(\boldsymbol{u},\boldsymbol{u}^*)\|].
$$
\end{lem}

To prove Lemma \ref{lem:supp} we face a conundrum: while we know
abstractly that $\nu_1$ is a compactly supported distribution, we 
are only able to compute its value for polynomial test functions 
(as we have an explicit formula for $\mu_1(\boldsymbol{w})$ in
section \ref{sec:ipolyenc}). To surmount this issue, we will use the
following general fact \cite[Lemma 4.9]{CGTV25}:
for any compactly supported distribution $\nu$, we have
$$
	\supp\nu\subseteq [-\rho,\rho] \qquad\text{with}\qquad
	\rho = \limsup_{p\to\infty} |\nu(x^p)|^{\frac{1}{p}}.
$$
Thus it suffices to show that
$$
	\limsup_{p\to\infty} |\nu_1(x^p)|^{\frac{1}{p}}
	\le
	\|P(\boldsymbol{u},\boldsymbol{u}^*)\|,
$$
which is tractable as we have access to the moments of $\nu_1$.
It is this moment estimate that is greatly
simplified by the assumption that $P$ has positive coefficients.

\subsubsection{Proof of Theorem \ref{thm:bc}}

We now use these ingredients to conclude the proof.

\begin{proof}[Proof of Theorem \ref{thm:bc}]
Fix $\varepsilon>0$ and a self-adjoint $P$ with positive coefficients.
Moreover, let $h$ be a nonnegative smooth function
that vanishes in a neighborhood of $[-\|P(\boldsymbol{u},\boldsymbol{u}^*)\|,
\|P(\boldsymbol{u},\boldsymbol{u}^*)\|]$ and such that
$h(x)=1$ for $|x|\ge \|P(\boldsymbol{u},\boldsymbol{u}^*)\|+\varepsilon$.

Note that $\nu_0(h)=\nu_1(h)=0$ by Lemma \ref{lem:supp}.
Thus Proposition \ref{prop:smex} yields
$$
	\mathbf{E}\big[ \tr
        h(P(\boldsymbol{U}^N,\boldsymbol{U}^{N*}))|_{1^\perp}\big]
	=
	N\,\mathbf{E}\big[ \ntr
        h(P(\boldsymbol{U}^N,\boldsymbol{U}^{N*}))|_{1^\perp}\big]
	=
	o(1)
$$
as $N\to\infty$. But since $\tr h(X)\ge 1$ whenever
$\|X\|\ge 
\|P(\boldsymbol{u},\boldsymbol{u}^*)\|+\varepsilon$, this
implies
$$	
	\mathbf{P}\big[
	\|P(\boldsymbol{U}^N,\boldsymbol{U}^{N*})|_{1^\perp}\|\ge 
	\|P(\boldsymbol{u},\boldsymbol{u}^*)\|+\varepsilon\big]=o(1).
$$
As $P,\varepsilon$ are arbitrary, we verified condition $a$
of Lemma \ref{lem:pos}.
\end{proof}

We now turn to the proofs of the various ingredients described above.

\subsection{Polynomial encoding}

The aim of this section is to prove Lemma \ref{lem:polyenc}. We follow
\cite{PL10}; see also \cite{Nic94,Col22}. We begin by noting that
$$
	N\,\mathbf{E}\big[\ntr U_{w_1}^N\cdots U_{w_q}^N|_{1^\perp}\big]
	=
	\mathbf{E}\big[\tr U_{w_1}^N\cdots U_{w_q}^N|_{1^\perp}\big]
	=
	\mathbf{E}\big[{\tr U_{w_1}^N\cdots U_{w_q}^N}\big] -1,
$$
so that it suffices to compute the rightmost expectation. Clearly
$$
	\mathbf{E}\big[{\tr U_{w_1}^N\cdots U_{w_q}^N}\big] =
	\sum_{i_1,\ldots,i_q\in[N]}
	\mathbf{E}\big[ (U_{w_1}^N)_{i_1i_2}
	(U_{w_2}^N)_{i_2i_3}\cdots (U_{w_q}^N)_{i_qi_1}\big].
$$
A tuple $\boldsymbol{i}=(i_1,\ldots,i_q)\in[N]^q$ is
\emph{realizable} if the corresponding summand is
nonzero. Denote by $\mathcal{I}_N(\boldsymbol{w})$ the set of all 
realizable tuples.

To bring out the dependence on dimension $N$, we note that by symmetry, 
the expectation inside the above sum only depends on how many distinct pairs 
of indices appear for each permutation matrix. To encode this information, 
we associate to each $\boldsymbol{i}\in \mathcal{I}_N(\boldsymbol{w})$ a 
directed edge-colored graph $\Gamma$ as follows. Number each distinct 
value among $(i_1,\ldots,i_q)$ by order of appearance, and assign to each 
a vertex. Now draw an edge colored $w\in[r]$ from one vertex to another if 
$(U^N_w)_{ii'}$ or $(U^N_{r+w})_{i'i}$ appears in the expectation, where 
$i,i'\in[N]$ are the values associated to the first and second vertex, 
respectively; see Figure \ref{fig:core}.

\begin{figure}
\centering
\begin{tikzpicture}

\fill (0,0) circle (0.075) node[below] {$\scriptstyle 1$};
\fill (3,0) circle (0.075) node[below] {$\scriptstyle 2$};
\fill (6,0) circle (0.075) node[below] {$\scriptstyle 3$};

\draw[->,thick,color=red] (0.075,0.075) to[out=30,in=150] (2.925,0.075);
\draw[color=red] (1.5,0.7) node {$\scriptstyle 1$};
\draw[->,thick,color=red] (3.075,0.075) to[out=30,in=150] (5.925,0.075);
\draw[color=red] (4.5,0.7) node {$\scriptstyle 1$};

\draw[->,thick,color=blue] (0.075,-0.075) to[out=-30,in=-150] (2.925,-0.075);
\draw[color=blue] (1.5,-0.7) node {$\scriptstyle 2$};
\draw[->,thick,color=blue] (3.075,-0.075) to[out=-30,in=-150] (5.925,-0.075);
\draw[color=blue] (4.5,-0.7) node {$\scriptstyle 2$};

\end{tikzpicture}
\caption{Graph $\Gamma$ associated to the term
$\mathbf{E}[
(U_1^N)_{58}
(U_2^N)_{86}
(U_1^{N*})_{68}
(U_2^{N*})_{85}]$. The vertices labelled $1,2,3$ correspond to the values 
$5,8,6$, respectively.
\label{fig:core}}
\end{figure}
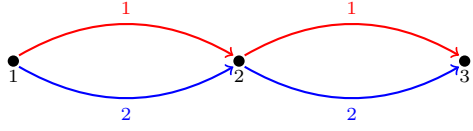

Denote by $\mathcal{G}(\boldsymbol{w})$ the set of graphs $\Gamma$ thus 
constructed, and note that this set is independent of $N$. For each such 
graph with $v_\Gamma$ vertices, we can recover all associated 
$\boldsymbol{i}\in\mathcal{I}_N(\boldsymbol{w})$ uniquely by assigning 
distinct values of $[N]$ to its vertices. There are 
$N(N-1)\cdots(N-v_\Gamma+1)$ ways to do this.
If the graph has $e_\Gamma^w$ edges with color 
$w$, then the corresponding expectation for each such $\boldsymbol{i}$ is
$$
	\mathbf{E}\big[ (U_{w_1}^N)_{i_1i_2}
	(U_{w_2}^N)_{i_2i_3}\cdots (U_{w_q}^N)_{i_qi_1}\big] =
	\prod_{w=1}^{r}
	\frac{1}{N(N-1)\cdots (N-e_\Gamma^w+1)},
$$
since the random variable in the expectation is the event that for
each $w$, the permutation matrix $U_w^N$ has $e_\Gamma^w$ of its rows
fixed as specified by the realizable tuple $\boldsymbol{i}$.
Here we presumed that $N\ge q$, which ensures that
$N\ge v_\Gamma$ and $N\ge e_\Gamma^w$.

In summary, we have proved the following.

\begin{lem}
For every $\boldsymbol{w}=(w_1,\ldots,w_q)$ and $N\ge q$, we have
$$
	\mathbf{E}\big[{\tr U_{w_1}^N\cdots U_{w_q}^N}\big] =
	\sum_{\Gamma\in \mathcal{G}(\boldsymbol{w})}
	\frac{N(N-1)\cdots(N-v_\Gamma+1)}{
	\prod_{w=1}^{r} N(N-1)\cdots (N-e_\Gamma^w+1)}.
$$
\end{lem}

The proof of Lemma \ref{lem:polyenc} is now straightforward.

\begin{proof}[Proof of Lemma \ref{lem:polyenc}]
We can rewrite the above lemma as
$$
	\mathbf{E}\big[\ntr U_{w_1}^N\cdots U_{w_q}^N|_{1^\perp}\big] =
	\sum_{\Gamma\in \mathcal{G}(\boldsymbol{w})}
	\bigg(\frac{1}{N}\bigg)^{e_\Gamma-v_\Gamma+1}
	\frac{\prod_{k=1}^{v_\Gamma-1} (1-\frac{k}{N})}{
	\prod_{w=1}^{r} \prod_{k=1}^{e_\Gamma^w-1} (1-\frac{k}{N})}
	-\frac{1}{N},
$$
where $e_\Gamma$ is the total number of edges in $\Gamma$.
As every $\Gamma$ is connected by construction, we have
$e_\Gamma-v_\Gamma+1\ge 0$ and thus the right-hand side is
a rational function of $\frac{1}{N}$.

Define a polynomial of degree $r(q-1)$ by
$$
	g_q(x) = (1-x)^r(1-2x)^r\cdots (1-(q-1)x)^r.
$$
Since $e_\Gamma^w\le e_\Gamma\le q$ for all $w$, it is clear that
$f_{\boldsymbol{w}}(\frac{1}{N})=
\mathbf{E}[\ntr U_{w_1}^N\cdots U_{w_q}^N|_{1^\perp}]\,g_q(\frac{1}{N})$
is a polynomial of degree at most $Cq$ for some constant $C$ (which 
depends on $r$).
\end{proof}

We can now read off the first terms in the $\frac{1}{N}$-expansion. 
Recall that $g\in\mathbf{F}_r\backslash\{e\}$ is called a \emph{proper 
power} if $g=v^k$ for some $v\in\mathbf{F}_r$, $k\ge 2$, and is called 
a \emph{non-power} otherwise. Every $g\ne e$ can be 
written uniquely as $g=v^k$ for a non-power $v$.

\begin{cor}
\label{cor:divisor}
We have
$$
	\mu_0(\boldsymbol{w}) =
	\lim_{N\to\infty} 
	\mathbf{E}\big[\ntr U_{w_1}^N\cdots U_{w_q}^N|_{1^\perp}\big]
	= 1_{g_{w_1}\cdots g_{w_q}=e}.
$$
Moreover, if $g_{w_1}\cdots g_{w_q}=v^k$ for a non-power $v$, then  
$$
	\mu_1(\boldsymbol{w}) =
	\lim_{N\to\infty} N\,
	\mathbf{E}\big[\ntr U_{w_1}^N\cdots U_{w_q}^N|_{1^\perp}\big]
	= \omega(k)-1,
$$
where $\omega(k)$ denotes the number of divisors of $k$.
\end{cor}

\begin{proof}
If $g_{w_1}\cdots g_{w_q}=e$, it is obvious that 
$\mu_0(\boldsymbol{w})=1$. We therefore assume this is not the case. We 
may further assume that $g_{w_1}\cdots g_{w_q}$ is cyclically reduced, 
since the left-hand side of Lemma \ref{lem:polyenc} is unchanged under 
cyclic reduction. Then every vertex of any 
$\Gamma\in\mathcal{G}(\boldsymbol{w})$ must have degree at least two.

For the first identity, it now suffices to note that there cannot exist 
$\Gamma\in\mathcal{G}(\boldsymbol{w})$ with $e_\Gamma-v_\Gamma+1=0$: this 
would imply that $\Gamma$ is a tree, which must have a vertex of 
degree one. Thus the expression in the proof of Lemma \ref{lem:polyenc} 
yields $\mu_0(\boldsymbol{w})=0$.

We can similarly read off from the proof of Lemma \ref{lem:polyenc} that
$$
	\mu_1(\boldsymbol{w}) = \#\{\Gamma\in\mathcal{G}(\boldsymbol{w}):
	e_\Gamma-v_\Gamma=0\} - 1.
$$
That $e_\Gamma-v_\Gamma=0$ implies (as each vertex has degree at least 
two) that $\Gamma$ is a cycle. As $\boldsymbol{w}$ defines a closed 
nonbacktracking walk in $\Gamma$, it must go around the cycle an integer
number of times, so the possible lengths of cycles are
the divisors of $k$.
\end{proof}

\subsection{The master inequality}

We now proceed to the core of the polynomial method. Our main tool 
is the following inequality of A.\ Markov \cite[p.\ 91]{Che98}.

\begin{lem}[Markov inequality]
\label{lem:markov}
For any real polynomial $f$ of degree $q$ and $a>0$
$$
	\|f'\|_{L^\infty[0,a]} \le \frac{2q^2}{a}\|f\|_{L^\infty[0,a]}.
$$
\end{lem}

As well known consequence of the Markov inequality is that a bound
on a polynomial on a sufficiently fine grid extends to a uniform bound
\cite[p.\ 91]{Che98}. For completeness, we spell out the argument in the 
form we will need it.

\begin{cor}
\label{cor:cheney}
For any real polynomial $f$ of degree $q$ and $M\ge 2q^2$, we have
$$
	\|f\|_{L^\infty[0,\frac{1}{M}]} \le 
	2\sup_{N\ge M} |f(\tfrac{1}{N})|.
$$
\end{cor}

\begin{proof}
For any $x\in[0,\frac{1}{M}]$, its distance to the set
$\{\frac{1}{N}\}_{N\ge M}$ is at most $\frac{1}{2M^2}$. Thus
$$
	\|f\|_{L^\infty[0,\frac{1}{M}]} \le
	\sup_{N\ge M} |f(\tfrac{1}{N})| + 
	\frac{1}{2M^2}\|f'\|_{L^\infty[0,\frac{1}{M}]}
	\le
	\sup_{N\ge M} |f(\tfrac{1}{N})| +
        \frac{q^2}{M} \|f\|_{L^\infty[0,\frac{1}{M}]}
$$
by the Markov inequality. The conclusion follows.
\end{proof}

In the following, we fix a self-adjoint
$P\in\mathbb{C}\langle x_1,\ldots,x_{2r}\rangle$ of degree $q_0$.
For every polynomial test function $h\in\mathcal{P}$ of degree $q$,
Lemma \ref{lem:polyenc} yields
$$
	\mathbf{E}\big[ \ntr h(P(\boldsymbol{U}^N,\boldsymbol{U}^{N*}))
	|_{1^\perp}\big] = \frac{f_h(\frac{1}{N})}{g_{qq_0}(\frac{1}{N})}
	=\Phi_h(\tfrac{1}{N})
$$
where $f_h,g_{qq_0}$ are real polynomials of degree at most $Cq$ and
$g_{qq_0}$ is defined in the proof of Lemma \ref{lem:polyenc}.
We define $\nu_0(h),\nu_1(h)$ for $h\in\mathcal{P}$ as in
\eqref{eq:expansion}, and denote by $K$ the sum of the moduli of
the coefficients of $P$. Note that all the above objects depend
on the choice of $P$, which we consider fixed.

The key idea is to use the Markov inequality to bound the
derivatives of $\Phi_h$.

\begin{lem}
\label{lem:polymethod}
For any $h\in\mathcal{P}$ of degree $q$, we have 
$$
	\|\Phi_h'\|_{L^\infty[0,\frac{1}{N}]} \le
	Cq^4\|h\|_{L^\infty[-K,K]},
	\qquad
	\|\Phi_h''\|_{L^\infty[0,\frac{1}{N}]} \le
	Cq^8\|h\|_{L^\infty[-K,K]}
$$
for all $N\ge Cq^2$, where $C$ is a constant (which depends on $P$).
\end{lem}

\begin{proof}
It is easily verified using the explicit expression for $g_{qq_0}$ in the 
proof of Lemma \ref{lem:polyenc} that there are constants $C,c>0$ (which 
depend on $P$) so that
$$
	c\le g_q(x)\le 1, \qquad
	\bigg|\frac{g_{qq_0}'(x)}{g_{qq_0}(x)}\bigg| \le Cq^2,\qquad
	\bigg|\frac{g_{qq_0}''(x)}{g_{qq_0}(x)}\bigg| \le Cq^4
$$
for all $x\in[0,\frac{1}{q^2}]$.
We now simply apply the chain rule. For the first derivative,
$$
	\|\Phi_h'\|_{L^\infty[0,\frac{1}{Cq^2}]} = 
	\bigg\|\frac{f_h'}{g_{qq_0}} -
	\frac{f_h}{g_{qq_0}}
	\frac{g_{qq_0}'}{g_{qq_0}}\bigg\|_{L^\infty[0,\frac{1}{Cq^2}]}
	\le
	\frac{3C}{c} q^4 \|f_h\|_{L^\infty[0,\frac{1}{Cq^2}]}
$$
using Lemma \ref{lem:markov}. But Corollary \ref{cor:cheney} yields
$$
	\|f_h\|_{L^\infty[0,\frac{1}{Cq^2}]} \lesssim
	\sup_{N\ge q^2} |f_h(\tfrac{1}{N})| \le
	\sup_{N\ge q^2} |\Phi_h(\tfrac{1}{N})| \le
	\|h\|_{L^\infty[-K,K]},
$$
where we used $g_q\le 1$ in the second inequality and that
$\|P(\boldsymbol{U}^N,\boldsymbol{U}^{N*})\|\le K$ in the last inequality.
The bound on $\Phi_h''$ is obtained
in a completely analogous manner.
\end{proof}

We now easily obtain a quantitative form of \eqref{eq:expansion}.

\begin{cor}[Master inequality]
\label{cor:master}
For every $h\in\mathcal{P}$ of degree $q$ and $N\ge 1$,
$$
	\bigg| 
	\mathbf{E}\big[ \ntr 
	h(P(\boldsymbol{U}^N,\boldsymbol{U}^{N*}))|_{1^\perp}\big]
	- \nu_0(h) - \frac{\nu_1(h)}{N} \bigg|
	\le \frac{Cq^8}{N^2} \|h\|_{L^\infty[-K,K]},
$$
as well as $|\nu_1(h)| \le Cq^4\|h\|_{L^\infty[-K,K]}$.
\end{cor}

\begin{proof}
The bound on $|\nu_1(h)|$ follows immediately from Lemma 
\ref{lem:polymethod} as $\nu_1(h)=\Phi_h'(0)$. 
Now note that the left-hand side of the equation display in the statement 
equals
$$
	\big|\Phi_h(\tfrac{1}{N}) - \Phi_h(0) - \tfrac{1}{N}\Phi_h'(0)\big|
	\le \tfrac{1}{2N^2}\|\Phi_h''\|_{L^\infty[0,\frac{1}{N}]}.
$$
Thus the bound in the statement follows for $N\ge Cq^2$ from
Lemma
\ref{lem:polymethod}. On the other hand, when $N<Cq^2$, we can trivially
bound
$$
	\bigg| 
	\mathbf{E}\big[ \ntr 
	h(P(\boldsymbol{U}^N,\boldsymbol{U}^{N*}))|_{1^\perp}\big] 
	- \nu_0(h) - \frac{\nu_1(h)}{N} \bigg| \le
	\bigg(2+ \frac{Cq^4}{N}\bigg)\|h\|_{L^\infty[-K,K]}
$$
by the triangle inequality, as 
$\nu_0=\mu_{P(\boldsymbol{u},\boldsymbol{u}^*)}$ is supported
in $[-K,K]$, and using the bound on $\nu_1(h)$.
The conclusion follows using $1<\frac{Cq^2}{N}$.
\end{proof}

\subsection{Extension to smooth functions}

We are now ready to prove Proposition~\ref{prop:smex}. To this end, we 
will show that Corollary \ref{cor:master} can be extended to smooth test 
functions $h$ using a simple Fourier-analytic argument.

Recall that the \emph{Chebyshev polynomial (of the first kind)} $T_n$ is
the polynomial of degree $n$ defined by 
$T_n(\cos\theta)=\cos(n\theta)$. Any $h\in\mathcal{P}$ of degree $q$
can be written as
$$
	h(x) = \sum_{n=0}^q a_n\,T_n(K^{-1}x)
$$
for some real coefficients $a_0,\ldots,a_q$. Note that the latter are 
merely the Fourier coefficients of the function $\tilde 
h:S^1\to\mathbb{R}$ defined by $\tilde h(\theta)=h(K\cos\theta)$.

\begin{proof}[Proof of Proposition~\ref{prop:smex}]
Fix any $h\in\mathcal{P}$ and let $a_0,\ldots,a_q$ be its Chebyshev
coefficients as above. As $\|T_n\|_{L^\infty[-1,1]}=1$ for all $n$,
we can estimate
$$
	|\nu_1(h)| \le
	\sum_{n=0}^q |a_n|\, |\nu_1(T_n(K^{-1}\cdot))|
	\le
	C\sum_{n=0}^q n^4|a_n|,
$$
using the estimate on $\nu_1$ in Corollary \ref{cor:master}.
Now note that $n^ka_n$ is the $n$th Fourier coefficient of
the $k$th derivative $\tilde h^{(k)}$ of $\tilde h$.
We can therefore estimate
$$
	|\nu_1(h)| \le
	C\Bigg(\sum_{n=0}^q \frac{1}{n^2}\Bigg)^{\frac{1}{2}}
	\Bigg(
	\sum_{n=0}^q  n^{10}|a_n|^2\Bigg)^{\frac{1}{2}} 
	\le C' \|\tilde h^{(5)}\|_{L^2(S^1)}
	\le C'' \|h\|_{C^5[-K,K]}
$$
by Cauchy--Schwarz and Parseval, where the last inequality is obtained by 
applying the chain rule to
$\tilde h(\theta)=h(K\cos\theta)$. Since
this estimate holds for all $h\in\mathcal{P}$, the definition
of $\nu_1$ extends uniquely by continuity to any $h\in 
C^\infty(\mathbb{R})$. In particular, $\nu_1$ extends to a compactly 
supported distribution.

Applying the identical argument to the first inequality of
Corollary \ref{cor:master} yields
$$
	\bigg| 
	\mathbf{E}\big[ \ntr 
	h(P(\boldsymbol{U}^N,\boldsymbol{U}^{N*}))|_{1^\perp}\big]
	- \nu_0(h) - \frac{\nu_1(h)}{N} \bigg|
	\le \frac{C}{N^2} \|h\|_{C^9[-K,K]}
$$
for all $h\in C^\infty(\mathbb{R})$. In particular,
\eqref{eq:expansion} remains valid for any $h\in C^\infty(\mathbb{R})$.
\end{proof}

\subsection{The infinitesimal distribution}

It remains to prove Lemma \ref{lem:supp}. As was explained in section 
\ref{sec:infdist}, this result follows immediately from the following 
lemma, whose proof uses a spectral graph theory argument due to 
\cite[Lemma 2.4]{Fri03}.

\begin{lem}
\label{lem:frieduke}
Assume that $P$ has positive coefficients. Then
$$
	\limsup_{p\to\infty} |\nu_1(x^p)|^{\frac{1}{p}}
	\le
	\|P(\boldsymbol{u},\boldsymbol{u}^*)\|.
$$
\end{lem}

To set up the proof, let us fix a noncommutative polynomial $P$ of degree 
$d$ with positive coefficients.
By homogeneity, we may assume without loss of generality that the
coefficients sum to one. It will be convenient to express
$$
	P(\boldsymbol{u},\boldsymbol{u}^*) =
	\sum_{i_1,\ldots,i_d=0}^{2r}
	a_{i_1,\ldots,i_d}\,u_{i_1}\cdots u_{i_d} =
	\mathbf{E}[u_{I_1}\cdots u_{I_d}],
$$
where $\boldsymbol{I}=(I_1,\ldots,I_d)$ are random variables such that
$\mathbf{P}[\boldsymbol{I}=(i_1,\ldots,i_d)]=a_{i_1,\ldots,i_d}$.
Now let
$(I_{sd+1},\ldots,I_{(s+1)d})$ be independent copies of 
$\boldsymbol{I}$
for $s\in\mathbb{N}$, so that we have 
$P(\boldsymbol{u},\boldsymbol{u}^*)^p
= \mathbf{E}[u_{I_1}\cdots u_{I_{pd}}]$. Then we can apply
Corollary \ref{cor:divisor} to compute
$$
	\nu_1(x^p) =
	-\tau\big(P(\boldsymbol{u},\boldsymbol{u}^*)^p\big) +
	\sum_{k=2}^{pd} (\omega(k)-1)
	\sum_{v\in\mathbf{F}_r^{\rm np}}
	\mathbf{E}\big[1_{g_{I_1}\cdots g_{I_{pd}}=v^k}\big],
$$
where $\mathbf{F}_r^{\rm np}$ denotes the set of non-powers in
$\mathbf{F}_r$.\footnote{Here we used that
$\mu_1(\boldsymbol{w})=-1$ if $g_{w_1}\cdots g_{w_q}=e$, since
$\ntr \id|_{1^\perp} = \frac{N-1}{N}$.\label{foot:muid}}

\begin{proof}[Proof of Lemma \ref{lem:frieduke}]
We would like to argue that if a word $g_{i_1}\cdots g_{i_{pd}}=v^k$,
it must be a 
concatenation of $k$ words that reduce to $v$. This is only true, however,
if $v$ is cyclically reduced: otherwise the last letters of $v$ may
cancel the first letters of the
next repetition of $v$, and the cancelled letters need not appear
in our word. The correct version of this statement is that there exist
$g,w\in\mathbf{F}_d$ 
with $v=gwg^{-1}$ (where $w$ is the cyclic reduction of $v$) so that
every word that reduces to $v^k$ is a concatenation of
words that reduce to $g,w,w,w^{k-2},g^{-1}$. Thus
\begin{multline*}
	\sum_{v\in\mathbf{F}_r^{\rm np}}
	1_{g_{i_1}\cdots g_{i_{pd}}=v^k} \le
	\sum_{g,w\in\mathbf{F}_r}
	\sum_{0\le t_1\le \cdots\le t_4\le pd}
	1_{g_{i_1}\cdots g_{i_{t_1}} = g}\,
	1_{g_{i_{t_1+1}}\cdots g_{i_{t_2}} = w} \times\mbox{} \\
	1_{g_{i_{t_2+1}}\cdots g_{i_{t_3}} = w}\,
	1_{g_{i_{t_3+1}}\cdots g_{i_{t_4}} = w^{k-2}}\,
	1_{g_{i_{t_4+1}}\cdots g_{i_{pd}} = g^{-1}}.
\end{multline*}
To relate this bound to the spectral properties of 
$P(\boldsymbol{u},\boldsymbol{u}^*)$, we make the simple observation
that the indicators above can be expressed as matrix elements
$$
	1_{g_{w_1}\cdots g_{w_q}=v} = \langle \delta_v,
	u_{w_1}\cdots u_{w_q}\,\delta_e\rangle.
$$
If we substitute the formula in the above inequality, and then
take the expectation with respect to each independent block of variables
$(I_{sd+1},\ldots,I_{(s+1)d})$ that lies entirely inside one of the
matrix elements, we obtain
\begin{multline*}
	\sum_{v\in\mathbf{F}_r^{\rm np}}
	\mathbf{E}\big[1_{g_{I_1}\cdots g_{I_{pd}}=v^k}\big]
	\le
	\sum_{g,w\in\mathbf{F}_r}
	\sum_{0\le t_1\le \cdots\le t_4\le pd}
	\mathbf{E}\big[
	\langle \delta_g,X_{1,\boldsymbol{t}}\,\delta_e\rangle\,
	\langle \delta_w,X_{2,\boldsymbol{t}}\,\delta_e\rangle 
	\times\mbox{} \\
	\langle\delta_w,X_{3,\boldsymbol{t}}\,\delta_e\rangle\,
	\langle\delta_{w^{k-2}},X_{4,\boldsymbol{t}}\,\delta_e\rangle\,
	\langle\delta_{g^{-1}},X_{5,\boldsymbol{t}}\,\delta_e\rangle
	\big]
\end{multline*}
with
$$
	X_{j,\boldsymbol{t}} =
	u_{I_{t_{j-1}+1}}\cdots u_{I_{a_j}}
	P(\boldsymbol{u},\boldsymbol{u}^*)^{m_j}
	\,u_{I_{b_j+1}}\cdots u_{I_{t_j}},
$$
where $a_j=\min\{sd:s\in\mathbb{Z}_+,~t_{j-1}\le sd\}\wedge t_j$,
$b_j=\max\{sd:s\in\mathbb{Z}_+,~sd\le t_j\}\vee a_j$, $m_jd=b_j-a_j$,
and we write $t_0=0$ and $t_5=pd$ for simplicity. 

The crux of the proof is now to note that as
$$
	\sum_{v\in\mathbf{F}_r} |\langle \delta_v,X_{j,\boldsymbol{t}}\,
	\delta_e\rangle|^2 =
	\|X_{j,\boldsymbol{t}}\,
	\delta_e\|^2 \le \|P(\boldsymbol{u},\boldsymbol{u}^*)\|^{2m_j},
$$
it follows readily using Cauchy--Schwarz that
\begin{align*}
	\sum_{v\in\mathbf{F}_r^{\rm np}}
	\mathbf{E}\big[1_{g_{I_1}\cdots g_{I_{pd}}=v^k}\big]
	&\le
	\sum_{0\le t_1\le \cdots\le t_4\le pd}
	\|P(\boldsymbol{u},\boldsymbol{u}^*)\|^{m_1+\cdots+m_5}
	\\ &\le
	(pd+1)^4 \|P(\boldsymbol{u},\boldsymbol{u}^*)\|^{p+O(1)},
\end{align*}
since each $X_{j,\boldsymbol{t}}$ contains at most $2d=O(1)$ variables
other than $P(\boldsymbol{u},\boldsymbol{u}^*)^{m_j}$.
As $\sum_{k=2}^{pd}(\omega(k)-1) \le (pd)^2$ and
$|\tau(P(\boldsymbol{u},\boldsymbol{u}^*)^p)|\le
\|P(\boldsymbol{u},\boldsymbol{u}^*)\|^p$, the conclusion follows
directly from the expression for $\nu_1(x^p)$ stated before the proof.
\end{proof}

\begin{rem}
\label{rem:mdlspos}
The proof of Lemma \ref{lem:frieduke} relies on positivization: since all 
the terms in the proof are positive, we are able to obtain upper bounds by 
overcounting as in the first equation display of the proof. While this 
argument only applies in first instance to polynomials with positive 
coefficients, strong convergence for arbitrary polynomials then follows 
\emph{a posteriori} by Lemma \ref{lem:pos}.

It is also possible, however, to apply a variant of the positivization 
trick directly to $\nu_1$. This argument \cite[\S 6.2]{MdlS24} shows that 
the validity of Lemma \ref{lem:frieduke} for polynomials with positive 
coefficients already implies its validity for all self-adjoint polynomials 
(even with matrix coefficients), so that the polynomial method can be 
applied directly to general polynomials. The advantage of this approach is 
that it yields much stronger quantitative bounds than can be achieved by 
applying Lemma~\ref{lem:pos}. Since we have not emphasized the 
quantitative features of the polynomial method in our presentation, we do 
not develop this approach further here.
\end{rem}

\subsection{Discussion: on the role of cancellations}
\label{sec:discpoly}

When encountered for the first time, the simplicity of proofs by the 
polynomial method may have the appearance of a magic trick. An explanation 
for the success of the method is that it uncovers a genuinely new 
phenomenon that is not captured by classical methods of random matrix 
theory. Now that we have provided a complete proof of Theorem \ref{thm:bc} 
by the polynomial method, we aim to revisit the proof to highlight 
where this phenomenon arises. For simplicity, we place the following 
discussion in the context of random matrices $X^N$ with 
limiting operator $X_{\rm F}$; the reader may keep in mind 
$$
	X^N = P(\boldsymbol{U}^N,\boldsymbol{U}^{N*})|_{1^\perp},
	\qquad
	X_{\rm F} = P(\boldsymbol{u},\boldsymbol{u}^*)
$$
in the context of Theorem \ref{thm:bc} and its proof.

\subsubsection{The moment method}
\label{sec:momentmethod}

It is instructive to first recall the classical moment method that is 
traditionally used in random matrix theory. Let us take for granted
that $X^N$ converges \emph{weakly} to $X_{\rm F}$, so that
\begin{equation}
\label{eq:momentmethod}
	\mathbf{E}[\ntr {(X^N)^{2p}}]^{\frac{1}{2p}}
	= (1+o(1)) \tau(X_{\rm F}^{2p})^{\frac{1}{2p}} \le
	(1+o(1)) \|X_{\rm F}\|
\end{equation}
as $N\to\infty$ with $p$ fixed. The premise of the moment method is that 
\emph{if} it could be shown that this convergence remains valid when $p$ 
is allowed to grow with $N$ at rate $p\gg \log N$, then a strong 
convergence upper bound would follow: indeed, since
$\|X^N\|^{2p} \le \tr[(X^N)^{2p}] = N\,\ntr[(X^N)^{2p}]$, we could then
estimate 
$$
	\mathbf{E}\|X^N\| \le
	N^{\frac{1}{2p}}
	\mathbf{E}[\ntr {(X^N)^{2p}}]^{\frac{1}{2p}} \le
	(1+o(1))\|X_{\rm F}\|,
$$
where we used that $N^{\frac{1}{2p}}=1+o(1)$ for $p\gg\log N$.

There are two difficulties in implementing the above method. First, 
establishing \eqref{eq:momentmethod} for 
$p\gg\log N$ can be technically challenging and often requires delicate
combinatorial estimates. When $p$ is fixed, we can write an expansion
$$
	\mathbf{E}[\ntr {(X^N)^{2p}}] =
	\alpha_0(p) + \frac{\alpha_1(p)}{N} +
	\frac{\alpha_2(p)}{N^2} + \cdots
$$
(this is immediate, for example, from \eqref{lem:polyenc}) and 
establishing \eqref{eq:momentmethod} requires us to understand only the 
lowest-order term $\alpha_0(p)$. In contrast, when $p\gg\log N$ the 
coefficients $\alpha_k(p)$ themselves grow faster than polynomially in 
$N$, so that it is necessary to understand the terms in the expansion to 
all orders.

In the setting of Theorem \ref{thm:bc}, however, there is a more serious 
problem: \eqref{eq:momentmethod} is not just difficult to prove, but 
actually fails altogether.

\begin{example}
\label{ex:tangles}
Consider the permutation model of $2r$-regular random graphs as in Theorem
\ref{thm:friedman}, so that $X^N=A^N|_{1^\perp}$ where $A^N$
is the adjacency matrix. We claim that $\|X^N\|=2r$ with probability
at least $N^{-r}$. As $\|X_{\rm F}\|=2\sqrt{2r-1}$, this implies
$$
	\mathbf{E}[\tr {(X^N)^{2p}}] \ge
	N^{-r}(2r)^{2p} \ge
	N^{-r}
	\bigg(\frac{4}{3}\bigg)^p \|X_{\rm F}\|^{2p},
$$
contradicting the validity of \eqref{eq:momentmethod} for $p\gg\log N$.

To prove the claim, note that any given point of $[N]$ is simultenously a 
fixed point of the random permutations $U_1^N,\ldots,U_r^N$ with 
probability $N^{-r}$. Thus with probability at least $N^{-r}$, random 
graph has a vertex with $2r$ self-loops which is disconnected from the 
rest of the graph, so that $A^N$ has eigenvalue $2r$ with multiplicity at 
least two. The latter cleatly implies that $\|X^N\|=2r$.
\end{example}

Example \ref{ex:tangles} shows that the appearance of outliers in the 
spectrum with polynomially small probability $\sim N^{-c}$ presents a 
fundamental obstruction to the moment method. In random graph models, this 
situation arises due to the appearance of ``bad'' subgraphs, called 
\emph{tangles}. Previous proofs \cite{Fri08,Bor20,BC19} of optimal 
spectral gaps in this setting must overcome these difficulties by 
conditioning on the absence of tangles, which significantly complicates 
the analysis and has made it difficult to adapt these methods to more 
challenging models.\footnote{A notable exception being the work of 
Anantharaman and Monk on random surfaces \cite{AM25i,AM25ii}.}

\subsubsection{A new phenomenon}

The polynomial method is essentially based on the same input as the moment 
method: we consider the spectral statistics
$$
	\mathbf{E}[\ntr h(X^N)] =
	\text{linear combination of }\mathbf{E}[\ntr {(X^N)^p}] 
	\text{ for }p\le q,
$$
where $h$ is any real polynomial of degree $q$, and aim to compare 
these with the spectral statistics of $X_{\rm F}$. Since we have shown in 
Example \ref{ex:tangles} that each moment can be larger than its limiting 
value by a factor exponential in the degree, that is, $\mathbf{E}[\ntr 
{(X^N)^{2p}}]\ge e^{Cp}\tau((X_{\rm F})^{2p})$ for $p\gg\log N$, it seems 
inevitable that $\mathbf{E}[\ntr h(X^N)]$ must be poorly approximated by 
$\tau(h(X_{\rm F}))$ for high degree polynomials $h$. The surprising 
feature of the polynomial method is that it defies this expectation: for 
example, a trivial modification of the proof of Corollary \ref{cor:master} 
yields the bound
\begin{equation}
\label{eq:polymeth}
	\big| \mathbf{E}[\ntr h(X^N)] -
	\tau(h(X_{\rm F})) \big|
	\le \frac{Cq^4}{N} \|h\|_{L^\infty[-K,K]}
\end{equation}
which depends only polynomially on the degree $q$.

There is of course no contradiction between these observations: if we 
choose $h(x)=x^p$ in \eqref{eq:polymeth}, then $\|h\|_{L^\infty[-K,K]}= 
K^p\ge e^{Cp}\|X_{\rm F}\|^p$ and we recover the exponential dependence on 
degree that was observed in Example \ref{ex:tangles}. On the other hand, 
\eqref{eq:polymeth} shows that the dependence on the degree becomes 
polynomial when $h$ is uniformly bounded on the interval 
$[-K,K]$. Thus the polynomial method reveals an unexpected cancellation 
phenomenon that happens when the moments are combined to form bounded test 
functions $h$.

The idea that classical tools from the analytic theory of polynomials, 
such as the Markov inequality, make it possible to capture such 
cancellations lies at the heart of the polynomial method. These 
cancellations would be very difficult to realize by a direct combinatorial 
analysis of the moments. The reason that this phenomenon greatly 
simplifies proofs of strong convergence is twofold. First, it only 
requires us to understand the $\frac{1}{N}$-expansion of the moments to 
first order, rather than to every order as would be required by the moment 
method. Second, this eliminates the need to deal with tangles, since 
tangles do not appear in the first-order term in the expansion. (The 
tangles are however visible in the higher order terms, which gives rise to 
the large deviations behavior in Figure \ref{fig:staircase}.)

\begin{rem} 
We have contrasted the polynomial method with the moment method since
both rely only on the ability to compute moments $\mathbf{E}[\ntr {(X^N)^p}]$.
Beside the moment method, another classical method of random 
matrix theory is based on resolvent statistics such as
$\ntr {(z-X^N)^{-1}}$.
This approach was used by Haagerup--Thorbj{\o}rnsen \cite{HT05}
and Schultz \cite{Sch05} to establish strong convergence for Gaussian 
ensembles, where strong analytic tools are available.
It is unclear, however, how such quantities can be 
computed or analyzed in the context of discrete models
as in Theorem \ref{thm:bc}. Nonetheless, let us note that the 
recent works \cite{HY24,HMY24,HMY25} 
have successfully used such an approach in the setting of random regular 
graphs.
\end{rem}

\section{Intrinsic freeness}
\label{sec:mconc}

The aim of this section is to explain the origin of the intrinsic freeness 
phenomenon that was introduced in section \ref{sec:intriintro}. Since 
Theorem \ref{thm:intrfree} requires a number of technical ingredients 
whose details do not in themselves shed significant light on the 
underlying phenomenon, we defer to \cite{BBV23,BCSV23} for a complete 
proof. Instead, we aim to give an informal discussion of the key ideas 
behind the proof: in particular, we aim to explain the underlying 
mechanism.

Before we can do so, however, we must first describe the limiting object 
$X_{\rm free}$ and explain why it is useful in practice, which we will do 
in section \ref{sec:semicircle}. We subsequently sketch some key ideas 
behind the proof of Theorem \ref{thm:intrfree} in section 
\ref{sec:pfintr}.

\subsection{The free model}
\label{sec:semicircle}

To work with Gaussian random matrices, we must 
recall how to compute moments of independent standard Gaussians 
$\boldsymbol{g}=(g_1,\ldots,g_r)$: given any $k_1,\ldots,k_n\in [r]$,
the Wick formula \cite[Theorem 22.3]{NS06} states that
$$
	\mathbf{E}[g_{k_1}\cdots g_{k_n}] =
	\sum_{\pi\in\mathrm{P}_2[n]} \prod_{\{i,j\}\in\pi}
	1_{k_i=k_j},
$$
where $\mathrm{P}_2[n]$ denotes the set of pairings of $[n]$
(that is, partitions into blocks of size two). This 
classical result is easily proved by induction on $n$ using integration by 
parts. A convenient way to rewrite the Wick formula is to introduce for 
every $\pi\in\mathrm{P}_2[n]$ and $j\in[n]$ random variables 
$\boldsymbol{g}^{j|\pi}=(g_1^{j|\pi},\ldots,g_r^{j|\pi})$ with the same 
law as $\boldsymbol{g}$ so 
that $\boldsymbol{g}^{j|\pi}=\boldsymbol{g}^{l|\pi}$ for $\{j,l\}\in\pi$, 
and $\boldsymbol{g}^{j|\pi},\boldsymbol{g}^{l|\pi}$ 
are independent otherwise. Then
$$
	\mathbf{E}[g_{k_1}\cdots g_{k_n}] = \sum_{\pi\in\mathrm{P}_2[n]}
	\mathbf{E}\big[g_{k_1}^{1|\pi}\cdots g_{k_n}^{n|\pi}\big],
$$
as the expectation in the sum factors as
$\prod_{\{i,j\}\in\pi}E[g_{k_i}g_{k_j}]$.

What happens if we replace the scalar Gaussians
$\boldsymbol{g}=(g_1,\ldots,g_r)$ by independent GUE matrices 
$\boldsymbol{G}^N=(G_1^N,\ldots,G_r^N)$?
To explain this, we need the following notion: a pairing 
$\pi$ has a \emph{crossing} if there exist 
pairs $\{i,j\},\{l,m\}\in\pi$ so that $i<l<j<m$. If we represent $\pi$ by 
drawing each element of $[n]$ as a vertex on a line, and drawing a 
semicircular arc between the vertices in each pair $\{i,j\}\in\pi$, the 
pairing has a crossing precisely when two of the arcs cross; see Figure 
\ref{fig:crossing}. 
\begin{figure}
\centering
\begin{tikzpicture}
\node[left] (0,0) {$\pi_1 = ~$};

\draw[thick] (0,0) -- (3.5,0);
\foreach \i in {1,...,8}
{
        \draw[fill=black] (0.5*\i-0.5,0) circle (0.05);
}
\begin{scope}
    \clip (-0.1,0) rectangle (3.6,1.4);
    \draw[thick] (1.75,-0.55) circle (1.85);
    \draw[thick] (0.75,0) circle (0.25);
    \draw[thick] (2.25,0) circle (0.75);
    \draw[thick] (2.25,0) circle (0.25);
\end{scope}

\begin{scope}[xshift=6cm]
\node[left] (0,0) {$\pi_2 = ~$};

\draw[thick] (0,0) -- (3.5,0);
\foreach \i in {1,...,8}
{
        \draw[fill=black] (0.5*\i-0.5,0) circle (0.05);
}
\begin{scope}
    \clip (-0.1,0) rectangle (3.6,1.1);
    \draw[thick] (1.5,-.68) circle (1.65);
    \draw[thick] (0.75,0) circle (0.25);
    \draw[thick] (2.5,0) circle (1);
    \draw[thick] (2.25,0) circle (0.25);
\end{scope}
\end{scope}
\end{tikzpicture}
\caption{Illustration of a noncrossing pairing 
$\pi_1$ and a crossing pairing $\pi_2$.\label{fig:crossing}}
\end{figure}
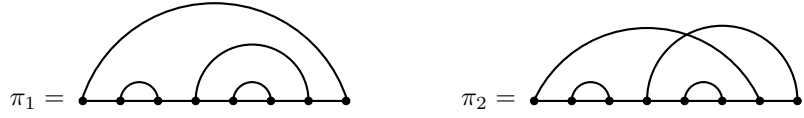

\begin{lem}
\label{lem:voic}
We have
$$
	\lim_{N\to\infty}
	\mathbf{E}\big[{\ntr G^N_{k_1}\cdots G^N_{k_n}}\big] =
	\sum_{\pi\in\mathrm{NC}_2[n]} \prod_{\{i,j\}\in\pi}
	1_{k_i=k_j},
$$
where $\mathrm{NC}_2[n]$ 
denotes the set of \emph{noncrossing} pairings.
\end{lem}

\begin{proof}
Define $\boldsymbol{G}^{N,j|\pi}=(G^{N,j|\pi}_1,\ldots,G^{N,j|\pi}_r)$ 
analogously to $\boldsymbol{g}^{j|\pi}$ above. Then 
$$
	\mathbf{E}\big[{\ntr G^N_{k_1}\cdots G^N_{k_n}}\big] = 
	\sum_{\pi\in\mathrm{P}_2[n]} \mathbf{E}\big[{\ntr 
	G^{N,1|\pi}_{k_1}\cdots G^{N,n|\pi}_{k_n}}\big]
$$
by the Wick formula.
Consider first a \emph{noncrossing} pairing $\pi$. Since pairs cannot
cross, there must be an adjacent pair 
$\{i,i+1\}\in\pi$, and if this pair is removed we obtain a
noncrossing pairing of $[n]\backslash\{i,i+1\}$. 
As
$\mathbf{E}[G^N_{k_i}G^N_{k_j}]=1_{k_i=k_j}\id$,\footnote{\label{foot:cross}%
This follows from a simple explicit computation using the following
characterization of GUE matrices: $G_i^N$ is a self-adjoint matrix whose
entries above the diagonal are i.i.d.\ complex Gaussians and
entries on the diagonal are i.i.d.\ real Gaussians with mean zero and
variance $\frac{1}{N}$.}
we obtain
$$
	\mathbf{E}\big[{\ntr
        G^{N,1|\pi}_{k_1}\cdots G^{N,n|\pi}_{k_n}}\big] =
	\prod_{\{i,j\}\in\pi} 1_{k_i=k_j}
$$
by repeatedly taking the expectation with respect to an adjacent pair.

On the other hand, if $\boldsymbol{\tilde G}^N$ is an independent copy of 
$\boldsymbol{G}^N$, we
can compute\cref{foot:cross}
\begin{equation}
\label{eq:guecross}
	\mathbf{E}\big[
	G_{k_i}^N \,A\, \tilde G_{k_l}^N \,B\, G_{k_j}^N \,C\, 
	\tilde G_{k_m}^N \big]
	=
	\frac{1}{N^2}\, CBA \, 1_{k_i=k_j}\, 1_{k_l=k_m}
\end{equation}
for any matrices $A,B,C$ that are independent of 
$\boldsymbol{G}^N,\boldsymbol{\tilde G}^N$.
Thus
$$
	\mathbf{E}\big[{\ntr
        G^{N,1|\pi}_{k_1}\cdots G^{N,n|\pi}_{k_n}}\big] =
	o(1)
$$
as $N\to\infty$ whenever $\pi$ is a \emph{crossing} pairing.
\end{proof}

In view of Lemma \ref{lem:voic}, the significance of the following 
definition of the limiting object associated to independent GUE matrices 
is self-evident.

\begin{defn}[Free semicircular family] 
\label{defn:semicircle}
A family 
$\boldsymbol{s}=(s_1,\ldots,s_r)\in\mathcal{A}$ of self-adjoint elements 
of a $C^*$-probability space $(\mathcal{A},\tau)$ such that
$$
	\tau(s_{k_1}\cdots s_{k_n}) =
	\sum_{\pi\in\mathrm{NC}_2[n]} \prod_{\{i,j\}\in\pi}
	1_{k_i=k_j}
$$
for all $n\in\mathbb{N}$ and $k_1,\ldots,k_n\in [r]$
is called a \emph{free semicircular family}.
\end{defn}

Free semicircular families can be constructed in various ways, 
guaranteeing their existence; see, e.g.\ \cite[pp.\ 102--108]{NS06}.
Lemma \ref{lem:voic} states that a family 
$\boldsymbol{G}^N$ of independent GUE matrices 
converges \emph{weakly} to a free semicircular family 
$\boldsymbol{s}$. 

\begin{rem}
The variables $s_i$ are called ``semicircular'' because their moments
$\tau(s_i^p) = |\mathrm{NC}_2[p]| = \int_{-2}^2 x^p\cdot \frac{1}{2\pi}
\sqrt{4-x^2}\,dx$ are the moments of the semicircle distribution.
Thus Lemma \ref{lem:voic} recovers the classical fact
that the empirical spectral distribution of a GUE matrix converges to the
semicircle distribution.
\iffalse
The definition of a free semicircular family is strongly reminiscent of
the Wick formula for standard Gaussians, and this is no coincidence: free
semicircular families play a parallel role in Voiculescu's free 
probability theory to Gaussians in classical probability theory.
We refer to \cite{NS06} for an excellent introduction to this theory.
\fi
\end{rem}

The \emph{intrinsic freeness principle} states that both the 
spectral distribution and spectral edges of a $D\times D$ self-adjoint 
Gaussian random matrix
$$
        X = A_0 + \sum_{i=1}^r A_i g_i
$$
are captured in a surprisingly general setting by those of the
operator
$$
        X_{\rm free} =
        A_0\otimes\id + \sum_{i=1}^r A_i\otimes s_i.
$$
This is unexpected, as this phenomenon does not arise as a limit of GUE 
type matrices which motivated the definition of $X_{\rm free}$ and thus 
it is not clear where the free behavior of $X$ comes 
from. The latter will be explained
in section \ref{sec:pfintr}.

Beside its fundamental interest, this principle is of considerable 
practical utility because the spectral statistics of the operator $X_{\rm 
free}$ can be explicitly computed by means of closed form equations, as we 
will presently explain. Let us first show how to compute the spectral 
distribution $\mu_{X_{\rm free}}$.

\begin{lem}[Matrix Dyson equation]
\label{lem:dyson}
For $z\in\mathbb{C}$ with $\mathrm{Im}\,z>0$, we denote by
$$
	G(z) = (\mathrm{id}\otimes\tau)\big[
	(z\id-X_{\rm free})^{-1}
	\big]
$$
the matrix Green's function of $X_{\rm free}$. Then $G(z)$ satisfies
the \emph{matrix Dyson equation}
$$
	G(z)^{-1} + A_0 + \sum_{i=1}^r A_i G(z) A_i = z\id,
$$
and
$
	\int f\,d\mu_{X_{\rm free}} = -\frac{1}{\pi}
	\lim_{\varepsilon\downarrow 0}
	\int f(x)\,
	\mathrm{Im}[\ntr G(x+i\varepsilon)] \, dx
$
for all $f\in C_b(\mathbb{R})$.
\end{lem}

\begin{proof}
We can construct all
$\pi\in\mathrm{NC}_2[n]$ as follows: first choose the pair $\{1,l\}$ 
containing the first point; and then pair the remaining points by choosing 
any noncrossing pairings
of the sets $\{2,\ldots,l-1\}$ and $\{l+1,\ldots,n\}$. Thus
$$
	\tau(s_{k_1}\cdots s_{k_n}) =
	\sum_{l=2}^n 1_{k_1=k_l}\, \tau(s_{k_2}\cdots s_{k_{l-1}})\,
	\tau(s_{k_{l+1}}\cdots s_{k_n})
$$
for $k_1,\ldots,k_n\in[r]$
by the definition of a free semicircular
family. In the following, it will be convenient to allow also 
$k_i=0$, where we define $s_0=\id$. In this case, the identity clearly
remains valid provided that $k_1>0$.

Now define the matrix moments
$$
	M_n = (\mathrm{id}\otimes\tau)[ X_{\rm free}^n] =
	\sum_{\boldsymbol{k}\in\{0,\ldots,r\}^n} A_{k_1}\cdots A_{k_n}
	\,\tau(s_{k_1}\cdots s_{k_n}).
$$
Applying the above identity yields for $n\ge 2$ the recursion (with 
$M_0=\id$, $M_1=A_0$)
$$
	M_n = A_0 M_{n-1} +
	\sum_{l=2}^{n}
	\sum_{k=1}^r A_k M_{l-2} A_k M_{n-l}.
$$
When $|z|$ is sufficiently large, we can write 
$G(z) = \sum_{n=0}^\infty z^{-n-1} M_n$, and the matrix Dyson equation
follows readily from the recursion for $M_n$. The equation remains valid
for all $z\in\mathbb{C}$ with $\mathrm{Im}\,z>0$ by analytic continuation.

The final claim follows as 
$-\frac{1}{\pi}\mathrm{Im}\,(x+i\varepsilon)^{-1} =
\frac{1}{\pi}\frac{\varepsilon}{x^2+\varepsilon^2}=
\rho_\varepsilon(x)$ 
is the
density of the Cauchy 
distribution with scale $\varepsilon$, so that
$-\frac{1}{\pi}\mathrm{Im}[\ntr G(x+i\varepsilon)]$ is the 
density of the convolution $\mu_{X_{\rm free}}*\rho_\varepsilon$
which converges weakly to $\mu_{X_{\rm free}}$ as $\varepsilon\to 0$.
\end{proof}

Lemma \ref{lem:dyson} shows that the spectral distribution of $X_{\rm 
free}$ can be computed by solving a system of quadratic equations for the 
entries of $G(z)$. While these equations usually do not have a closed
form solution, they are well behaved and are amenable to analysis and 
numerical computation \cite{HRS07,AEK20}.

The spectral edges of $X_{\rm free}$ can in principle be obtained from its 
spectral distribution (cf.\ Lemma \ref{lem:faith}). However, the following 
formula of Lehner \cite{Leh99}, which we state without proof,\footnote{%
	The difficulty is to upper bound $\lambda_{\rm
        max}(X_{\rm free})$: as
	$M=G(z)>0$ for any
	$z>\lambda_{\rm max}(X_{\rm free})$,
	Lemma~\ref{lem:dyson}
	shows that $\lambda_{\rm 
        max}(X_{\rm free})$ is lower bounded by the right-hand
	side of Lehner's formula.
} 
provides an 
often more powerful tool: it expresses the outer edges of the spectrum
of $X_{\rm free}$ in 
terms of a variational principle.

\begin{thm}[Lehner] 
\label{thm:lehner}
We have
$$
	\lambda_{\rm max}(X_{\rm free}) =
	\inf_{M>0} \lambda_{\rm max}\Bigg(
	M^{-1} + A_0 + \sum_{i=1}^r A_iMA_i
	\Bigg),
$$
where we denote
$\lambda_{\rm max}(X) = \sup\spc(X)$ for any self-adjoint operator $X$.
\end{thm}

Various applications of this formula are 
illustrated in \cite{BCSV23}. On the other hand, in applications where the 
exact location of the edge is not important, the following simple bounds 
often suffice and are easy to use:
$$
	\|A_0\| \vee \Bigg\|\sum_{i=1}^r A_i^2\Bigg\|^{1/2} 
	\le
	\|X_{\rm free}\| 
	\le \|A_0\| + 2\Bigg\|\sum_{i=1}^r A_i^2\Bigg\|^{1/2}.
$$
These bounds admit a simple direct proof \cite[p.\ 208]{Pis03}.

By connecting the spectral statistics of a random matrix $X$ to those of 
$X_{\rm free}$, the intrinsic freeness principle makes it possible to 
understand the spectra of complicated random matrix models that would be 
difficult to analyze directly. One may view the operator $X_{\rm free}$ as 
a ``platonic ideal'': a perfect object which captures the essence of the 
random matrices $X$ that exist in the real world.

\subsection{Interpolation and crossings}
\label{sec:pfintr}

We now aim to explain how the intrinsic freeness principle actually 
arises. In this section, we will roughly sketch the most basic ideas 
behind the proof of Theorem \ref{thm:intrfree}.

The most natural way to interpolate between $X$ and $X_{\rm free}$ is to 
define
$$
        X^N = A_0\otimes\id + \sum_{i=1}^r A_i\otimes G_i^N 
$$
as in section \ref{sec:intriintro}, where $G_1^N,\ldots,G_r^N$ are 
independent GUE matrices. Then $X^N=X$ when $N=1$, while $X^N\to X_{\rm 
free}$ weakly as $N\to\infty$ by Lemma \ref{lem:voic}. One may thus be 
tempted to approach intrinsic freeness by applying the polynomial method 
to $X^N$. The problem with this approach, however, is that the small 
parameter that arises in the polynomial method is not $\tilde v(X)$ as in 
Theorem \ref{thm:intrfree}, but rather $\frac{1}{N}$. This is useless for 
understanding what happens when $N=1$.

The basic issue here is that unlike classical strong convergence, 
the intrinsic freeness phenomenon is truly nonasymptotic in nature: it 
aims to capture an intrinsic property of $X$ that causes it to behave as 
the corresponding free model. Thus we cannot hope to deduce such a property 
from the asymptotic behavior of the model $X^N$ alone; the proof must 
explicitly explain where intrinsic freeness comes from, and why it is 
quantified by a parameter such as $\tilde v(X)$.

\subsubsection{The interpolation method}

Rather than using $X^N$ as an interpolating family,
the proof of intrinsic freeness is based on a \emph{continuous} 
interpolating family parametrized by $q\in[0,1]$.
Roughly speaking, we would like to define
$$
	\text{`` }
	X_q = \sqrt{q}\, X + \sqrt{1-q}\, X_{\rm free} \text{ ''},
$$
and apply the fundamental theorem of calculus as 
explained in section~\ref{sec:introinterpol} to bound the discrepancy 
between the spectral statistics of $X=X_1$ and $X_{\rm free}=X_0$.
The obvious problem with the above definition is that it makes no sense:
$X$ is random matrix and $X_{\rm free}$ is a deterministic operator, which 
live in different spaces. To implement this program, we will construct
proxies for $X$ and $X_{\rm free}$ that are high-dimensional random 
matrices of the same dimension.

To this end, we proceed as follows. Let $G_1^N,\ldots,G_r^N$ be 
independent GUE matrices, and let $D_1^N,\ldots,D_r^N$ be independent
diagonal matrices with i.i.d.\ standard Gaussian entries on the diagonal.
Then we define the $DN\times DN$ random matrices
$$
	X_q^N =
	A_0\otimes\id + 
	\sum_{i=1}^r A_i\otimes\big(\sqrt{q}\,D_i^N+\sqrt{1-q}\,G_i^N\big).
$$
The significance of this definition is that
$$
	\mathbf{E}[\ntr {(X_1^N)^p}] = 
	\mathbf{E}[\ntr X^p],\qquad\quad
	\lim_{N\to\infty}\mathbf{E}[\ntr {(X_0^N)^p}] =
	({\ntr}\otimes\tau)\big(X_{\rm free}^p\big)
$$
for all $p\in\mathbb{N}$; the first identity follows as $X_1^N$ is a 
block-diagonal matrix with i.i.d.\ copies of $X$ on the diagonal, while 
the second follows by Lemma \ref{lem:voic} as $X_0^N=X^N$.
Thus $X_q^N$ does indeed interpolate between $X$ and $X_{\rm free}$ in the 
limit as $N\to\infty$. (We emphasize that we now view $q$ as the 
interpolation parameter, as opposed to the interpolation parameter 
$\frac{1}{N}$ in the polynomial method.)

Now that we have defined a suitable interpolation, we aim to compute the 
rate of change $\frac{d}{dq}\mathbf{E}[\ntr h(X_q^N)]$ of spectral 
statistics along the interpolation: if it is small, then the spectral 
statistics of $X$ and $X_{\rm free}$ must be nearly the same. For 
simplicity, we will illustrate the method using moments $h(x)=x^{2p}$, 
which suffices to capture the operator norm by the moment method.\footnote{% 
To achieve Theorem 
\ref{thm:intrfree} in its full strength, one uses instead spectral 
statistics of the form $h(x)=|z-x|^{-2p}$ for $z\in\mathbb{C}$, 
$\mathrm{Im}\,z>0$. The computations involved are however very similar.} 
We state the resulting expression informally; the computation 
is somewhat tedious (see the proof of \cite[Lemma 5.4]{BBV23}) but uses 
only standard tools of Gaussian analysis.

\begin{lem}[Informal statement]
\label{lem:interpolcross}
For any $p\in\mathbb{N}$, we have
\begin{multline*}
	\frac{d}{dq}\mathbf{E}[\ntr {(X_q^N)^{2p}}] =
	\text{sum of terms of the form}\\
	\mathbf{E}\big[{\ntr 
	H_a^N \, (X_q^N)^{m_1} \,
	\tilde H_b^N \, (X_q^N)^{m_2} \,
	H_a^N \, (\tilde X_q^N)^{m_3} \,
	\tilde H_b^N \, (\tilde X_q^N)^{m_4}
	}\big]
	\\
	\phantom{\sum}
	\text{with }m_1+m_2+m_3+m_4=2p-4\text{ and }a,b\in\{0,1\},
\end{multline*}
where
$\tilde X_q^N$ is a suitably constructed (dependent) copy of $X_q^N$ and
$H_q^N,\tilde H_q^N$ are independent copies of
$X_q^N-\mathbf{E}[X_q^N]$.
\end{lem}

From a conceptual perspective, the expression of Lemma 
\ref{lem:interpolcross} should not be unexpected. Indeed, the explicit 
formulas for $\mathbf{E}[\ntr X^{2p}]$ and $({\ntr}\otimes\tau)(X_{\rm 
free}^{2p})$ that 
arise from the Wick formula and Definition \ref{defn:semicircle}, 
respectively, differ only in that the former has a sum over all pairings 
while the latter sums only over noncrossing pairings. Thus the difference 
between these two quantities is a sum over all pairings that contain at 
least one crossing. The point of the interpolation method, however, is 
that by changing $q$ infinitesimally we can isolate the effect of a 
\emph{single} crossing---this is precisely what Lemma 
\ref{lem:interpolcross} shows. This key feature of the interpolation 
method is crucial for accessing the edges of the spectrum (see section
\ref{sec:whyinterpol}).

\subsubsection{The crossing inequality}

By Lemma \ref{lem:interpolcross}, it remains to control the 
effect of a single crossing. We can now finally explain the significance 
of the mysterious parameter $\tilde v(X)$: this parameter controls the 
contribution of crossings. The following result is a 
combination of \cite[Lemma 4.5 and Proposition 4.6]{BBV23}.

\begin{lem}[Crossing inequality]
\label{lem:crossineq}
Let $H,\tilde H$ be any independent and centered self-adjoint random 
matrices, and $1\le p_1,\ldots,p_4\le\infty$ with 
$\frac{1}{p_1}+\cdots+\frac{1}{p_4}=1$. Then
$$
	\big|\mathbf{E}\big[{\ntr
	H\,M_1\,\tilde H\,M_2\, H\,M_3\,\tilde H\,M_4
	}\big]\big| \le
	\tilde v(H)^2\, \tilde v(\tilde H)^2
	\prod_{i=1}^4 \big({\ntr |M_i|^{p_i}}\big)^{\frac{1}{p_i}}
$$
for any matrices $M_1,\ldots,M_4$ that are independent of $H,\tilde H$.
\end{lem}

Rather than reproduce the details of the proof of this inequality here, we 
aim to explain the intuition behind the proof.

\begin{proof}[Idea behind the proof of Lemma \ref{lem:crossineq}]
We first observe that it suffices by the Riesz--Thorin 
interpolation theorem \cite[p.\ 202]{BS88} to prove the theorem for the 
case $p_4=1$. Thus the proof reduces
to bounding the \emph{matrix alignment parameter}
$$
	w(H,\tilde H)^4 =
	\sup_{\|M_1\|,\|M_2\|,\|M_3\|\le 1}
	\big\|
	\mathbf{E}\big[
	H\,M_1\,\tilde H\,M_2\, H\,M_3\,\tilde H
	\big]
	\big\|
$$
by 
$$
	\tilde v(H)^2\,\tilde v(\tilde H)^2=
	\|\mathbf{E}[H^2]\|^{\frac{1}{2}}\,
	\|\mathrm{Cov}(H)\|^{\frac{1}{2}}\,
	\|\mathbf{E}[\tilde H^2]\|^{\frac{1}{2}}\,
	\|\mathrm{Cov}(\tilde H)\|^{\frac{1}{2}}.
$$
How can we do this? The basic intuition behind the proof is as follows.
Note first that if $G$ is a GUE matrix, then 
$\mathrm{Cov}(G)=\frac{1}{N}\id$. Thus
$$
	\mathrm{Cov}(H) \le
	N\,\|\mathrm{Cov}(H)\|\,\mathrm{Cov}(G)
$$
for any random matrix $H$. If it were to be the case that $w(H,\tilde H)$ 
is monotone as a function of $\mathrm{Cov}(H)$ and $\mathrm{Cov}(\tilde H)$,
then one could bound
$$
	w(H,\tilde H)^4 \mathop{\stackrel{?}{\le}}
	N^2\,\|\mathrm{Cov}(H)\|\,\|\mathrm{Cov}(\tilde H)\|\,
	w(G,\tilde G)
	=
	\|\mathrm{Cov}(H)\|\,\|\mathrm{Cov}(\tilde H)\|
$$
using that $w(G,\tilde G)=\frac{1}{N^2}$ for independent GUE matrices
$G,\tilde G$ by \eqref{eq:guecross}.

Unfortunately, $w(H,\tilde H)$ is \emph{not} monotone
as a function of $\mathrm{Cov}(H)$ and $\mathrm{Cov}(\tilde H)$, so 
the above reasoning does not apply directly. However, we can use a trick 
to rescue the argument. The key observation is that the parameter 
$w(H,\tilde H)$ can be ``symmetrized'' by applying the Cauchy--Schwarz 
inequality as is illustrated informally in Figure \ref{fig:cscross}.
This results in two symmetric terms---a single pair which is readily 
bounded by $\|\mathbf{E}[H^2]\|$, and a double crossing that is a positive
functional of (and hence monotone in) $\mathrm{Cov}(H)$. We can thus apply 
the above logic to the double crossing to replace $H$ by a GUE matrix,
which yields a factor $\|\mathrm{Cov}(H)\|$. The term that remains
can now be bounded using a similar argument.
\begin{figure}
\begin{center}
\begin{tikzpicture}[scale=.8]
\usetikzlibrary{patterns}

\draw[thick] (0,0) -- (0.25,0);
\draw[thick] (0.75,0) -- (1.25,0);
\draw[thick] (1.75,0) -- (2.25,0);
\draw[thick] (2.75,0) -- (3,0);

\draw[fill=black] (0,0) circle (0.05);
\draw[fill=black] (1,0) circle (0.05);
\draw[fill=black] (2,0) circle (0.05);
\draw[fill=black] (3,0) circle (0.05);

\begin{scope}
    \clip (-0.05,0) rectangle (3.3,1.1);
    \draw[thick] (1,0) circle (1);
    \draw[thick] (2,0) circle (1);
\end{scope}

\draw (0.5,0) node {$\scriptstyle M_1$};
\draw (1.5,0) node {$\scriptstyle M_2$};
\draw (2.5,0) node {$\scriptstyle M_3$};

\draw (1,0.95) node[anchor=south] {$\scriptstyle H$};
\draw (2,0.95) node[anchor=south] {$\scriptstyle \tilde H$};

\draw (3.8,0.5) node {$\le$};

\begin{scope}[shift={(5,0)}]

\draw[thick] (-0.2,-0.2) -- (-0.3,-0.2) -- (-0.3,1.6) -- (-0.2,1.6);
\draw[thick] (2.2,-0.2) -- (2.3,-0.2) -- (2.3,1.6) -- (2.2,1.6);
\draw (2.25,1.6) node[right] {$\frac{1}{2}$};

\draw[thick] (0,0) -- (0.25,0);
\draw[thick] (0.75,0) -- (1.25,0);
\draw[thick] (1.75,0) -- (2,0);

\draw[fill=black] (0,0) circle (0.05);
\draw[fill=black] (2,0) circle (0.05);

\begin{scope}
    \clip (-0.05,0) rectangle (3.3,1.1);
    \draw[thick] (1,0) circle (1);
\end{scope}

\draw (0.5,0) node {$\scriptstyle M_1$};
\draw (1.5,0) node {$\scriptstyle M_1^*$};

\draw (1,0.95) node[anchor=south] {$\scriptstyle H$};

\end{scope}

\begin{scope}[shift={(8.5,0)}]

\draw[thick] (-0.2,-0.2) -- (-0.3,-0.2) -- (-0.3,1.6) -- (-0.2,1.6);
\draw[thick] (5.2,-0.2) -- (5.3,-0.2) -- (5.3,1.6) -- (5.2,1.6);
\draw (5.25,1.6) node[right] {$\frac{1}{2}$};

\draw[thick] (0,0) -- (0.25,0);
\draw[thick] (0.75,0) -- (1.25,0);
\draw[thick] (1.75,0) -- (3.25,0);
\draw[thick] (3.75,0) -- (4.25,0);
\draw[thick] (4.75,0) -- (5,0);

\draw[fill=black] (0,0) circle (0.05);
\draw[fill=black] (1,0) circle (0.05);
\draw[fill=black] (2,0) circle (0.05);
\draw[fill=black] (3,0) circle (0.05);
\draw[fill=black] (4,0) circle (0.05);
\draw[fill=black] (5,0) circle (0.05);

\begin{scope}
    \clip (-0.05,0) rectangle (5.05,1.6);
    \draw[thick] (1,0) circle (1);
    \draw[thick] (4,0) circle (1);
    \draw[thick] (2.5,0) circle (1.5);
\end{scope}

\draw (0.5,0) node {$\scriptstyle M_3^*$};
\draw (1.5,0) node {$\scriptstyle M_2^*$};
\draw (3.5,0) node {$\scriptstyle M_2$};
\draw (4.5,0) node {$\scriptstyle M_3$};

\draw (1,0.95) node[anchor=south] {$\scriptstyle \tilde H$};
\draw (4,0.95) node[anchor=south] {$\scriptstyle \tilde H$};
\draw (2.5,1.5) node[anchor=south] {$\scriptstyle H$};

\end{scope}
\end{tikzpicture}
\end{center}
\caption{Cauchy--Schwarz argument in the proof of Lemma
\ref{lem:crossineq}.\label{fig:cscross}}
\end{figure}
\end{proof}

We can now sketch how all the above ingredients fit together. Combining 
Lemmas~\ref{lem:interpolcross}~and~\ref{lem:crossineq} with 
$p_i=\frac{2p-4}{m_i}$ yields an inequality of the form
$$
	\bigg|
	\frac{d}{dq}\mathbf{E}[\ntr {(X_q^N)^{2p}}]
	\bigg|
	\lesssim
	p^4 \,\tilde v(X)^4 \,
	\mathbf{E}[\ntr {(X_q^N)^{2p-4}}].
$$
Using $\mathbf{E}[\ntr {(X_q^N)^{2p-4}}] \le
\mathbf{E}[\ntr {(X_q^N)^{2p}}]^{1-\frac{2}{p}}$ by Jensen's inequality,
we obtain a differential inequality that can be integrated by a 
straightforward change of variables. This yields (after taking 
$N\to\infty$) the final inequality
$$
	\big|\mathbf{E}[\ntr X^{2p}]^{\frac{1}{2p}} -
	({\ntr}\otimes\tau)\big(X_{\rm free}^{2p}\big)^{\frac{1}{2p}}
	\big| \lesssim
	p^{\frac{3}{4}} \tilde v(X).
$$
This inequality captures the intrinsic freeness phenomenon for the
moments of $X$. Since the right-hand side depends only polynomially on the
degree $p$, however, one can apply the moment method as explained in 
section \ref{sec:momentmethod} to deduce also a bound on the operator norm 
$\|X\|$ by $\|X_{\rm free}\|$. In this manner, we achieve both
weak convergence and norm convergence of $X$ to $X_{\rm free}$ as
$\tilde v(X)\to 0$.

\begin{rem}
The matrix alignment parameter $w(H,\tilde H)$ that appears in the proof 
of Lemma \ref{lem:crossineq} (as well as the use of the Riesz--Thorin 
theorem in this context) was first introduced in the work of Tropp 
\cite{Tro18}, which predates the discovery of the intrinsic freeness 
principle. Let us briefly explain how it appears there.

The idea of \cite{Tro18} is to mimic the classical proof of the 
Schwinger-Dyson equation for GUE matrices, see, e.g., \cite[Chapter 
2]{Gui19}, in the context of a general Gaussian random matrix. Tropp 
observed that the error term that arises from this argument can be 
naturally bounded by $w(H,\tilde H)$, and that this parameter is small in 
some examples (e.g., for matrices with independent entries).

The reason this argument cannot give rise to generally applicable bounds 
is that it fails to capture the intrinsic freeness phenomenon. Indeed, 
the validity of the Schwinger-Dyson equation for GUE matrices requires
that $H,\tilde H$ 
\emph{themselves} behave as free semicircular variables;
this is not at all the case in 
general, as the spectral distribution $X_{\rm free}$ need not look 
anything like a semicircle. To ensure this is the case, \cite{Tro18} has
to
impose strong symmetry assumptions on $H$ that are close in 
spirit to the classical setting of Voiculescu's asymptotic freeness.\footnote{%
The paper \cite{Tro18} also develops another set of inequalities
that are applicable to general Gaussian matrices, but are suboptimal by a 
dimension-dependent multiplicative factor. We do not discuss these
inequalities as they are less closely connected to the topic of this 
survey.}

In contrast, intrinsic freeness captures a more subtle property of random 
matrices: $\tilde v(H)$ does not quantify whether $H$ itself behaves 
freely, but rather how sensitive the model $H=\sum_{i=1}^n A_ig_i$ is to 
whether the scalar variables $g_i$ are taken to be commutative or free. 
Consequently, when $\tilde v(H)$ is small, the variables $g_i$ can be 
replaced by their free counterparts $s_i$ (i.e., ``liberated'') with a 
negligible effect on the spectral statistics. This viewpoint paves the way 
to the development of the interpolation method which is key to subsequent 
developments.

The works of Haagerup--Thorbj{\o}rnsen \cite{HT05} and Tropp \cite{Tro18} 
may nonetheless be viewed as precursors to the intrinsic freeness 
principle, and provided the motivation for the development of the theory
that is described in this section.
\end{rem}

\subsection{Discussion: on the role of interpolation}
\label{sec:whyinterpol}

To conclude this section, we aim to explain why the interpolation method 
plays an essential role in the development of intrinsic freeness. For 
simplicity we will assume in this section that $A_0=0$, so that 
$X=\sum_{i=1}^r A_ig_i$ is a centered Gaussian matrix.

Since the moments of $X$ and $X_{\rm free}$ can be easily computed
explicitly, it is tempting to reason directly using the resulting 
expressions. More precisely, note that
\begin{alignat*}{2}
	\mathbf{E}[\ntr X^{2p}] =&
	\centermathcell{\sum_{\pi\in\mathrm{P}_2[2p]}}
	&&\mathbf{E}[\ntr X^{1|\pi}\cdots X^{2p|\pi}]
\intertext{by the Wick formula, while}
	({\ntr}\otimes\tau)\big(
	X_{\rm free}^{2p}\big) =&
	\centermathcell{\sum_{\pi\in\mathrm{NC}_2[2p]}}
	&&\mathbf{E}[\ntr X^{1|\pi}\cdots X^{2p|\pi}]
\end{alignat*}
by Definition \ref{defn:semicircle}. Thus clearly the difference between 
these two expressions involves only a sum over crossing pairings, and
we can control each term in the sum directly using Lemma 
\ref{lem:crossineq}. This elementary approach yields the inequality
$$
	\big|
	\mathbf{E}[\ntr X^{2p}] -
	({\ntr}\otimes\tau)\big(
        X_{\rm free}^{2p}\big)
	\big|
	\le (Cp)^p\, \tilde v(X)^4\, \mathbf{E}[\ntr X^{2p-4}],
$$
where we used that the number of crossing pairings of $[2p]$ is of order
$(Cp)^p$ for a universal constant $C$. In particular, we obtain
$$
	\big|
	\mathbf{E}[\ntr X^{2p}]^{\frac{1}{2p}} -
	({\ntr}\otimes\tau)\big(
        X_{\rm free}^{2p}\big)^{\frac{1}{2p}}
	\big|
	\lesssim \sqrt{p}\, \tilde v(X)^{\frac{2}{p}}\, 
	\big(\mathbf{E}[\ntr X^{2p}]^{\frac{1}{2p}}\big)^{1-\frac{2}{p}}.
$$
This inequality suffices to prove weak convergence of $X$ to $X_{\rm 
free}$ as $\tilde v(X)\to 0$, but is far too weak to provide access to the 
edges of the spectrum. To see why, recall from section 
\ref{sec:momentmethod} that to bound the norm of the $D\times D$ matrix 
$X$ by the moment method, we must control $\mathbf{E}[\ntr 
X^{2p}]^{\frac{1}{2p}}$ for $p\gg\log D$. However, even when $X$ is a GUE 
matrix we only have $\tilde v(X) = D^{-\frac{1}{4}}$, so that the error 
term $\sqrt{p}\, \tilde v(X)^{\frac{2}{p}}$ in the above inequality 
diverges as $D\to\infty$ when $p\gg\log D$.

The reason for the inefficiency of this approach is readily understood. 
What we used is that the difference between the moments of $X$ and $X_{\rm 
free}$ is a sum of terms with \emph{at least} one crossing. However, most 
pairings of $[2p]$ contain not just one crossing, but many (typically of 
order $p$) crossings at the same time. Unfortunately, 
Lemma~\ref{lem:crossineq} can only capture the effect of a single 
crossing: it cannot be iterated to obtain an improved bound in the 
presence of multiple crossings, as the H\"older type bound destroys the 
structure of the pairing. Thus we are forced to ignore the effect of 
multiple crossings, which results in a loss of information.

The key feature of the interpolation method that is captured by Lemma 
\ref{lem:interpolcross} is that when we move \emph{infinitesimally} from 
$X$ to $X_{\rm free}$, the change of the moments is controlled by a 
\emph{single} crossing rather than by many crossings at the same time. 
This is the reason why we are able to obtain an efficient bound using the 
somewhat crude crossing inequality provided by Lemma \ref{lem:crossineq}.

\begin{rem}
In the special case that $X$ is a GUE matrix, we obtained a 
much better result than Lemma \ref{lem:crossineq}: the crossing 
\emph{identity} \eqref{eq:guecross} captures the effect of a crossing 
exactly. This identity can be iterated in the presence of multiple 
crossings, which results in the genus expansion for GUE matrices (see, 
e.g., \cite[\S 1.7]{MS17}). This is a rather special feature of 
classical random matrix models, however, and we do not know of any method 
that can meaningfully capture the effect of multiple crossings in the 
setting of arbitrarily structured random matrices.
\end{rem}

\section{Applications}
\label{sec:appl}

In recent years, strong convergence has led to several striking 
applications to problems in different areas of mathematics, which has in 
turn motivated new developments surrounding the strong convergence 
phenomenon. The aim of this section is to briefly describe some of these 
applications. The discussion is necessarily at a high level, since 
the detailed background needed to understand each application is beyond 
the scope of this survey. Our primary aim is to give a hint as to why and 
how strong convergence enters in these different settings.

We will focus on applications where strong convergence enters in a 
non-obvious manner. In particular, we omit applications of the 
intrinsic freeness principle in applied mathematics, since it is 
generally applied in a direct manner to analyze complicated 
random matrices that arise in such applications.

\subsection{Random lifts of graphs}
\label{sec:lifts}

We begin by recalling some basic notions that can be found, for example, 
in \cite[\S 6]{HLW06}.

Let $G=(V,E)$ be a connected graph. A connected graph $G'=(V',E')$ is said 
to \emph{cover} $G$ if there is a surjective map $f:V'\to V$ that maps the 
local neighborhood of each vertex $v'$ in $G'$ bijectively to the local 
neighborhood of $f(v')$ in $G$ (the local neighborhood 
consists of the given vertex and the edges incident to it).\footnote{% 
This definition is slightly ambiguous if $G$ has a self-loop, which we 
gloss over for simplicity.}

Every connected graph $G$ has a \emph{universal cover} $\tilde G$ which
covers all other covers of $G$. Given a base vertex $v_0$ in $G$, one can 
construct $\tilde G$ by choosing its vertex set to be the set of all 
finite non-backtracking paths in $G$ starting at $v_0$, with two vertices 
being joined by an edge if one of the paths extends the other by one step; 
thus $\tilde G$ is a tree (the construction does not 
depend on the choice of $v_0$).

It is clear that if $G$ is any $d$-regular graph, then $\tilde G$ is the 
infinite $d$-regular tree. In particular, all $d$-regular graphs have the 
same universal cover. In this setting, we have an \emph{optimal spectral 
gap phenomenon}: for any sequence of $d$-regular graphs with diverging 
number of vertices, the maximum nontrivial eigenvalue is asymptotically 
lower bounded by the spectral radius of the universal cover (Lemma 
\ref{lem:alonboppana}), and this bound is attained by random 
$d$-regular graphs (Theorem \ref{thm:friedman}).

It is expected that the optimal spectral gap phenomenon is a very general 
one that is not specific to the setting of $d$-regular graphs. Progress 
in this direction was achieved only recently, however, and makes crucial 
use of strong convergence. In this section, we will describe such a 
phenomenon in the setting of \emph{non-regular} graphs; the setting of 
hyperbolic surfaces will be discussed in section \ref{sec:buser} below.

\subsubsection{Random lifts}
\label{sec:liftcons}

From the perspective of the lower bound, there is nothing particularly 
special about $d$-regular graphs beside that they all have the same 
universal cover. Indeed, for any sequence of graphs with diverging number 
of vertices that have the \emph{same} universal cover, the maximum 
nontrivial eigenvalue is asymptotically lower bounded by the spectral 
radius of the universal cover. This follows by a straightforward 
adaptation of Lemma \ref{lem:alonboppana}, cf.\ \cite[Theorem 6.6]{HLW06}.

What may be less obvious, however, is how to construct a model of random 
graphs that share the same universal cover beyond the regular setting. The 
natural way to think about this problem, which dates back to Friedman 
\cite{Fri03} (see also \cite{AL02}), is as follows. Fix any finite 
connected base graph $G$; we will then construct random graphs with an 
increasing number of vertices by choosing a sequence of random 
\emph{finite} covers of $G$. By construction, the universal cover of all 
these graphs coincides with the universal cover $\tilde G$ of the base 
graph.

To this end, let us explain how to construct finite covers of a finite 
connected graph $G=(V,E)$. Fix an arbitrary orientation $(x,y)$ for 
every edge $\{x,y\}\in E$, and denote by $E_{\mathrm{or}}$ the set of 
oriented edges. Fix also $N\in\mathbb{N}$ and a permutation $\sigma_e\in 
\mathbf{S}_N$ for each $e\in E_{\rm or}$. Then we can construct a graph 
$G^N=(V^N,E^N)$ with
$$
	V^N = V\times [N]
$$
and
$$
	E^N = \big\{ \{(x,i),(y,\sigma_e(i))\} : 
	e=(x,y)\in E_{\rm or},~i\in[N]\big\}.
$$
In other words, $G^N$ is obtained by taking $N$ copies of 
$G$, and scrambling the endpoints of the $N$ copies of each edge $e$ 
according the permutation $\sigma_e$ (see Figure~\ref{fig:grcover}).
Then $G^N$ is a cover of $G$ with covering map $f:(x,i)\mapsto x$. 

Conversely, it is not difficult to see that \emph{any} finite cover 
of $G$ can be obtained in this manner by some choice of $N$ and 
$\sigma_e$ (as all 
fibers $f^{-1}(x)$
of a covering map $f$ must have the same cardinality $N$, called 
the 
\emph{degree} of the cover), and that the set of graphs thus constructed
is independent of the choice of orientation $E_{\rm or}$.
\begin{figure}
\centering
\begin{tikzpicture}

\begin{scope}[thick,decoration={
    markings,
    mark=at position 0.5 with {\arrow{>}}}
    ] 
    \draw[color=red,postaction={decorate}] (0,0) -- (1.5,0) node[midway,below]
	 {$\scriptstyle (12)(3)$};
 
   \draw[postaction={decorate}] (1.5,0) -- (1.5,1.5);
    \draw (1.8,0.75) node[rotate=-90] {$\scriptstyle (1)(2)(3)$};

    \draw[color=blue,postaction={decorate}] (1.5,0) -- (0,1.5);
    \draw[color=blue] (0.92,0.92) node[rotate=-45] {$\scriptstyle (1)(23)$};

    \draw[color=cyan,postaction={decorate}] (1.5,1.5) -- (0,1.5)
	 node[midway,above] {$\scriptstyle (123)$};

    \draw[postaction={decorate}] (0,1.5) -- (0,0);
    \draw (-.3,0.75) node[rotate=90] {$\scriptstyle (1)(2)(3)$};
\end{scope}

\draw[fill=black] (0,0) circle (0.05);
\draw[fill=black] (1.5,0) circle (0.05);
\draw[fill=black] (1.5,1.5) circle (0.05);
\draw[fill=black] (0,1.5) circle (0.05);

\begin{scope}[xshift=4cm]

\fill[color=black!10!white] (-.2,-.2) rectangle (1.7,1.7);
\fill[color=black!10!white] (2.8,-.2) rectangle (4.7,1.7);
\fill[color=black!10!white] (5.8,-.2) rectangle (7.7,1.7);

\end{scope}

\begin{scope}[xshift=4cm,thick,decoration={
    markings,
    mark=at position 0.5 with {\arrow{>}}}
    ] 

    \draw[postaction={decorate}] (1.5,0) -- (1.5,1.5);
    \draw[postaction={decorate}] (4.5,0) -- (4.5,1.5);
    \draw[postaction={decorate}] (7.5,0) -- (7.5,1.5);

    \draw[postaction={decorate}] (0,1.5) -- (0,0);
    \draw[postaction={decorate}] (3,1.5) -- (3,0);
    \draw[postaction={decorate}] (6,1.5) -- (6,0);

    \draw[color=blue,postaction={decorate}] (1.5,0) -- (0,1.5);
    \draw[color=blue,postaction={decorate}] (4.5,0) -- (6,1.5);
    \draw[color=blue,postaction={decorate}] (7.5,0) to[out=100,in=30] (3,1.5);

    \draw[color=red,postaction={decorate}] (3,0) -- (1.5,0);
    \draw[color=red,postaction={decorate}] (0,0) to[out=-20,in=200] (4.5,0);
    \draw[color=red,postaction={decorate}] (6,0) -- (7.5,0);

    \draw[color=cyan,postaction={decorate}] (1.5,1.5) -- (3,1.5);
    \draw[color=cyan,postaction={decorate}] (4.5,1.5) -- (6,1.5);
    \draw[color=cyan,postaction={decorate}] (7.5,1.5) 
to[out=160,in=20] (0,1.5);

\end{scope}

\begin{scope}[xshift=4cm]
\draw[fill=black] (0,0) circle (0.05);
\draw[fill=black] (1.5,0) circle (0.05);
\draw[fill=black] (1.5,1.5) circle (0.05);
\draw[fill=black] (0,1.5) circle (0.05);
\end{scope}

\begin{scope}[xshift=7cm]
\draw[fill=black] (0,0) circle (0.05);
\draw[fill=black] (1.5,0) circle (0.05);
\draw[fill=black] (1.5,1.5) circle (0.05);
\draw[fill=black] (0,1.5) circle (0.05);
\end{scope}

\begin{scope}[xshift=10cm]
\draw[fill=black] (0,0) circle (0.05);
\draw[fill=black] (1.5,0) circle (0.05);
\draw[fill=black] (1.5,1.5) circle (0.05);
\draw[fill=black] (0,1.5) circle (0.05);
\end{scope}

\end{tikzpicture}
\caption{A finite cover $G^N$ of degree $N=3$ (right) of 
a base graph $G$ (left). The three copies of the vertices of
$G$ in $G^N$ are highlighted by the shaded regions.\label{fig:grcover}}
\end{figure}
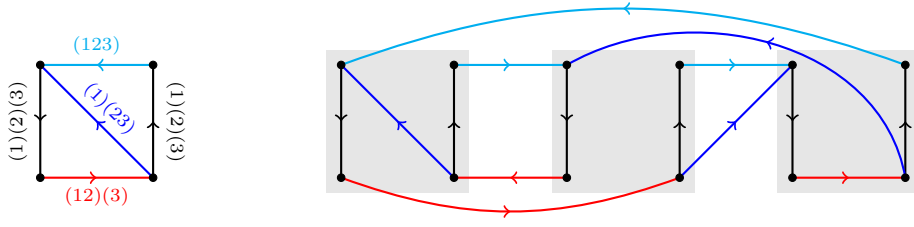

\begin{rem}
$G^N$ need not be connected for every choice 
of $\sigma_e$; for example, if each $\sigma_e$ is the identity 
permutation, then $G^N$ consists of $N$ disjoint copies of $G$.
It is always the case, however, that each connected component of $G^N$
is a cover of $G$.
\end{rem}

The above construction immediately gives rise to the natural model of 
random covers of graphs: given a finite connected base graph $G$, a random 
cover $G^N$ of degree $N$ is obtained by choosing the permutations 
$\sigma_e$ in the above construction independently and uniformly at random 
from $\mathbf{S}_N$. This model is commonly referred to as the 
\emph{random lift} model in graph theory (as a cover of degree $N$ of a 
finite graph is sometimes referred to in graph theory as an $N$-lift).

\subsubsection{Old and new eigenvalues}

From now on, we fix the base graph $G=(V,E)$ and its random lifts $G^N$ as
above. Then it is clear from the construction that the adjacency matrix 
$A^N$ of $G^N$ can be expressed as
$$
	A^N = \sum_{e=(x,y)\in E_{\rm or}}
	\big( e_ye_x^*\otimes U_e^N + e_xe_y^*\otimes U_e^{N*}
	\big),
$$
where $\{e_x\}_{x\in V}$ is the coordinate basis of $\mathbb{C}^{V}$ and 
$\{U_e^N\}_{e\in E_{\rm or}}$ are i.i.d.\ random permutation matrices of 
dimension $N$. The significance of strong convergence for this model is 
now obvious: we have encoded the adjacency matrix of the random lift model 
as a polynomial of degree one with matrix coefficients of i.i.d.\ 
permutation matrices, to which Theorem \ref{thm:bc} can be applied.

Before we can do so, however, we must clarify the nature of the optimal 
spectral gap phenomenon in the present setting. In first instance, one 
might hope to establish the obvious converse to the lower bound, 
that is, that $\|A^N|_{1^\perp}\|$ converges to 
the spectral radius $\varrho$ of the universal cover $\tilde G$.
Such a statement cannot be true in general, however, for the following 
reason. Note that for any $v\in\mathbb{C}^{V}$, we have
$$
	A^N(v\otimes 1) = Av\otimes 1,
$$
where $A$ denotes the adjacency matrix of $G$. Thus any 
eigenvalue $\lambda$ of $G$ is also an 
eigenvalue of $G^N$, since the corresponding eigenvector $v$ of $A$ lifts 
to an eigenvector $v\otimes 1$ of $A^N$: in other words, the eigenvalues 
of the base graph are always inherited by its covers. In particular, 
if the base graph $G$ happens to have an eigenvalue $\lambda$ that is 
strictly larger than $\varrho$, then $\|A^N|_{1^\perp}\|\ge\lambda>\varrho$
for all $N$.

For this reason, the best we can hope for is to show that the \emph{new} 
eigenvalues of $G^N$, that is, those eigenvalues that are not inherited 
from $G$, are asymptotically bounded by the spectral radius of $\tilde G$. 
More precisely, 
denote by 
$$
	A^N_{\rm new} =
	A^N|_{(\mathbb{C}^{V}\otimes 1)^\perp} =
	\sum_{e=(x,y)\in E_{\rm or}}
        \big( e_ye_x^*\otimes U_e^N|_{1^\perp} + e_xe_y^*\otimes 
	U_e^{N*}|_{1^\perp}
        \big)
$$
the restriction of $A^N$ to the space spanned by the new eigenvalues. Then 
we aim to show that $\|A^N_{\rm new}\|$ converges to the spectral radius 
of $\tilde G$. This is the correct formulation of the optimal spectral gap 
phenomenon for the random lift model: indeed, a variant of the 
lower bound shows that for \emph{any} sequence of covers of $G$ with 
diverging number of vertices, the maximum \emph{new} eigenvalue is 
asymptotically lower bounded by the spectral radius of $\tilde G$ 
\cite[\S 4]{Fri03}.

As was noted by Bordenave and Collins \cite{BC19}, the validity of the 
optimal spectral gap phenomenon for random lifts, conjectured by 
Friedman \cite{Fri03}, is now a simple corollary of strong convergence of 
random permutation matrices.

\begin{cor}[Optimal spectral gap of random lifts]
\label{cor:lift}
Fix any finite connected graph $G$, and denote by $\varrho$ the spectral
radius of its universal cover $\tilde G$. Then 
$$
	\lim_{N\to\infty}\|A^N_{\rm new}\| = \varrho
	\quad\text{in probability}.
$$
\end{cor}

\begin{proof}
It follows immediately from Theorem \ref{thm:bc} that $\|A^N_{\rm new}\|
\to\|a\|$ with
$$
	a = 
	\sum_{e=(x,y)\in E_{\rm or}}
	\big( e_ye_x^*\otimes \lambda(g_e) + e_xe_y^*\otimes
	\lambda(g_e^{-1})
        \big),
$$
where $g_e$ are the generators of a free group $\mathbf{F}$ and $\lambda$ 
is the left-regular representation of $\mathbf{F}$. It remains to show 
that in fact $\|a\|=\varrho$.

To see this, note that by construction, $a$ is an adjacency matrix of
an infinite graph with vertex set $V\times\mathbf{F}$. Moreover, all
vertices reachable from an initial vertex $(v_0,g)$ have the
form $(v_k,g_{(v_{k-1},v_k)}\cdots g_{(v_0,v_1)}g)$ where
$(v_0,\ldots,v_k)$ is a path in $G$ and we define 
$g_{(y,x)}=g_{(x,y)}^{-1}$ for $(x,y)\in E_{\rm or}$.
Note that this description is not unique: two paths 
define the same vertex if $g_{(v_{k-1},v_k)}\cdots g_{(v_0,v_1)}$
reduces to the same element of $\mathbf{F}$.
Thus the vertices reachable from $(v_0,g)$ are uniquely indexed by paths 
$(v_1,\ldots,v_k)$ so that $g_{(v_{k-1},v_k)}\cdots 
g_{(v_0,v_1)}$ is reduced, i.e., by nonbacktracking paths.
We have therefore shown that $a$ is the adjecency matrix of an infinite 
graph, each of whose connected components is isomorphic to $\tilde G$.
\end{proof}

Corollary \ref{cor:lift} may be viewed as a 
far-reaching generalization of Theorem \ref{thm:friedman}. Indeed, the 
permutation model of random $2r$-regular graphs is a special case of the 
random lift model, obtained by choosing the base graph $G$ to consist of 
a single vertex with $r$ self-loops (often called a ``bouquet'').

Even though Corollary \ref{cor:lift} is only concerned with the new 
eigenvalues of $G^N$, it implies the classical spectral gap property 
$\|A^N|_{1^\perp}\|\to\varrho$ whenever the base graph satisfies 
$\|A|_{1^\perp}\|\le\varrho$. Another simple consequence is that whenever 
the base graph satisfies $\|A\|>\varrho$, the random lift $G^N$ is 
connected with probability $1-o(1)$; this
holds if and only if $G$ has at least two cycles \cite[Theorem 2]{HR19}.

\iffalse
The operator $a$ in the proof yields a method for
computing the spectral radius of the universal cover of any finite
graph by means of a variational formula due to Lehner 
\cite{Leh99}, which is the analogue of Theorem \ref{thm:lehner} for
$C^*_{\rm red}(\mathbf{F})$; see also \cite{GK23}.
\fi

\subsection{Buser's conjecture}
\label{sec:buser}

Let $X$ be a hyperbolic surface, that is, a connected Riemannian surface 
of constant curvature $-1$. Then $X$ has the hyperbolic plane $\mathbb{H}$ 
as its universal cover, and we can in fact obtain 
$X=\Gamma\backslash\mathbb{H}$ as a quotient of the hyperbolic plane
by a Fuchsian group $\Gamma$ (i.e., a discrete subgroup of 
$\mathrm{PSL}_2(\mathbb{R})$) which is isomorphic to the fundamental
group $\Gamma \simeq \pi_1(X)$.

If $X$ is a closed hyperbolic surface, its Laplacian $\Delta_X$ has
discrete eigenvalues
$$
	0=\lambda_0(X)<\lambda_1(X)\le\lambda_2(X)\le\cdots
$$
The following is the direct analogue in this setting of Lemma 
\ref{lem:alonboppana}.

\begin{lem}[Huber \cite{Hub74}, Cheng \cite{Che75}]
For any sequence $X^N$ of closed hyperbolic surfaces with diverging 
diameter, we have
$$
	\lambda_1(X^N) \le \frac{1}{4} + o(1)
	\quad \text{as}\quad N\to\infty.
$$
\end{lem}

\smallskip
The significance of the value $\lambda_1(\mathbb{H})=\frac{1}{4}$ is that 
it is the bottom of the spectrum  of the Laplacian
$\Delta_{\mathbb{H}}$ on the hyperbolic plane.

It is therefore natural to ask whether there exist closed hyperbolic 
surfaces with arbitrarily large diameter (or, equivalently in this 
setting, arbitrarily large genus) that attain this bound. The existence of 
such surfaces with optimal spectral gap, a long-standing 
conjecture\footnote{% 
Curiously, Buser has at different times conjectured both existence 
\cite{Bus84} and nonexistence \cite{Bus78} of such surfaces.
On the other hand, the (very much open) Selberg eigenvalue 
conjecture in number theory \cite{Sar95} predicts that a specific class of 
noncompact hyperbolic surfaces have this property.} of Buser 
\cite{Bus84}, was resolved by Hide and Magee \cite{HM23} by means of a 
striking application of strong convergence.

\subsubsection{Random covers}

The basic approach of the work of Hide and Magee is to prove an optimal 
spectral gap phenomenon for random covers $X^N$ of a given base surface 
$X$, in direct analogy with the random lift model for graphs. To explain 
how such covers are constructed, we must first sketch the analogue in the 
present setting of the covering construction described in section 
\ref{sec:liftcons}.

Let us begin with an informal discussion. The action of a Fuchsian group 
$\Gamma$ on $\mathbb{H}$ defines a Dirichlet fundamental domain $F$ whose 
translates $\{\gamma F:\gamma\in\Gamma\}$ tile $\mathbb{H}$; $F$ is a 
polygon whose sides are given by $F\cap \gamma F$ and $F\cap 
\gamma^{-1}F$ for some generating set 
$\gamma\in\{\gamma_1,\ldots,\gamma_s\}$ of $\Gamma$. Then 
$X=\Gamma\backslash\mathbb{H}$ is obtained from $F$ by gluing each pair of 
sides $F\cap \gamma_i F$ and $F\cap \gamma_i^{-1}F$.
See Figure \ref{fig:gluing} and
\cite[Chapter 9]{Bea95}.
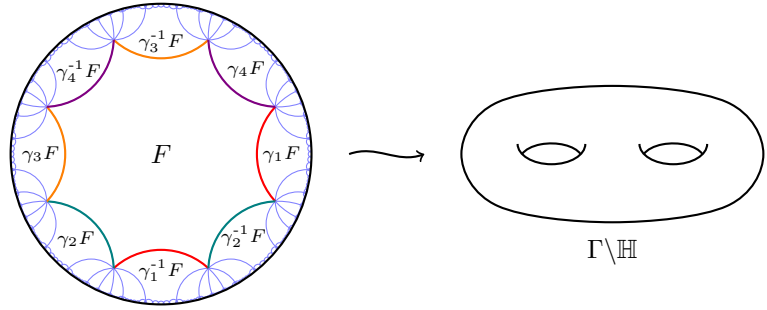
\begin{figure}
\centering
\begin{tikzpicture}[scale=2]

\begin{scope}

\clip (0,0) circle (1);

\foreach \i in {1,...,8}
{
        \hgline{45*(\i-1)-25}{45*(\i-1)+25};
        \hgline{45*(\i-1)-30}{45*(\i-1)-6};
        \hgline{45*(\i-1)-15}{45*(\i-1)-39};

        \hgline{45*(\i-1)-3}{45*(\i-1)-6.5};    
        \hgline{45*(\i-1)-.75}{45*(\i-1)-3.5};  
        \hgline{45*(\i-1)+1.25}{45*(\i-1)-1.25};        
        \hgline{45*(\i-1)+3}{45*(\i-1)+6.5};    
        \hgline{45*(\i-1)+.75}{45*(\i-1)+3.5};  

        \hgline{45*(\i-1)-8.5}{45*(\i-1)-5.5};  
        \hgline{45*(\i-1)-9.6}{45*(\i-1)-8.2};  
        \hgline{45*(\i-1)-11}{45*(\i-1)-9.6};   
        \hgline{45*(\i-1)-12.4}{45*(\i-1)-11};  
        \hgline{45*(\i-1)-13.8}{45*(\i-1)-12.4};        
        \hgline{45*(\i-1)-13.5}{45*(\i-1)-15.5};        

        \hgline{45*(\i-1)+8.5}{45*(\i-1)+5.5};  
        \hgline{45*(\i-1)+9.6}{45*(\i-1)+8.2};  
        \hgline{45*(\i-1)+11}{45*(\i-1)+9.6};   
        \hgline{45*(\i-1)+12.4}{45*(\i-1)+11};  
        \hgline{45*(\i-1)+13.8}{45*(\i-1)+12.4};        
        \hgline{45*(\i-1)+13.5}{45*(\i-1)+15.5};        

        \hgline{45*(\i-1)+16.5}{45*(\i-1)+14.5};        
        \hgline{45*(\i-1)-16.5}{45*(\i-1)-14.5};        

        \hgline{45*(\i-1)-26}{45*(\i-1)-24.5};  
        \hgline{45*(\i-1)-25.5}{45*(\i-1)-24};  
        \hgline{45*(\i-1)+26}{45*(\i-1)+24.5};  
        \hgline{45*(\i-1)+25.5}{45*(\i-1)+24};  

        \hgline{45*(\i-1)+23.3}{45*(\i-1)+24.1};        
        \hgline{45*(\i-1)+22.5}{45*(\i-1)+23.3};        
        \hgline{45*(\i-1)+21.7}{45*(\i-1)+22.5};
        \hgline{45*(\i-1)+20.9}{45*(\i-1)+21.7};

        \hgline{45*(\i-1)-17}{45*(\i-1)-16.25}; 
        \hgline{45*(\i-1)-17.75}{45*(\i-1)-17}; 
        \hgline{45*(\i-1)-18.5}{45*(\i-1)-17.75};       
        \hgline{45*(\i-1)-19.25}{45*(\i-1)-18.5};       
        \hgline{45*(\i-1)+17}{45*(\i-1)+16.25}; 
        \hgline{45*(\i-1)+17.75}{45*(\i-1)+17}; 
        \hgline{45*(\i-1)+18.5}{45*(\i-1)+17.75};       
        \hgline{45*(\i-1)+19.25}{45*(\i-1)+18.5};       

}
\end{scope}

\begin{scope}
\clip (0,0) circle (0.8222);
\hglinec{-25}{25}{thick,red};
\hglinec{-90-25}{-90+25}{thick,red};
\hglinec{-45-25}{-45+25}{thick,teal};
\hglinec{-135-25}{-135+25}{thick,teal};
\hglinec{-270-25}{-270+25}{thick,orange};
\hglinec{-180-25}{-180+25}{thick,orange};
\hglinec{-225-25}{-225+25}{thick,violet};
\hglinec{45-25}{45+25}{thick,violet};
\end{scope}

\draw (0,0) node {$F$};
\draw (0.8,0) node {$\scriptstyle\gamma_1 F$};
\begin{scope}[rotate=-45]
\draw (0.775,0) node {$\scriptstyle\gamma_2^{\text{-}1} F$};
\end{scope}
\begin{scope}[rotate=-90]
\draw (0.775,0) node {$\scriptstyle\gamma_1^{\text{-}1} F$};
\end{scope}
\begin{scope}[rotate=-135]
\draw (0.8,0) node {$\scriptstyle\gamma_2 F$};
\end{scope}
\begin{scope}[rotate=-180]
\draw (0.8,0) node {$\scriptstyle\gamma_3 F$};
\end{scope}
\begin{scope}[rotate=-225]
\draw (0.775,0) node {$\scriptstyle\gamma_4^{\text{-}1} F$};
\end{scope}
\begin{scope}[rotate=-270]
\draw (0.75,0) node {$\scriptstyle\gamma_3^{\text{-}1} F$};
\end{scope}
\begin{scope}[rotate=45]
\draw (0.8,0) node {$\scriptstyle\gamma_4 F$};
\end{scope}

\draw[thick] (0,0) circle (1);

\begin{scope}[xshift=3cm,scale=.9]

\node[rectangle, minimum height=1.8cm, minimum width=4cm] (sample) {};
\superellipse[draw, thick]{sample}{0.76};
\draw[thick] (-0.2,0.075) arc(0:-180:.25cm and .15cm);
\draw[thick] (0.7,0.075) arc(0:-180:.25cm and .15cm);

\draw[thick] (-0.235,0) arc(20:160:.23cm and .12cm);
\draw[thick] (0.9-0.235,0) arc(20:160:.23cm and .12cm);

\end{scope}

\draw[->,thick] (1.25,0) to[out=20,in=200] (1.75,0);
\draw (3,-0.65) node {$\Gamma\backslash\mathbb{H}$};

\end{tikzpicture}
\caption{Illustration of a tiling of $\mathbb{H}$ (in the Poincar\'e disc
model) by hyperbolic octagons; gluing the sides of the fundamental domain 
$F$ yields a genus $2$ surface.\label{fig:gluing}}
\end{figure}

To construct a candidate $N$-fold cover $X^N$ of $X$, we fix $N$ copies 
$F\times [N]$ of the fundamental domain and permutations 
$\sigma_1,\ldots,\sigma_s\in\mathbf{S}_N$. We then glue the side $(F\cap 
\gamma_i F) \times \{k\}$ to the corrsponding side 
$(F\cap\gamma_i^{-1}F)\times \{\sigma_i(k)\}$, that is, we scramble the 
gluing of the sides between the copies of $F$. Unlike in the case of 
graphs, however, it need not be the case that every choice of $\sigma_i$ 
yields a valid covering: if we glue the sides without regard for the 
corners of $F$, the resulting surface may develop singularities. The 
additional condition that is needed to obtain a valid covering is that 
$\sigma_1,\ldots,\sigma_s$ must satisfy the same relations as 
$\gamma_1,\ldots,\gamma_s$; that is, we must choose 
$\sigma_i=\pi_N(\gamma_i)$ for some 
$\pi_N\in\mathrm{Hom}(\Gamma,\mathbf{S}_N)$.

More formally, this construction can be implemented as follows. Fix a base 
surface $X=\Gamma\backslash\mathbb{H}$ and a homomorphism
$\pi_N\in\mathrm{Hom}(\Gamma,\mathbf{S}_N)$. Define
$$
	X^N = \Gamma \backslash (\mathbb{H}\times[N]),
$$
where we let $\gamma\in\Gamma$ act 
on $\mathbb{H}\times[N]$ as $\gamma (z,i) = (\gamma z,\pi_N(\gamma)i)$.
Then $X^N$ is an $N$-fold cover of $X$, and every $N$-fold cover of $X$ 
arises in this manner for some choice of $\pi_N$; cf.\ 
\cite[pp.\ 68--70]{Hat02} or
\cite[\S 14a and \S 16d]{Ful95}.

To define a \emph{random} cover of $X$ we may now simply choose a random 
homomorphism $\pi_N$, or equivalently, choose $\sigma_i=\pi_N(\gamma_i)$ 
to be random permutations. The major complication that 
arises here is that these permutations cannot in general be chosen 
independently, since they must satisfy the relations of 
$\Gamma$. For example, if $X$ is a closed orientable surface of genus 
$g$, then $\Gamma$ is the surface group
$$
	\Gamma \simeq \Gamma_g = \big\langle \gamma_1,\ldots,\gamma_{2g}
	~\big|~
	[\gamma_1,\gamma_2]\cdots[\gamma_{2g-1},\gamma_{2g}]=1\big\rangle
$$
where $[g,h]=ghg^{-1}h^{-1}$. In this case, the random permutations
$\sigma_i$ must be chosen to satisfy 
$[\sigma_1,\sigma_2]\cdots[\sigma_{2g-1},\sigma_{2g}]=1$, which precludes 
them from being independent. The reason this issue does not arise for 
graphs is that the fundamental group of every graph is 
free, and thus there are no relations to be satisfied.

The above obstacle has been addressed in three distinct ways.
\medskip
\begin{enumerate}[1.]
\itemsep\medskipamount
\item
While the fundamental group of a closed hyperbolic surface is never free, 
there are finite volume \emph{noncompact} hyperbolic surfaces with a free 
fundamental group; e.g., the thrice punctured sphere has 
$\pi_1(X)=\mathbf{F}_2$ and admits a finite volume hyperbolic metric 
with three cusps. Thus random covers of such surfaces can be defined using 
independent random permutation matrices. Hide and Magee \cite{HM23} proved 
an optimal spectal gap phenomenon for this model; this leads indirectly to 
a solution to Buser's conjecture by compactifying the resulting surfaces.

\item Louder and Magee \cite{LM25} showed that surface groups can be 
approximately embedded in free groups by mapping each generator of 
$\Gamma$ to a suitable word in the free group. This gives rise to a 
non-uniform random model of covers of \emph{closed} hyperbolic surfaces by 
choosing $\pi_N$ that maps each generator of $\Gamma$ to the coresponding 
word in independent random permutation matrices. 

\item
Finally, the most natural model of random covers of closed surfaces
is to choose $\pi_N\in\mathrm{Hom}(\Gamma,\mathbf{S}_N)$ uniformly at
random, that is, choose $\sigma_i=\pi_N(\gamma_i)$ uniformly at random 
among the set of tuples
$\sigma_1,\ldots,\sigma_s\in\mathbf{S}_N$ that satisfy the relation of 
$\Gamma$. This corresponds to choosing an $N$-fold cover of $X$ uniformly
at random \cite{MNP22,MP23}.
The challenge in analyzing this model is that $\sigma_i$ have a 
complicated
dependence structure that cannot be reduced to
independent random permutations.
\end{enumerate}

\medskip

These three approaches give rise to distinct models of random covers. The 
advantage of the first two approaches is that their analysis is based on 
strong convergence of independent random permutations (Theorem 
\ref{thm:bc}). This suffices for proving the \emph{existence} of covers 
with optimal spectral gaps, i.e., to resolve Buser's conjecture, but 
leaves unclear whether optimal spectral gaps are rare or 
common. That \emph{typical} covers of closed surfaces have an optimal 
spectral gap was recently proved by Magee, Puder, and the author 
\cite{MPV25} by resolving the strong convergence problem for uniformly 
random $\pi_N\in\mathrm{Hom}(\Gamma_g,\mathbf{S}_N)$ (cf.\ section 
\ref{sec:nonfree}).

The aim of the remainder of this section is to sketch how the optimal 
spectral gap problem for the Laplacian $\Delta_{X^N}$ of a random cover is 
encoded as a strong convergence problem. This reduction proceeds in 
an analogous manner for the three models 
described above. We therefore fix in the following a base surface 
$X=\Gamma\backslash\mathbb{H}$ and a sequence of random homomorphisms 
$\pi_N\in\mathrm{Hom}(\Gamma,\mathbf{S}_N)$ as in any of the above 
models. The key assumption that will be needed, which holds in all three 
models, is that the random
matrices $(U_1^N,\ldots,U_s^N)|_{1^\perp}$ defined by
\begin{align*}
	U_i^N &= \pi_N(\gamma_i) \\
\intertext{converge strongly to the operators $(u_1,\ldots,u_s)$ defined by}
	u_i &= \lambda_\Gamma(\gamma_i).
\end{align*}
Here we implicitly identify $\pi_N(\gamma_i)\in\mathbf{S}_N$ with
the corresponding $N\times N$ permutation matrix, and
$\lambda_\Gamma$ denotes the left-regular representation of $\Gamma$.

\begin{rem}
Beside models of random covers of hyperbolic surfaces, another important 
model of random surfaces is obtained by sampling from the Weil--Petersson 
measure on the moduli space of hyperbolic surfaces of genus $g$; this 
may be viewed as the natural notion of a typical surface of genus $g$.
In a tour-de-force, Anantharaman and Monk \cite{AM25i,AM25ii} 
proved that the Weil--Petersson model also exhibits an optimal spectral 
gap phenomenon by using methods inspired by Friedman's original proof of 
Theorem \ref{thm:friedman}. In contrast to random cover models, it does
not appear that this result can be reduced to a strong convergence 
problem. Nonetheless, it was recently shown by Hide--Macera--Thomas
\cite{HMT25b} that the the polynomial method, which plays a key role in 
\cite{MPV25}, can be applied directly to achieve a new proof of this 
result.
\end{rem}

\subsubsection{Exploiting strong convergence}

In contrast to the setting of random lifts of graphs, it is not 
immediately clear how the Laplacian spectrum of random surface covers 
relates to strong convergence. This connection is due to Hide and 
Magee~\cite{HM23}; for expository purposes, we sketch a variant of
their argument \cite{HMN25}.

We begin with some basic observations. Any $f\in L^2(X)$ lifts to a 
function in $L^2(X^N)$ by composing it with the covering 
map $\iota:X^N\to X$. As
$$
	\Delta_{X^N}(f\circ\iota) = \Delta_Xf \circ \iota,
$$
it follows precisely as for random lifts of graphs that the spectrum of 
the base surface $X$ is a subset of that of any of its covers 
$X^N$. What we aim to show is that
the smallest \emph{new} eigenvalue of $\Delta_{X^N}$, that is, 
the smallest eigenvalue of its restriction $\Delta^{\rm 
new}_{X^N}$ to the orthogonal complement of 
functions lifted from $X$, converges to the bottom of the spectrum of
$\Delta_{\mathbb{H}}$. In other words, we aim to prove that
$$
	\lim_{N\to\infty}
	\big\|e^{-\Delta_{X^N}^{\rm new}}\big\| 
	=	
	\big\|e^{-\Delta_{\mathbb{H}}}\big\| =
	e^{-\frac{1}{4}}.
$$
This leads us to consider the heat operators
$e^{-\Delta_{\mathbb{H}}}$ and $e^{-\Delta_{X^N}}$.

Recall that $e^{-\Delta_{\mathbb{H}}}$
is an integral operator on $L^2(\mathbb{H})$ with a smooth kernel
$p_\mathbb{H}(x,y)$. The Laplacian $\Delta_{X^N}$ on 
$X^N=\Gamma\backslash(\mathbb{H}\times[N])$ is obtained by restricting
the Laplacian on $\mathbb{H}\times[N]$ to 
functions that are invariant under $\Gamma$. In 
particular, this implies that $e^{-\Delta_{X^N}}$ may be 
viewed as an integral operator on $L^2(F\times[N])$ with kernel
$$
	p_{X^N}((x,i),(y,j)) =
	\sum_{\gamma\in\Gamma}
	p_{\mathbb{H}}(x,\gamma y)\, 1_{i=\pi_N(\gamma)j}
$$
by parameterizing $X^N$ as $F\times [N]$, where $F$ is the 
fundamental domain of the action of 
$\Gamma$ on $\mathbb{H}$. See, for example, \cite[\S 3.7]{Ber16} or
\cite[\S 2]{HMN25}.

In the following, we identify $L^2(F\times[N]) \simeq
L^2(F)\otimes\mathbb{C}^N$, and denote by $a_\gamma$ the integral operator 
on $L^2(F)$ with kernel $p_{\mathbb{H}}(x,\gamma y)$. In this notation, 
the above expression can be rewritten in the more suggestive form 
\begin{align*}
	e^{-\Delta_{X^N}} &=
	\sum_{\gamma\in\Gamma} 
	a_\gamma \otimes \pi_N(\gamma).\\
\intertext{In particular, we have}
	e^{-\Delta_{X^N}^{\rm new}} &=
	\sum_{\gamma\in\Gamma} 
	a_\gamma \otimes \pi_N(\gamma)|_{1^\perp}.
\end{align*}
Since $\pi_N$ is a homomorphism, each 
$\pi_N(\gamma)=\pi_N(\gamma_{i_1}\cdots\gamma_{i_k})=U_{i_1}^N\cdots 
U_{i_k}^N$ can be written 
as a word in the random permutation matrices $U_i^N=\pi_N(\gamma_i)$ 
associated to the generators $\gamma_i$ of $\Gamma$. Thus
$e^{-\Delta_{X^N}^{\rm new}}$ is nearly, but not exactly, a noncommutative 
polynomial
of $(U_1^N,\ldots,U_s^N)|_{1^\perp}$ with matrix coefficients:
\smallskip
\begin{enumerate}[$\bullet$]
\itemsep\medskipamount
\item The above sum is over all $\gamma\in\Gamma$ with no
bound on the word length $|\gamma|$. However, as $p_{\mathbb{H}}(x,y)$
decays rapidly as a function of $\mathrm{dist}_{\mathbb{H}}(x,y)$ (this 
can be read 
off from the explicit expression for $p_\mathbb{H}(x,y)$ \cite{HM23} or
from general heat kernel estimates \cite{HMN25}), the size of the 
coefficients 
$\|a_\gamma\|$ decays rapidly as function of $|\gamma|$. The infinite sum 
is therefore well approximated by a finite sum.
\item The coefficients $a_\gamma$ are operators rather than matrices.
However, when $X$ is a closed surface, $a_\gamma$ are compact operators
and are therefore well approximated by matrices. (The 
argument in the case that
$X$ is noncompact requires an additional truncation to remove the cusps;
see \cite{HM23,Moy25} for details.)
\end{enumerate}
\smallskip
We therefore conclude that $e^{-\Delta_{X^N}^{\rm new}}$ is well 
approximated in operator norm by a noncommutative polynomial in 
$(U_1^N,\ldots,U_s^N)|_{1^\perp}$ with matrix coefficients.
In particular, we can apply strong convergence to conclude that
$$
	\lim_{N\to\infty}
	\big\|e^{-\Delta_{X^N}^{\rm new}}\big\|
	=
	\|b\|
$$
with
$$
	b = 
	\sum_{\gamma\in\Gamma} 
	a_\gamma \otimes \lambda_\Gamma(\gamma).
$$
It remains to observe that the operator $b$ is $e^{-\Delta_{\mathbb{H}}}$ 
in disguise. To see this, note that the map $\eta:F\times\Gamma\to\mathbb{H}$
defined by $\eta(x,g)=g^{-1}x$ is a.e.\ invertible, as the
translates of the fundamental domain tile $\mathbb{H}$. Thus
$f\mapsto \tilde f=f\circ\eta$ defines an isomorphism 
$L^2(\mathbb{H})\simeq L^2(F\times\Gamma)$.
We can now readily compute for any $f\in L^2(\mathbb{H})$
$$
	b\tilde f(x,g) =
	\sum_{\gamma\in\Gamma} \int_F
	p_{\mathbb{H}}(x,\gamma y)\,f(g^{-1}\gamma y)\,
	dy =
	\int_{\mathbb{H}}
	p_{\mathbb{H}}(g^{-1}x,y)\,f(y)\,dy =
	\widetilde{e^{-\Delta_{\mathbb{H}}}f}(x,g),
$$
where we used that $p_{\mathbb{H}}(g^{-1}x,y)=p_{\mathbb{H}}(x,gy)$.

\begin{rem}
There are several variants of the above argument. The original work of 
Hide and Magee \cite{HM23} used the resolvent $(z-\Delta_{X^N})^{-1}$ 
instead of $e^{-\Delta_{X^N}}$. The heat operator approach of 
Hide--Moy--Naud \cite{HMN25,Moy25} has the advantage that it extends to 
surfaces with variable negative curvature by using heat 
kernel estimates. For hyperbolic surfaces, another 
variant due to Hide--Macera--Thomas \cite{HMT25} uses a 
specially designed function $h$ with the property that $h(\Delta_{X^N})$ 
is already a noncommutative polynomial of $U_1^N,\ldots,U_s^N$ with 
operator coefficients, avoiding the need to truncate the sum over 
$\gamma$. The advantage of this approach is that it leads to much better 
quantitative estimates, since the truncation of the sum is the main source 
of loss in the previous arguments. Finally, Magee \cite{Mag24} presents a 
more general perspective that uses the continuity of induced 
representations under strong convergence.
\end{rem}

\subsection{Random Schreier graphs}
\label{sec:cassidy}

In this section, we take a different perspective on random regular graphs 
that will lead us in a new direction.

\begin{defn}
Given $\sigma_1,\ldots,\sigma_r\in\mathbf{S}_N$ and an action 
$\mathbf{S}_N\curvearrowright V$ of the symmetric group on a finite set 
$V$, the \emph{Schreier graph} $\mathrm{Sch}(\mathbf{S}_N\curvearrowright 
V; \sigma_1,\ldots,\sigma_r)$ is the $2r$-regular graph with vertex set 
$V$, where each vertex $v\in V$ has neighbors 
$\sigma_i(v),\sigma_i^{-1}(v)$ for $i=1,\ldots,r$ (allowing for
multiple edges and self-loops).
\end{defn}

The permutation model of random $2r$-regular graphs that was introduced
in section \ref{sec:introgaps} is merely the special case
$\mathrm{Sch}(\mathbf{S}_N\curvearrowright [N]; \sigma_1,\ldots,\sigma_r)$
where $\mathbf{S}_N\curvearrowright [N]$ is the natural action of 
permutations of $[N]$ on the points of $[N]$, and 
$\sigma_1,\ldots,\sigma_r$ are independent and uniformly distributed
random elements of $\mathbf{S}_N$.

We may however ask what happens if we consider other actions of the
symmetric group. Following \cite{FJRST96,Cas24}, denote by $[N]_k$ the
set of all $k$-tuples of distinct elements of $[N]$. Then we obtain
the natural action $\mathbf{S}_N\curvearrowright [N]_k$ by letting 
$\sigma$ act on each element of the tuple, that is,
$\sigma(i_1,\ldots,i_k) = (\sigma(i_1),\ldots,\sigma(i_k))$. If we again 
choose $\sigma_1,\ldots,\sigma_r$ to be i.i.d.\ uniform
random elements of $\mathbf{S}_N$, then
$$
	\mathrm{Sch}(\mathbf{S}_N\curvearrowright [N]_k; 
	\sigma_1,\ldots,\sigma_r)
$$
yields a new model of random $2r$-regular graphs that generalizes the 
permutation model. The interesting aspect of these graphs is that even
though the number of vertices $\sim N^k$ grows rapidly as we increase $k$,
the number of random bits $\sim rN\log N$ that generate the graph is fixed 
independently of $k$. We may therefore think of the model as becoming 
increasingly less random as $k$ is increased.\footnote{%
A different, much less explicit approach to derandomization of random 
graphs from a theoretical computer science perspective may be found in 
\cite{MOP19,OW20}.}

What is far from obvious is whether the optimal spectral gap of the random 
graph persists as we increase $k$. Let us consider the two extremes.
\smallskip
\begin{enumerate}[$\bullet$]
\itemsep\medskipamount
\item The case $k=1$ is the permutation model of random regular graphs,
which have an optimal spectral gap by Theorem \ref{thm:friedman}.
\item The case $k=N$ corresponds to the Cayley graph of $\mathbf{S}_N$ 
with the random generators $\sigma_1,\ldots,\sigma_r$, since 
$[N]_N\simeq\mathbf{S}_N$. Whether random Cayley graphs of 
$\mathbf{S}_N$ have an optimal 
spectral gap is a long-standing question (see section \ref{sec:cayley}) 
that remains wide open: it has not even been shown that the maximum 
nontrivial eigenvalue is bounded away from the trivial eigenvalue in this 
setting.
\end{enumerate}
\smallskip
The intermediate values of $k$ interpolate between these two extremes.
In a major improvement over previously known results,
Cassidy \cite{Cas24} recently proved that the optimal spectral gap 
persists in the range $k\le N^\alpha$ for some $\alpha<1$.

\begin{thm}
\label{thm:cassidy}
Denote by $A^{N,k}$ the adjacency matrix of 
$\mathrm{Sch}(\mathbf{S}_N\curvearrowright 
[N]_k;\sigma_1,\ldots,\sigma_r)$ where $\sigma_1,\ldots,\sigma_r$
are i.i.d.\ uniform random elements of $\mathbf{S}_N$.
Then
$$
	\|A^{N,k_N}|_{1^\perp}\| = 2\sqrt{2r-1} + o(1)
	\quad\text{with probability}\quad 1-o(1)
$$
as $N\to\infty$ whenever $k_N\le N^{\frac{1}{20}-\delta}$, 
for any $\delta>0$.
\end{thm}

This yields a natural model of random $2r$-regular 
graphs with $|V|$ vertices that has an optimal spectral gap using only 
$\sim (\log|V|)^{20+\delta}$ bits of randomness, as compared to
$\sim|V|\log|V|$ bits for ordinary random regular graphs.

Theorem \ref{thm:cassidy} arises from a much more general result about 
strong convergence of representations of $\mathbf{S}_N$. To motivate
this result, note that we can write
$$
	A^{N,k} = \pi_{N,k}(\sigma_1) + \pi_{N,k}(\sigma_1)^* +
	\cdots + \pi_{N,k}(\sigma_r) + \pi_{N,k}(\sigma_r)^*,
$$
where $\pi_{N,k}:\mathbf{S}_N\to \mathrm{M}_{[N]_k}(\mathbb{C})$  
maps $\sigma\in\mathbf{S}_N$ to the permutation 
matrix defined by its action on $[N]_k$. Then $\pi_{N,k}$ is clearly 
a group representation of $\mathbf{S}_N$, so it
decomposes as a direct sum of irreducible representations
$\pi_N^\lambda$. Theorem~\ref{thm:cassidy} now follows 
from the following result about strong convergence of
irreducible representations of $\mathbf{S}_N$ that vastly generalizes
Theorem \ref{thm:bc} (which is the special case where $\pi_N^\lambda =
\mathrm{std}_N$ is the standard representation, so that
$\dim(\mathrm{std}_N)=N-1$). For 
expository purposes, we 
state the result in a slightly more general form than is given in 
\cite{Cas24}.

\begin{thm}
\label{thm:cassidystrong}
Let $\boldsymbol{\sigma}=(\sigma_1,\ldots,\sigma_r)$ be i.i.d.\ uniform 
random elements of $\mathbf{S}_N$, and let
$\boldsymbol{u}=(u_1,\ldots,u_r)$ be defined as in Theorem \ref{thm:bc}. 
Then
$$
	\lim_{N\to\infty}
	\max_{1<\dim(\pi_N^\lambda)\le \exp(N^{\frac{1}{20}-\delta})}
	\big|
	\|P(\pi_N^\lambda(\boldsymbol{\sigma}),
	\pi_N^\lambda(\boldsymbol{\sigma}^*))\| -
	\|P(\boldsymbol{u},\boldsymbol{u}^*)\|
	\big|=0
$$
in probability
for every $\delta>0$, $D\in\mathbb{N}$, and 
$P\in\mathrm{M}_D(\mathbb{C})\otimes
\mathbb{C}\langle x_1,\ldots,x_{2r}\rangle$, where the maximum is taken
over irreducible representations $\pi_N^\lambda$ of $\mathbf{S}_N$.
\end{thm}

\begin{proof}
The irreducible representations $\pi_N^\lambda$ are
indexed by Young diagrams $\lambda \vdash N$. The argument in
\cite[\S 6]{Eti14} shows that for any $0<\delta'<\delta$ and sufficiently
large $N$, every irreducible representation with
$\dim(\pi_N^\lambda)\le \exp(N^{\frac{1}{20}-\delta})$ has the property
that the first row of either $\lambda$ or of the conjugate diagram
$\lambda'$ has length
at least $N-N^{\frac{1}{20}-\delta'}$. In the first case, the conclusion 
\begin{equation}
\label{eq:cassidy}
	\lim_{N\to\infty}
	\max_{\lambda_1 \ge N-N^{\frac{1}{20}-\delta'}}
	\big|
	\|P(\pi_N^\lambda(\boldsymbol{\sigma}),
	\pi_N^\lambda(\boldsymbol{\sigma}^*))\| -
	\|P(\boldsymbol{u},\boldsymbol{u}^*)\|
	\big|=0
\end{equation}
follows from the proof of \cite[Theorem 1.9]{Cas24}.

On the other hand, as
$\pi_N^{\lambda'}(\sigma)=\mathrm{sgn}(\sigma)\pi_N^\lambda(\sigma)$
\cite[Theorem 6.7]{Jam78}, we obtain
$$
	\lim_{N\to\infty}
	\max_{\lambda_1 \ge N-N^{\frac{1}{20}-\delta'}}
	\big|
	\|P(\pi_N^{\lambda'}(\boldsymbol{\sigma}),
	\pi_N^{\lambda'}(\boldsymbol{\sigma}^*))\| -
	\|P(\mathrm{sgn}(\boldsymbol{\sigma})\boldsymbol{u},\mathrm{sgn}(\boldsymbol{\sigma})\boldsymbol{u}^*)\|
	\big|=0
$$
using that \eqref{eq:cassidy} holds uniformly over any finite set of 
polynomials 
$P$, and thus in particular over the polynomials 
$(\boldsymbol{u},\boldsymbol{u}^*)\mapsto P(\boldsymbol{\varepsilon 
u},\boldsymbol{\varepsilon u}^*)$ for all 
choices of signs $\boldsymbol{\varepsilon}\in 
\{-1,1\}^r$. It remains to note that 
$\|P(\mathrm{sgn}(\boldsymbol{\sigma})\boldsymbol{u},\mathrm{sgn}(\boldsymbol{\sigma})\boldsymbol{u}^*)\| 
= \|P(\boldsymbol{u},\boldsymbol{u}^*)\|$ by the Fell absorption principle 
\cite[Proposition 8.1]{Pis03}.
\end{proof}

The above results are made possible by a marriage of two complementary 
developments: new representation-theoretic ideas due Cassidy, and the 
polynomial method for proving strong convergence. Here, we merely give a 
hint of the underlying phenomenon, and refer to \cite{Cas24} for the 
details.

Fix $\lambda\vdash k$, and consider the sequence of Young diagrams 
$\lambda(N)\vdash N$ (for $N\ge 2k$) so that removing the first row of 
$\lambda(N)$ yields $\lambda$; in particular, the first row has length 
$N-k$. Then the sequence of representations $\pi_N^{\lambda(N)}$ is called 
\emph{stable}~\cite{Far14}. As was the case in section \ref{sec:poly}, any 
stable representation has the property that 
$$
	\mathbf{E}\big[\tr \pi_N^{\lambda(N)}(\sigma_{w_1}\cdots\sigma_{w_k})
	\,\big] = \Psi_{\boldsymbol{w}}^\lambda(\tfrac{1}{N})
$$
is a rational 
function of $\frac{1}{N}$. Moreover, as in
Corollary \ref{cor:divisor},
\begin{equation}
\label{eq:precassidy}
	\mathbf{E}\big[\tr \pi_N^{\lambda(N)}(\sigma_{w_1}\cdots\sigma_{w_k})
	\,\big] =
	O\bigg(\frac{1}{N}\bigg)
\end{equation}
if $g_{w_1}\cdots g_{w_r}$ is a non-power, 
where $g_1,\ldots,g_r$ are free generators of $\mathbf{F}_r$; see
\cite{HP23}. These facts already suffice, by the polynomial method, for 
proving a form of Theorem \ref{thm:cassidystrong} that applies to 
representations of polynomial dimension $\dim(\pi_N^\lambda)\le N^k$ for 
any fixed $k$ \cite{CGTV25}. This falls far short, however, of
Theorem \ref{thm:cassidystrong}.

The key new ingredient that is developed in \cite{Cas24} is a major
improvement of \eqref{eq:precassidy}: when $g_{w_1}\cdots g_{w_r}$ is a 
non-power, it turns out that in fact
$$
	\mathbf{E}\big[\tr \pi_N^{\lambda(N)}(\sigma_{w_1}\cdots\sigma_{w_k})
	\,\big] =
	O\bigg(\frac{1}{\dim\big(\pi_N^{\lambda(N)}\big)}\bigg).
$$
The surprising aspect of this bound is that it exhibits \emph{more} 
cancellation as the dimension of the representation increases---contrary 
to what one may expect, since the model becomes ``less random''. This 
phenomenon therefore captures a kind of pseudorandomness in 
high-dimensional representations. This is achieved in \cite{Cas24} by 
combining a new representation of the stable characters of $\mathbf{S}_N$ 
with ideas from low-dimensional topology. The improved estimate makes it 
possible to Taylor expand the rational function 
$\Psi_{\boldsymbol{w}}^\lambda$ to much higher order in the polynomial 
method, enabling it to reach representations of quasi-exponential 
dimension.

Taken more broadly, high dimensional representations of finite and matrix 
groups form a natural setting for the study of strong convergence and give 
rise to many interesting questions. For the unitary group $U(N)$, strong 
convergence was established earlier by Bordenave and Collins \cite{BC20} 
for representations of polynomial dimension $N^k$, and by Magee and de la 
Salle \cite{MdlS24} for representations of quasi-exponential dimension 
$\exp(N^\alpha)$ (further improved in \cite{CGV25} using complementary 
ideas). On the other hand it is a folklore conjecture (see, e.g., 
\cite[Conjecture~1.6]{RS19}) that \emph{any} sequence of representations 
of $\mathbf{S}_N$ of diverging dimension should give rise to optimal 
spectral gaps; Theorem \ref{thm:cassidystrong} is at present the best 
known result in this direction. Analogous questions for finite simple 
groups of Lie type remain entirely open.

\subsection{The Peterson--Thom conjecture}
\label{sec:hayes}

In this section, we discuss a very different application of strong 
convergence to the theory of von~Neumann algebras, which has motivated 
many recent works in this area.

Recall that a von~Neumann algebra is defined as a unital $C^*$-algebra,
but is closed in the strong operator topology rather than the operator 
norm topology; see Remark \ref{rem:vna}. An important example is the 
\emph{free group factor}
$$
	L(\mathbf{F}_r) = 
	\mathrm{cl}_{\rm SOT}\big(\mathop{\mathrm{span}}
	\{\lambda(g):g\in\mathbf{F}_r\}\big),
$$
i.e., the closure of $C^*_{\rm red}(\mathbf{F}_r)$ in the strong operator 
topology. Von~Neumann algebras are much ``bigger'' than $C^*$-algebras and 
thus much less well understood; for example, it is not even known whether
or not $L(\mathbf{F}_r)$ and $L(\mathbf{F}_s)$
are isomorphic for $r\ne s$, which is one of the major open problems in 
this area.

However, the subclass of \emph{amenable} von~Neumann algebras---the 
counterpart in this context of the notion of an amenable group---is very 
well understood due to the work of Connes \cite{Con76}. For 
example, amenable von~Neumann algebras can be characterized as those that 
are approximately finite dimensional, i.e., the closure in the strong 
operator topology of an increasing net of matrix algebras. It is therefore 
natural to try to gain a better understanding of non-amenable von~Neumann 
algebras such as $L(\mathbf{F}_r)$ by studying the collection of its 
amenable subalgebras. The following conjecture of Peterson and Thom 
\cite{PT11}---now a theorem due to the works to be discussed below---is in 
this spirit: it states that two distinct \emph{maximal} amenable 
subalgebras of $L(\mathbf{F}_r)$ cannot have a too large overlap.

\begin{thm}[Peterson--Thom conjecture]
\label{thm:pt}
Let $r\ge 2$.
If $M_1$ and $M_2$ are distinct maximal amenable von~Neumann subalgebras 
of $L(\mathbf{F}_r)$, then $M_1\cap M_2$ is not diffuse.
\end{thm}

A von~Neumann algebra is called diffuse if it has no minimal projection. 
Being non-diffuse is a strong constraint: if $M$ is not diffuse, then the 
spectral distribution $\mu_a$ of every self-adjoint $a\in M$ must have an 
atom. (Here and below, we always compute laws with respect to the 
canonical trace $\tau$ on $L(\mathbf{F}_r)$.) 

\begin{example}
Let $M_i$ be the von~Neumann subalgebra of $L(\mathbf{F}_2)$ generated 
by $\lambda(g_i)$, where $g_1,g_2$ are free generators of $\mathbf{F}_2$.
Then $M_1,M_2\simeq L(\mathbb{Z})$ are maximal amenable, but $M_1\cap M_2$ 
is trivial and thus certainly not diffuse.
\end{example}

The affirmative solution of the Peterson-Thom conjecture was made 
possible by the 
work of Hayes \cite{Hay22}, who in fact provides a much stronger result. 
For every von~Neumann subalgebra $M\le L(\mathbf{F}_r)$, Hayes 
defines a quantity $h(M:L(\mathbf{F}_r))$ called the \emph{$1$-bounded 
entropy in the presence of $L(\mathbf{F}_r)$}, see \cite[\S 
2.2~and~Appendix]{HJK25}, that satisfies
$h(M:L(\mathbf{F}_r))\ge 0$ for every $M$ and
$h(M:L(\mathbf{F}_r))=0$ if $M$ is amenable. Hayes' main result is that 
the converse of this property also holds---thus providing an 
entropic 
characterization of amenable subalgebras of 
$L(\mathbf{F}_r)$.

\begin{thm}[Hayes]
\label{thm:hayes}
$M\le L(\mathbf{F}_r)$ is amenable if and only 
if $h(M:L(\mathbf{F}_r))=0$.
\end{thm}

Theorem \ref{thm:pt} follows immediately from Theorem \ref{thm:hayes} 
using the following subadditivity property of the $1$-bounded entropy 
\cite[\S 2.2]{HJK25}: 
$$
	h(M_1\vee M_2:L(\mathbf{F}_r)) \le 
	h(M_1:L(\mathbf{F}_r))+h(M_2:L(\mathbf{F}_r))
$$
whenever $M_1\cap M_2$ is diffuse, where
$M_1\vee M_2$ is the
von~Neumann algebra generated by $M_1,M_2$. Indeed, 
it follows that if $M_1\ne M_2$ are amenable and $M_1\cap M_2$ is 
diffuse then $M_1\vee M_2$ is amenable, so $M_1,M_2$ cannot
be maximal amenable.

Theorem~\ref{thm:hayes} is not stated as such in \cite{Hay22}. The key 
insight of Hayes was that the validity of Theorem \ref{thm:hayes} can be 
reduced (in a highly nontrivial fashion) to proving strong convergence of 
a certain random matrix model. This problem was outside the reach of the 
methods that were available when \cite{Hay22} was written, and thus 
Theorem~\ref{thm:hayes} was given there as a conditional statement. Hayes' 
work strongly influenced new developments on the random matrix side, and 
the requisite strong convergence has now been proved by
several approaches \cite{BC22,BC24,MdlS24,Par24,CGV25}. This has 
in turn not only completed the proofs of Theorems 
\ref{thm:pt}~and~\ref{thm:hayes}, but also led to new developments on the 
operator algebras side \cite{HJK25}.

In the remainder of this section, we aim to discuss the relevant strong 
convergence problem, and to give a hint as to how it gives rise to 
Theorem \ref{thm:hayes}.

\subsubsection{Tensor models}
\label{sec:subexpon}

Let $\boldsymbol{U}^N=(U_1^N,\ldots,U_r^N)$ be independent
Haar-distributed random unitary matrices of dimension $N$, and let
$\boldsymbol{u}=(u_1,\ldots,u_r)$ be the standard generators
of $L(\mathbf{F}_r)$ as defined in section \ref{sec:introgaps}.
That $\boldsymbol{U}^N$ strongly converges to $\boldsymbol{u}$ is
a consequence of the Haagerup-Thorbj{\o}rnsen theorem for GUE matrices, as 
was shown by Collins and Male \cite{CM14}. 
The basic question posed by Hayes is
whether strong convergence continues to hold if we consider the tensor 
product of two 
independent copies of this model. More precisely: 

\begin{qst}
Let $\boldsymbol{\tilde U}^N$ be an independent copy of 
$\boldsymbol{U}^N$.
Is it true that the
family
\begin{align*}
	(\boldsymbol{U}^N\otimes\id,~
	\id\otimes\boldsymbol{\tilde U}^N) &=
	(U_1^N\otimes\id,~\ldots,~U_r^N\otimes\id,~
	\id\otimes\tilde U_1^N,~\ldots,~\id\otimes\tilde U_r^N) \\
\intertext{of random unitaries of dimension $N^2$ converges strongly to}
	(\boldsymbol{u}\otimes\id,~\id\otimes\boldsymbol{u}) &=
	(u_1\otimes\id,~\ldots,~u_r\otimes\id,~
	\id\otimes u_1,~\ldots,~\id\otimes u_r)
\end{align*}
as $N\to\infty$?\footnote{%
Recall that we always denote by $\otimes=\otimes_{\rm min}$ 
the minimal tensor product, see section \ref{sec:scalarmtx}.}
(Alternatively,
one may replace $\boldsymbol{U}^N,\boldsymbol{\tilde U}^N$ by
independent GUE matrices and $\boldsymbol{u}$ by a free semicircular
family.)
\end{qst}
The main result of Hayes \cite[Theorem 1.1]{Hay22} states that an 
affirmative answer to this
question implies the validity of Theorem \ref{thm:hayes}. 

Because $\boldsymbol{U}^N$ and $\boldsymbol{\tilde U}^N$ are independent, 
it is natural to attempt to apply strong convergence of
each copy separately. To this end, note that for any noncommutative 
polynomial $P\in\mathbb{C}\langle x_1,\ldots,x_{4r}\rangle$, we can write
$$
	P\big(
	\boldsymbol{U}^N\otimes\id,~
	\boldsymbol{U}^{N*}\otimes\id,~
	\id\otimes\boldsymbol{\tilde U}^N,~
	\id\otimes\boldsymbol{\tilde U}^{N*}\big) =
	P_N\big(\boldsymbol{U}^N,
	\boldsymbol{U}^{N*}\big)
$$
where $P_N\in\mathrm{M}_N(\mathbb{C})\otimes\mathbb{C}\langle 
x_1,\ldots,x_{2r}\rangle$ is a noncommutative polynomial with 
matrix coefficients of dimension $N$ that depend only on 
$\boldsymbol{\tilde U}^N$.
We can now condition on $\boldsymbol{\tilde U}^N$ and think of $P_N$ as a 
determinstic polynomial with matrix coefficients. In particular,
one may hope to use strong convergence of 
$\boldsymbol{U}^N$ to $\boldsymbol{u}$ to show that
\begin{equation}
\label{eq:mtxcoeff}
	\big\|P_N\big(\boldsymbol{U}^N,
        \boldsymbol{U}^{N*}\big)\big\|
	\stackrel{?}{=}
	(1+o(1))
	\|P_N(\boldsymbol{u},\boldsymbol{u}^*)\|
\end{equation}
as $N\to\infty$. If \eqref{eq:mtxcoeff} holds, 
then the proof of strong convergence of the tensor model is readily 
completed. Indeed, we may now write
$P_N(\boldsymbol{u},\boldsymbol{u}^*) = 
Q(\boldsymbol{\tilde U}^N,\boldsymbol{\tilde U}^{N*})$
where $Q\in C^*_{\rm red}(\mathbf{F}_r)\otimes \mathbb{C}\langle
x_1,\ldots,x_{2r}\rangle$ is a polynomial with operator coefficients that
depend only on $\boldsymbol{u}$. Since $C^*_{\rm red}(\mathbf{F}_r)$ is 
exact, Lemma \ref{lem:exact} yields
$$
	\big\|Q\big(\boldsymbol{\tilde U}^N,\boldsymbol{\tilde 
	U}^{N*}\big)\big\| =
	(1+o(1))\|Q(\boldsymbol{u},\boldsymbol{u}^*)\|
$$
as $N\to\infty$. Finally, as
$$
	Q(\boldsymbol{u},\boldsymbol{u}^*) =
	P(\boldsymbol{u}\otimes\id,~
	\boldsymbol{u}^*\otimes\id,~
	\id\otimes\boldsymbol{u},~
	\id\otimes\boldsymbol{u}^*),
$$
the desired strong convergence property is established.

This argument reduces the question of strong convergence of the 
tensor product of two independent families of random unitaries to a 
question about strong convergence \eqref{eq:mtxcoeff} of a single family 
of random unitaries for polynomials with matrix coefficients. The latter 
is far from obvious, however. While norm convergence of any \emph{fixed} 
polynomial $P$ with matrix coefficients is an automatic consequence of 
strong convergence of $\boldsymbol{U}^N$ (Lemma \ref{lem:mtxstr}), here 
the polynomial $P_N\in 
\mathrm{M}_{D_N}(\mathbb{C})\otimes\mathbb{C}\langle 
x_1,\ldots,x_{2r}\rangle$ and the dimension $D_N$ of the matrix 
coefficients changes with $N$. This cannot follow from 
strong convergence alone, but may be obtained if the
proof of strong convergence provides sufficiently strong quantitative 
estimates.

The question of strong convergence of polynomials $P_N$ with matrix 
coefficients of increasing dimension $D_N$ was first raised by Pisier 
\cite{Pis14} in his study of subexponential operator spaces. Pisier noted 
that \eqref{eq:mtxcoeff} can \emph{fail} for matrix coefficients of 
dimension $D_N \ge e^{CN^2}$ (see \cite[Appendix A]{CGV25}); while a 
careful inspection of the quantitative estimates in the strong convergence 
proof of Haagerup--Thorbj{\o}rnsen yields that \eqref{eq:mtxcoeff} holds 
for matrix coefficients of dimension $D_N = o(N^{1/4})$. This leaves a 
huge gap between the upper and lower bound, and in particular excludes the 
case $D_N=N$ that is required to prove Theorem \ref{thm:hayes}.

Recent advances in strong convergence have led to a greatly improved 
understanding of this problem by means of several independent methods 
\cite{BC24,MdlS24,Par24,CGV25}, all of which suffice to complete the proof 
of Theorem \ref{thm:hayes}. The best result to date, obtained by the 
polynomial method \cite{CGV25}, is that strong convergence in the GUE and 
Haar unitary models remains valid for matrix coefficients of dimension 
$D_N=e^{o(N)}$. Let us briefly sketch how this is achieved.

The arguments that we developed in section \ref{sec:poly} for random 
permutation matrices can be applied in a very similar manner to random 
unitary matrices. In particular, one obtains
as in the proof of Proposition \ref{prop:smex}
an estimate of the form 
$$
	\bigg| 
	\mathbf{E}\big[ \ntr 
	h(P(\boldsymbol{U}^N,\boldsymbol{U}^{N*}))\,\big]
	- \nu_0(h) - \frac{\nu_1(h)}{N} \bigg|
	\le \frac{C}{N^2} \|h\|_{C^\ell[-K,K]}.
$$
Here $P$ is any noncommutative polynomial with matrix 
coefficients of
dimension $D$, $\ell$ is an absolute constant, $C$ is a constant that 
depends only on the degree of $P$, and $\nu_0,\nu_1$ are Schwartz 
distributions that are supported in 
$[-\|P(\boldsymbol{u},\boldsymbol{u}^*)\|,
\|P(\boldsymbol{u},\boldsymbol{u}^*)\|]$.
If we choose a test function $h$ that vanishes in the latter 
interval, we obtain 
$$
	\bigg| 
	\mathbf{E}\big[ \tr 
	h(P(\boldsymbol{U}^N,\boldsymbol{U}^{N*}))\,\big]
	\bigg|
	\le \frac{CD}{N} \|h\|_{C^\ell[-K,K]}
$$
as $P(\boldsymbol{U}^N,\boldsymbol{U}^{N*})$
has dimension $DN$. Repeating the proof of Theorem \ref{thm:bc} now 
yields strong convergence whenever the
right-hand side is $o(1)$, that is, for $D=o(N)$. This does not suffice 
to prove the Peterson-Thom conjecture.

The above estimate was obtained by Taylor expanding the rational 
function in the polynomial method to first order. Nothing is preventing 
us, however, from expanding to higher order $m$; then 
a very similar argument yields 
$$
	\bigg| 
	\mathbf{E}\big[ \ntr 
	h(P(\boldsymbol{U}^N,\boldsymbol{U}^{N*}))\,\big]
	- \sum_{k=0}^{m}
	\frac{\nu_k(h)}{N^k} \bigg|
	\le \frac{C(m)}{N^{m+1}} \|h\|_{C^{\ell(m)}[-K,K]}
$$
where all $\nu_i$ are Schwartz distributions. The new ingredient that
now arises is that we must show that
the support of \emph{each} $\nu_i$ is included in 
$[-\|P(\boldsymbol{u},\boldsymbol{u}^*)\|,
\|P(\boldsymbol{u},\boldsymbol{u}^*)\|]$. Surprisingly, a very 
simple technique that is developed in \cite{CGV25} (see also
\cite{Par23,Par23b}) shows that
this property follows automatically in the present setting from 
concentration of measure. This yields strong convergence 
for $D=o(N^m)$ for any $m\in\mathbb{N}$. Reaching $D=e^{o(N)}$ is
harder and requires several additional ideas.

\subsubsection{Some ideas behind the reduction}

In the remainder of this section, we aim to give a hint as to how the 
purely operator-algebraic statement of Theorem \ref{thm:hayes} is reduced 
to a strong convergence problem. Since we cannot do justice to the details 
of the argument within the scope of this survey, we must content ourselves 
with an impressionistic sketch. From now on, we fix a nonamenable $M\le 
L(\mathbf{F}_r)$ with $h(M:L(\mathbf{F}_r))=0$, and aim to prove a 
contradiction.

The starting point for the proof is the following theorem of Haagerup 
and Connes \cite[Lemma 2.2]{Haa85} that provides a spectral 
characterization of amenability.

\begin{thm}[Haagerup--Connes]
\label{thm:hc}
A tracial von~Neumann algebra $(M,\tau)$ is nonamenable if and only if
there is a nontrivial projection $q\in M$ that commutes with every element 
of $M$, and unitaries $v_1,\ldots,v_r\in M$, so that $h_i=qv_i$
satisfy
$$
	\Bigg\|\frac{1}{r}\sum_{i=1}^r h_i\otimes \overline{h_i}
	\Bigg\| < 1.
$$
Here $\bar x\in B(\bar H)$ denotes the complex conjugate of an operator 
$x\in B(H)$.\footnote{%
More concretely, if 
$x\in\mathrm{M}_N(\mathbb{C})$ is a matrix, then it conjugate $\bar x$ may 
be identified with the elementwise complex conjugate of $x$; while if $x$ is 
a polynomial $P(\boldsymbol{u},\boldsymbol{u}^*)$ in the standard 
generators $u_i=\lambda(g_i)$ of 
$L(\mathbf{F}_r)$, then its conjugate $\bar x$ may be identified with 
the polynomial $\bar P(\boldsymbol{u},\boldsymbol{u}^*)$
where the coefficients of $\bar P$ 
are the complex conjugates of the coefficients of $P$.}
\end{thm}

The above spectral property is very much false for matrices:
if $H_i=QV_i$ where 
$V_1,\ldots,V_r$ are unitary matrices and $Q$ is nontrivial projection
that commutes with them, and we define the unit norm vector
$z = (\tr Q)^{-1/2}\sum_{k,l} Q_{kl} \,e_k\otimes e_l$, then
\begin{equation}
\label{eq:qexpander}
	\Bigg\langle z,
	\Bigg(
	\frac{1}{r}\sum_{i=1}^r H_i\otimes \overline{H_i}
	\Bigg)
	z\bigg\rangle =
	\frac{1}{r}\sum_{i=1}^r 
	\frac{\tr QH_iQH_i^*}{\tr Q} = 1.
\end{equation}
Of course, this just shows that $\mathrm{M}_N(\mathbb{C})$ is amenable.

Since we assumed that $M$ is nonamenable, we can choose $h_1,\ldots,h_r\in 
M$ as in Theorem \ref{thm:hc}. To simplify the discussion, let us suppose 
that $h_i=P_i(\boldsymbol{u},\boldsymbol{u}^*)$ are polynomials of the 
standard generators $\boldsymbol{u}$ of $L(\mathbf{F}_r)$: this clearly 
need \emph{not} be true in general, and we will return to this issue at 
the end of this section. Let
$$
	H_i^N = P_i(\boldsymbol{U}^N,\boldsymbol{U}^{N*}),\qquad\quad
	\tilde H_i^N = P_i(\boldsymbol{\tilde U}^N,\boldsymbol{\tilde U}^{N*}).
$$
Then strong convergence of $(\boldsymbol{U}^N\otimes\id,~\id\otimes
\boldsymbol{\tilde U}^N)$ implies that there exists $\delta>0$ with
\begin{equation}
\label{eq:hayesnona}
	\Bigg\|\frac{1}{r}\sum_{i=1}^r
	H_i^N\otimes \overline{\tilde H_i^N}
	\Bigg\| \le 1-\delta
\end{equation}
with probability $1-o(1)$ as $N\to\infty$. The crux of the proof is now to 
show that $h(M:L(\mathbf{F}_r))=0$ implies 
``microstate collapse'': with 
high probability, there is a unitary matrix $V$ so that $H_i^N\approx 
V\tilde H_i^NV^*$ for all $i$. Thus \eqref{eq:hayesnona} contradicts 
\eqref{eq:qexpander}, and we have achieved the desired conclusion.

We now aim to explain the origin of microstates collapse
without giving a precise definition of $h(M:L(\mathbf{F}_r))$. 
Roughly speaking, $h(M:L(\mathbf{F}_r))$ measures the growth rate
as $N\to\infty$ of the metric entropy with respect to the metric
$$
	d^{\mathrm{orb}}(\boldsymbol{A}^N,\boldsymbol{B}^N) =
	\inf_{V\in U(N)}
	\Bigg(
	\sum_{i=1}^r \ntr |A_i^N - VB_i^N V^*|^2
	\Bigg)^{1/2}
$$
of the set of families $\boldsymbol{A}^N=(A_1^N,\ldots,A_r^N)$ of 
$N$-dimensional matrices whose law lies in a weak$^*$ neighborhood 
of the law of $\boldsymbol{h}=(h_1,\ldots,h_r)$ (recall that 
the notion of a law was defined in the proof of Lemma 
\ref{lem:upperlower}; in particular, weak$^*$ convergence of laws is 
equivalent to weak convergence of matrices). As 
$\boldsymbol{H}^N$
converges weakly to $\boldsymbol{h}$, the following is essentially
a consequence of the definition: if
$h(M:L(\mathbf{F}_r))=0$, then for all $N$ 
sufficiently large, there 
is a set $\Omega^N \subset (\mathrm{M}_N(\mathbb{C}))^{r}$
so that 
$$
	\mathbf{P}\big[\boldsymbol{H}^N \in \Omega^N\big]=1-o(1)
$$
and $\Omega^N$ can be covered by 
$e^{o(N^2)}$ balls of radius $o(1)$ in the metric $d^{\rm orb}$.
In particular, this implies that  at least one of these balls must have 
probability greater than $e^{-o(N^2)}$; in other words, 
there exist \emph{nonrandom} $\boldsymbol{A}^N$ so that
$$
	\mathbf{P}\big[d^{\rm orb}(\boldsymbol{H}^N,\boldsymbol{A}^N)
	= o(1)\big] \ge e^{-o(N^2)}.
$$
We now conclude by a beautiful application of the concentration of measure 
phenomenon \cite{Led01}, which states in the present context that for 
\emph{any} set $\Omega$ such that
$\mathbf{P}[\boldsymbol{H}^N\in\Omega]\ge e^{-C\varepsilon^2 N^2}$, 
taking
an $\varepsilon$-neighborhood $\Omega_\varepsilon$ of $\Omega$ with 
respect to the metric $d^{\rm orb}$ yields
$\mathbf{P}[\boldsymbol{H}^N\in\Omega_\varepsilon]\ge 
1-e^{-C\varepsilon^2 N^2}$. Thus we finally obtain
$$
	\mathbf{P}\big[d^{\rm orb}(\boldsymbol{H}^N,\boldsymbol{A}^N)
	= o(1)\big] =1-o(1).
$$
Since $\boldsymbol{\tilde H}^N$ is an independent copy of
$\boldsymbol{H}^N$ and thus satisfies the same property,
it follows that 
$d^{\rm orb}(\boldsymbol{H}^N,\boldsymbol{\tilde H}^N)=o(1)$
with probability $1-o(1)$.

While we have overlooked many details in the above sketch of the proof, we 
made one simplification that is especially problematic: we assumed that 
$h_i$ are polynomials of the standard generators $\boldsymbol{u}$. In 
general, however, all we know is that $h_i$ can be approximated 
by such polynomials in the strong operator topology. This does not suffice 
for our purposes, since such an approximation need not preserve the 
conclusion of Theorem \ref{thm:hc} on the \emph{norm} of (tensor products 
of) $h_i$. Indeed, from a broader perspective, it seems surprising that 
strong convergence has anything meaningful to say about the von~Neumann 
algebra $L(\mathbf{F}_r)$: strong convergence is a statement about norms 
of polynomials, so it would appear that it should not provide any 
meaningful information on objects that live outside the norm-closure 
$C^*_{\rm red}(\mathbf{F}_r)$ of the set of polynomials of the standard 
generators.

This issue is surmounted in \cite{Hay22} by using that any given 
$h_1,\ldots,h_r\in L(\mathbf{F}_r)$ can be approximated by 
$\varphi_k(h_1),\ldots,\varphi_k(h_r)\in C^*_{\rm red}(\mathbf{F}_r)$ in a 
special way: not only do $\varphi_k(h_i)\to h_i$ in the strong
operator topology, but in addition $\varphi_k$ are contractive completely 
positive maps (this uses exactness of $C^*_{\rm red}(\mathbf{F}_r)$). 
Consequently, even though the approximation does not preserve the norm, it 
preserves the \emph{upper bound} on the norm that appears in Theorem 
\ref{thm:hc}. Since only the upper bound is needed in the proof, this
suffices to make the rest of the argument work.

\subsection{Minimal surfaces}
\label{sec:minsurf}

We finally discuss yet another unexpected application of strong 
convergence to the theory of minimal surfaces.

An immersed surface $X$ in a Riemannian manifold $M$ is called a 
\emph{minimal surface} if it is a critical point (or, what is 
equivalent in this case, a local minimizer) of the area under compact 
perturbations; think of a soap film. Minimal surfaces have fascinated 
mathematicians since the 18th century and are a major research topic in 
geometric analysis; see \cite{MJ11,CM11} for an introduction.

We will use a slightly more general notion of a minimal surface that need 
only be immersed outside a set of isolated branch points (at which the 
surface can self-intersect locally), cf.\ \cite{GOR73}. These objects, 
called \emph{branched} minimal surfaces, arise naturally when taking 
limits of immersed minimal surfaces. For simplicity we will take ``minimal 
surface'' to mean a branched minimal surface.

A basic question one may ask is how the geometry of a minimal surface is 
constrained by that of the manifold it sits in. For example, a question in 
this spirit is: can an $N$-dimensional sphere---a manifold with constant 
\emph{positive} curvature---contain a minimal surface that has constant 
\emph{negative} curvature? It was shown by Bryant~\cite{Bry85} that the 
answer is no.
\iffalse
\footnote{%
However, whether there exist closed minimal surfaces
in Euclidean spheres that have variable negative curvature remains open; 
see \cite[Problem Section, \S 101]{Yau82}.}
\fi
Thus the following result of Song~\cite{Son24}, 
which shows 
the answer is ``almost'' yes in high dimension, appears rather 
surprising.

\begin{thm}[Song]
\label{thm:song}
There exist closed minimal surfaces $X_j$ in
Euclidean unit spheres $\mathbb{S}^{N_j}$ so that the Gaussian curvature 
$K_j$ 
of $X_j$ satisfies
$$
	\lim_{j\to\infty}\frac{1}{\mathrm{Area}(X_j)}\int_{X_j}
	|K_j+8| = 0.
$$
\end{thm}

\smallskip

The minimal surfaces in this theorem arise from a random construction: one 
finds, by a variational argument, a sequence of minimal surfaces in 
finite-dimensional spheres that are symmetric under the action of a set of 
random rotations. Strong convergence is applied in the analysis in a 
non-obvious manner to understand the limiting behavior of these surfaces.

In the remainder of this section, we aim to give an impressionistic sketch 
of some of the ingredients of the proof of Theorem \ref{thm:song}. Our 
primary aim is to give a hint of the role that strong convergence plays in 
the proof.

\subsubsection{Harmonic maps}

We must first recall the connection between minimal surfaces 
and harmonic maps. If $f:X\to M$ is a map from a Riemann surface $X$ to a 
Riemannian manifold $M$, its Dirichlet energy is defined by
$$
	\mathrm{E}(f) = \frac{1}{2}\int_X |df|^2.
$$
A critical point of the energy is called a \emph{harmonic map}. If $f$ is 
weakly conformal (i.e., conformal away from branch points), then 
$\mathrm{E}(f)$ 
coincides with the area of the surface $f(X)$ in $M$. Thus a weakly 
conformal map $f$ is harmonic if and only if $f(X)$ is a minimal surface 
in $M$. See, e.g., \cite[\S 4.2.1]{Moo17}.

This viewpoint yields a variational method for constructing minimal 
surfaces. Clearly any minimizer of the energy is, by definition, a 
harmonic map. In general, such a map is not guaranteed to be weakly 
conformal. However, this will be the case if we take $X$ to be a surface 
with a unique conformal class---the thrice punctured sphere---and then a 
minimizer of $\mathrm{E}(f)$ automatically defines a minimal surface 
$f(X)$ in $M$. We will make this choice of $X$ from now on.\footnote{%
More generally, one obtains minimal surfaces by minimizing the energy both 
with respect to the map $f$ and with respect to the conformal class of 
$X$; see \cite[Theorem 4.8.6]{Moo17}.}

The construction in \cite{Son24} uses a variant of the variational 
method which produces minimal surfaces that have many 
symmetries. Let us write $X=\Gamma\backslash\mathbb{H}$, and 
consider a unitary representation $\pi_N:\Gamma\to U(N)$ with
finite range $|\pi_N(\Gamma)|<\infty$ which we view as acting on the unit 
sphere $\mathbb{S}^{2N-1}$ of $\mathbb{C}^N$ with its standard Euclidean 
metric. The following variational problem is considered in \cite{Son24}:
$$
	\mathrm{E}(X,\pi_N) = 
	\inf\bigg\{
	\frac{1}{2}\int_F |df|^2 ; ~
	f:\mathbb{H}\to\mathbb{S}^{2N-1}\text{ is }
	\pi_N\text{-equivariant}\bigg\},
$$
where $F$ is the fundamental domain
of the action of $\Gamma$ on $\mathbb{H}$.
To interpret this variational problem, note that a $\pi_N$-equivariant 
map $f:\mathbb{H}\to\mathbb{S}^{2N-1}$ can be identified with a map
$f:X^N\to\mathbb{S}^{2N-1}$ on the surface 
$X^N=\Gamma_N\backslash\mathbb{H}$, where\footnote{%
As $\pi_N$ has finite range, $\Gamma_N$ is a finite index subgroup of 
$\Gamma$ and thus $X^N$ is a finite cover of $X$.
This construction of covering spaces is different than the
one considered in section \ref{sec:buser}.}
$$
	\Gamma_N=\ker\pi_N = \{\gamma\in\Gamma:\pi_N(\gamma)=1\}.
$$
Since a minimizer $f_N$ in $\mathrm{E}(X,\pi_N)$ minimizes the Dirichlet
energy, it defines\footnote{%
Even though $X^N$ has punctures, taking the
closure of $f_N(X^N)$ yields a closed surface.
This is a nontrivial property of harmonic maps, see
\cite[\S 4.6.4]{Moo17}.}
a minimal surface $f_N(X^N)$ in
$\mathbb{S}^{2N-1}$ that has many symmetries (it contains many
rotated copies of the image $f_N(F)$ of the fundamental 
domain).

Once a minimizer $f_N$ has been chosen for every $N$, we can take 
$N\to\infty$ to obtain a limiting object. Indeed, if we embed each 
$\mathbb{S}^{2N-1}$ in the unit sphere $\mathbb{S}^\infty$ of an 
infinite-dimensional Hilbert space $H$, we can view all 
$f_N:\mathbb{H}\to\mathbb{S}^\infty$ on the same footing. Then the 
properties of harmonic maps furnish enough compactness to ensure that 
$f_N$ converges along a subsequence to a limiting map 
$f_\infty:\mathbb{H}\to\mathbb{S}^\infty$, which is 
$\pi_\infty$-equivariant for some unitary representation 
$\pi_\infty:\Gamma\to B(H)$.

\subsubsection{An infinite-dimensional model}

So far, it is not at all clear why choosing our energy-minimizing maps to 
have many symmetries helps our cause. The reason is that certain 
equivariant maps into the infinite-dimensional sphere $\mathbb{S}^\infty$ 
turn out to have remarkable properties, which will make it possible to 
realize them as the limit $f_\infty$ of the finite-dimensional minimal 
surfaces constructed above.

Recall that minimal surfaces in the spheres $\mathbb{S}^{2N-1}$ cannot 
have constant negative curvature. The situation is very different, 
however, in infinite dimension: one can isometrically embed the hyperbolic 
plane $\mathbb{H}$ in the Hilbert sphere $\mathbb{S}^\infty$ by means of 
an energy-minimizing map. What is more surprising is that this phenomenon 
is very rigid: any energy-minimizing map 
$\varphi:\mathbb{H}\to\mathbb{S}^\infty$ that is equivariant with respect 
to a certain class of representations is necessarily an isometry.

More precisely, we have the following \cite[Corollary 2.4]{Son24}. Here 
two unitary representations $\rho_1:\Gamma\to B(H_1)$ and 
$\rho_2:\Gamma\to B(H_2)$ are said to be \emph{weakly equivalent} if any 
matrix element of $\rho_1$ can be approximated uniformly on compacts by 
finite linear combinations of matrix elements of $\rho_2$, and vice versa.

\begin{thm}
\label{thm:songkey}
Let $\rho:\Gamma\to B(H)$ be a unitary representation of $\Gamma$ that is 
weakly equivalent to the regular representation $\lambda_\Gamma$.
Then any $\rho$-equivariant energy-minimizing map
$\varphi:\mathbb{H}\to\mathbb{S}^\infty$ must satisfy
$\varphi^*g_{\mathbb{S}^\infty} = \frac{1}{8}g_{\rm hyp}$, where $g_{\rm 
hyp}$ denotes the hyperbolic metric on $\mathbb{H}$
(so $\frac{1}{8}g_{\rm hyp}$ is the metric on 
$\mathbb{H}$ with constant curvature $-8$).
\end{thm}

The proof of this result is one of the main ingredients of 
\cite{Son24}. Very roughly speaking, one first produces a single $\rho$ 
and $\varphi$ that satisfy the conclusion of the theorem by an explicit 
construction; weak equivalence is then used to transfer the conclusion to 
other $\rho$ and $\varphi$ as in the theorem.

Theorem \ref{thm:songkey} explains the utility of constructing equivariant 
minimal surfaces: if we choose the sequence of representations $\pi_N$ in 
such a way that the limiting representation $\pi_\infty$ is weakly 
equivalent to the regular representation, then this will automatically 
imply that the metrics $f_N^*g_{\mathbb{S}^{2N-1}}$ on the minimal 
surfaces converge to the metric 
$f_\infty^*g_{\mathbb{S}^\infty}=\frac{1}{8}g_{\rm hyp}$ with constant 
curvature $-8$.

\subsubsection{Weak containment and strong convergence}

At first sight, none of the above appears to be related to strong 
convergence. However, the following classical result \cite[Theorem 
F.4.4]{BHV08} makes the connection immediately obvious.

\begin{prop}
\label{prop:weakequiv}
Let $\Gamma$ be a finitely generated group with generating set
$\boldsymbol{g}=(g_1,\ldots,g_r)$, and let
$\rho_1:\Gamma\to B(H_1)$ and
$\rho_2:\Gamma\to B(H_2)$ be unitary representations.
Then the following are equivalent:
\begin{enumerate}[1.]
\itemsep\smallskipamount
\item
$\rho_1$ and $\rho_2$ are weakly equivalent.
\item
$\|P(\rho_1(\boldsymbol{g}),\rho_1(\boldsymbol{g})^*)\|=
\|P(\rho_2(\boldsymbol{g}),\rho_2(\boldsymbol{g})^*)\|$
for all $P\in\mathbb{C}\langle x_1,\ldots,x_{2r}\rangle$.
\end{enumerate}
\end{prop}

\smallskip

In the present setting, $X$ is the thrice punctured sphere whose 
fundamental group is $\Gamma\simeq\mathbf{F}_2$. Thus we can define a 
random representation $\pi_N:\Gamma\to U(N)$ with finite range by choosing 
$\pi_N(g_1)=U_1^N|_{1^\perp}$ and $\pi_N(g_2)=U_2^N|_{1^\perp}$, where 
$U_1^N,U_2^N$ are independent random permutation matrices of dimension 
$N+1$ and we identified $\mathbb{C}^N\simeq \mathbb{C}^{N+1}\cap 1^\perp$. 
Since Theorem \ref{thm:bc} yields
$$
	\lim_{N\to\infty}
	\|P(\pi_N(\boldsymbol{g}),\pi_N(\boldsymbol{g})^*)\| =
	\|P(\lambda_\Gamma(\boldsymbol{g}),\lambda_\Gamma(\boldsymbol{g})^*)\|,
$$
it follows from Proposition~\ref{prop:weakequiv} that the limiting
representation $\pi_\infty$ must be weakly   
equivalent to the regular representation. Thus we obtain a sequence of
random minimal surfaces $f_N(X^N)$ in $\mathbb{S}^{2N-1}$ with the desired
property.

\section{Open problems}
\label{sec:open}

Despite rapid developments on the topic of strong convergence in recent 
years, many challenging questions remain poorly understood. We therefore 
conclude this survey by highlighting a number of open problems and 
research directions.

\subsection{Strong convergence without freeness}
\label{sec:nonfree}

Until recently, nearly all known strong convergence results were concerned 
with polynomials of \emph{independent} random matrices, and thus with 
limiting objects that are free.
As we have seen in section 
\ref{sec:buser}, however, it is of considerable interest in applications 
to achieve strong convergence in non-free settings; for example, to 
establish optimal spectral gaps for random covers of hyperbolic manifolds, 
one needs models of random permutation matrices that converge strongly 
to the regular representation of the fundamental group of the base 
manifold. Such questions are challenging, in part, because they 
give rise to complicated dependent models of random matrices.

The systematic study of strong convergence to the regular representation 
of non-free 
groups was pioneered by Magee; see the survey \cite{Mag24}. To date, 
a small number of positive results are known in this direction:
\smallskip
\begin{enumerate}[$\bullet$]
\itemsep\medskipamount
\item Louder and Magee \cite{LM25} show that there are models of
random permutation matrices that strongly converge to the regular 
representation of any \emph{fully 
residually free} group: that is, a group that locally embeds in a free 
group. The prime example of 
finitely residually free groups are surface groups.
\item Magee and Thomas \cite{MT23} show that there are models of random 
unitary (but not permutation!)\ matrices that strongly converge to the 
regular representation of any right-angled Artin group;
these are obtained from GUE matrices that act on overlapping factors
of a tensor product (see also \cite[\S 9.4]{CGV25}). This also implies
a strong convergence result for any group that virtually embeds in
a right-angled Artin group, such as fundamental groups of closed 
hyperbolic $3$-manifolds.
\item Magee, Puder, and the author \cite{MPV25} show that uniformly
random permutation representations of the fundamental groups of
orientable closed hyperbolic surfaces strongly converge to the 
regular 
representation.
\end{enumerate}
\smallskip
On the other hand, not every discrete group admits a strongly convergent 
model: there cannot be a model of random permutation matrices that 
strongly converges to the regular representation of 
$\mathbf{F}_2\times\mathbf{F}_2\times\mathbf{F}_2$ (a very special case of 
a right-angled Artin group) \cite[Proposition 2.7]{Mag24}, or a model of 
random unitary matrices that converges strongly to the regular 
representation of $\mathrm{SL}_d(\mathbb{Z})$ with $d\ge 4$ \cite{MdlS23}. 
Thus existence of strongly convergent models cannot be taken for 
granted.

To give a hint of the difficulties that arise in non-free settings,
recall that the fundamental group of a closed 
orientable surface of genus $2$ is
$$
	\Gamma = \big\langle g_1,g_2,g_3,g_4 ~ \big| ~ 
	[g_1,g_2][g_3,g_4]=1\big\rangle.
$$
The most natural random matrix model of this group is obtained by
sampling $4$-tuples of random permutation matrices 
$U_1^N,U_2^N,U_3^N,U_4^N$ 
uniformly at random from the set of such matrices that satisfy 
$[U_1^N,U_2^N][U_3^N,U_4^N]=\id$. This constraint introduces 
complicated dependencies, which causes the model to behave very 
differently than independent random permutation matrices. For example, 
unlike in the setting of section \ref{sec:poly}, the expected traces of 
monomials of these matrices are not even analytic, let alone rational, as 
a function of $\frac{1}{N}$.

For surface groups, one can use the representation theory of 
$\mathbf{S}_N$ to analyze this model;
in particular, this 
enabled Magee and Puder \cite{MP23} to show that its spectral 
statistics admit an asymptotic expansion in $\frac{1}{N}$. The proof of 
strong convergence of this model in \cite{MPV25} is made possible by an 
extension of the polynomial method to models that admit ``good'' 
asymptotic expansions.

However, even for models that look
superficially similar to surface groups,
essentially nothing is known.
For example, perhaps the the simplest fundamental group of 
a (non-orientable, finite volume) hyperbolic $3$-manifold is 
$$
	\Gamma = \big\langle g_1,g_2 ~ \big| ~
	g_1^2g_2^2=g_2g_1\big\rangle.
$$
This is the fundamental group of the Gieseking manifold, which is obtained 
by gluing the sides of a tetrahedron \cite[\S V.2]{Mag74}.
Whether sampling uniformly from the set of permutation matrices 
$U_1^N,U_2^N$ with $(U_1^N)^2(U_2^N)^2=U_2^NU_1^N$ yields
a strongly convergent model is not known. Such questions
are of considerable interest, since they provide a route to
extending Buser's conjecture to higher dimensions.

\subsection{Random Cayley graphs}
\label{sec:cayley}

Let $\boldsymbol{\sigma}=(\sigma_1,\ldots,\sigma_r)$ be i.i.d.\ uniform 
random elements of $\mathbf{S}_N$, and let $\pi_N$ be an irreducible 
representation of $\mathbf{S}_N$. The results in section 
\ref{sec:cassidy} show that the random matrices 
$\pi_N(\boldsymbol{\sigma})$ strongly converge to the regular 
representation $\lambda(\boldsymbol{g})$ of the free generators 
$\boldsymbol{g}=(g_1,\ldots,g_r)$ of $\mathbf{F}_r$ for any sequence of 
irreducible representations with $1<\dim(\pi_N)\le 
\exp(N^{\frac{1}{20}-\delta})$. What happens beyond this regime is a 
mystery: it may even be the case that strong convergence holds for 
\emph{any} irreducible representations with 
$\dim(\pi_N)\to\infty$.

Such questions are of particular interest since they are closely connected 
to the expansion of random Cayley graphs of finite groups. Let us recall 
that the \emph{Cayley graph} 
$\mathrm{Cay}(\mathbf{S}_N;\sigma_1,\ldots,\sigma_r)$ is the graph whose 
vertex set is $\mathbf{S}_N$, and whose edges are defined by connecting 
each vertex $\tau$ to its neighbors $\sigma_i\tau$ and 
$\sigma_i^{-1}\tau$ for $i=1,\ldots,r$.
Its adjacency matrix is therefore given by
$$
	A^N = 
	\lambda_{\mathbf{S}_N}(\sigma_1)+
	\lambda_{\mathbf{S}_N}(\sigma_1)^*+ \cdots +
	\lambda_{\mathbf{S}_N}(\sigma_r)+
	\lambda_{\mathbf{S}_N}(\sigma_r)^*,	
$$
where $\lambda_{\mathbf{S}_N}$ is the left-regular representation of 
$\mathbf{S}_N$. It is a folklore question whether there are sequences of 
finite groups so that, if generators are chosen independently and 
uniformly at random, the assocated Cayley graph has an optimal spectral 
gap. This question is open for any sequence of finite groups.

Now recall that every irreducible representation of a finite group is 
contained in its regular representation with multiplicity equal to its
dimension. Thus
$$
	\|A^N|_{1^\perp}\| =
	\sup_{\pi_N\ne\mathrm{triv}}
	\|\pi_N(\sigma_1)+\pi_N(\sigma_1)^* + \cdots +
	\pi_N(\sigma_r)+\pi_N(\sigma_r)^*\|,
$$
where the supremum is over all nontrivial irreducible representations 
$\pi_N$ (the trivial repesentation is removed by
restricting to $1^\perp$). Thus in order to establish optimal 
spectral gaps for Cayley graphs, we must understand the random 
matrices $\pi_N(\boldsymbol{\sigma})$ defined by \emph{all} irreducible 
representations $\pi_N$.

Note that random Cayley graphs of $\mathbf{S}_N$ cannot have an optimal 
spectral gap with probability $1-o(1)$, as there is a nontrivial 
$1$-dimensional representation (the sign representation). The latter 
produces a single eigenvalue that is distributed as twice the sum of $r$ 
independent Bernoulli variables, which exceeds the lower bound of Lemma 
\ref{lem:alonboppana} with constant probability. Thus we 
interpret the optimal spectral gap question to mean whether all 
eigenvalues, except those coming from the trivial and sign 
representations, meet the lower bound with probability $1-o(1)$.
That this is the case is a well known conjecture, see, e.g., 
\cite[Conjecture~1.6]{RS19}. 
However, to date it has not even been shown that such graphs are 
expanders, i.e., that their nontrivial eigenvalues are bounded away from 
the trivial eigenvalue as $N\to\infty$; nor is there any known 
construction of Cayley graphs of $\mathbf{S}_N$ that achieve an optimal 
spectral gap. The only known result, due to Kassabov \cite{Kas07}, is that 
there exists a choice of generators for which the Cayley graph is 
an expander.

The analogous question is of significant interest for other finite 
groups, such as $\mathrm{SL}_N(\mathbb{F}_q)$ (here we may take either 
$q\to\infty$ or $N\to\infty$). In some ways, these groups are considerably 
better understood than the symmetric group: in this setting, Bourgain and 
Gamburd \cite{BG08} (see also \cite{Tao15}) show that random Cayley graphs 
are expanders, while Lubotzky--Phillips--Sarnak \cite{LPS88} and
Margulis~\cite{Mar88} provide a 
determinstic choice of generators for which the Cayley graph has an 
optimal spectral gap. That random Cayley graphs of these groups have an 
optimal spectral gap is suggested by numerical evidence 
\cite{LR92,LR93,RS19}. However, the study of strong convergence in the 
context of such groups has so far remained out of reach.

\begin{rem}
The above questions are concerned with random Cayley graphs with a bounded
number of generators. If the number of generators is allowed to diverge
with the size of the group, rather general results are known:
expansion follows from a classical result of Alon and Roichman 
\cite{AR94}, while optimal spectral gaps were obtained by Brailovskaya and 
the author in \cite[\S 3.2.3]{BvH24}.
\end{rem}

\subsection{Representations of a fixed group}

In section \ref{sec:cassidy} and above, we always 
considered strong convergence in the context of a sequence of groups 
$\mathbf{G}_N$ of increasing size and a representation $\pi_N$ of each 
$\mathbf{G}_N$. It is a tantalizing question \cite{BC20} whether strong 
convergence might even arise when the group $\mathbf{G}$ is fixed, and 
we take a sequence of irreducible representations $\pi_N$ of $\mathbf{G}$ 
of dimension tending to infinity. Since the entropy of the random 
generators that are sampled from the group is now fixed, strong 
convergence would have to arise in this setting entirely from the 
pseudorandom behavior of high-dimensional representations.

This situation cannot arise, of course, for a finite group $\mathbf{G}$, 
since it has only finitely many irreducible representations. The question 
makes sense, however, when $\mathbf{G}$ is a compact Lie 
group. The simplest model of this kind arises when 
$\mathbf{G}=\mathrm{SU}(2)$, which has a single sequence of irreducible 
representations 
$
	\pi_N = \mathrm{sym}^N V
$
where $V$ is the standard 
representation. The question is then, if $\boldsymbol{U}=(U_1,\ldots,U_r)$ 
are sampled independently from the Haar measure on $\mathrm{SU}(2)$, 
whether $\pi_N(\boldsymbol{U})$ strongly converges to the 
regular representation $\lambda(\boldsymbol{g})$ of the free generators 
$\boldsymbol{g}=(g_1,\ldots,g_r)$ of $\mathbf{F}_r$. A special case of 
this question is discussed in detail by 
Gamburd--Jakobson--Sarnak~\cite{GJS99}, who present numerical evidence in its favor.

Let us note that while strong convergence of representations of a fixed 
group is poorly understood, the corresponding weak convergence property is 
known to hold in great generality. For example, if 
$\boldsymbol{U}=(U_1,\ldots,U_r)$ are sampled independently from the Haar 
measure on \emph{any} compact connected semisimple Lie group $\mathbf{G}$, 
and if $\pi_N$ is \emph{any} sequence of irreducible representations of 
$\mathbf{G}$ with $\dim(\pi_N)\to\infty$, then $\pi_N(\mathbf{G})$ 
converges weakly to $\lambda(\boldsymbol{g})$; see \cite[Proposition 
7.2(1)]{AG25}.

\subsection{Deterministic constructions}

To date, all known instances of the strong convergence phenomenon require 
random constructions (except in amenable situations, cf.\ \cite[\S 
2.1]{Mag24}). This is in contrast to the setting of graphs with an optimal 
spectral gap, for which explicit deterministic constructions exist and 
even predate the understanding of random graphs \cite{LPS88,Mar88}.
It remains a major challenge to achieve strong convergence by a 
deterministic construction.

A potential candidate arises from the celebrated works of 
Lubotzky--Phillips--Sarnak~\cite{LPS88} and Margulis \cite{Mar88}, who 
show that the Cayley graph of $\mathrm{PSL}_2(\mathbb{F}_q)$ defined by a 
certain explicit deterministic choice of generators has an optimal 
spectral gap.  We may therefore ask, by extension, whether the matrices 
obtained by applying the regular representation of 
$\mathrm{PSL}_2(\mathbb{F}_q)$ to these generators converge strongly to 
the regular representation of the free generators of $\mathbf{F}_r$ (cf.\ 
section \ref{sec:cayley}). 
This question was raised by Voiculescu \cite[p.\ 146]{Voi93} in an early 
paper that motivated the development of strong convergence of random 
matrices by Haagerup and Thorbj{\o}rnsen.
However, the deterministic question remains open, and
the methods of \cite{LPS88,Mar88} appear to be 
powerless for addressing this question.

Another tantalizing candidate is the following simple model.
Let $q$ be a prime and $\mathbf{P}^1(\mathbb{F}_q)$ be the projective line 
over $\mathbb{F}_q$; thus $z\in 
\mathbf{P}^1(\mathbb{F}_q)$ may take the 
values $0,1,2,\ldots,q-1,\infty$. $\mathrm{PSL}_2(\mathbb{Z})$ 
acts on $\mathbf{P}^1(\mathbb{F}_q)$ by M\"obius transformations
$$
	\begin{bmatrix}
	a & b \\
	c & d
	\end{bmatrix} z =
	\frac{az+b}{cz+d};
$$
this is just the linear action of 
$\mathrm{PSL}_2(\mathbb{Z})$ if we parametrize 
$\mathbf{P}^1(\mathbb{F}_q)$ 
by homogeneous coordinates $z=[z_1:z_2]$. Let
$\pi_q\in\mathrm{Hom}(\mathrm{PSL}_2(\mathbb{Z});\mathbf{S}_{q+1})$
be the homomorphism defined by this action, that is, $\pi_q(X)$ is the 
permutation of the elements of $\mathbf{P}^1(\mathbb{F}_q)$ that maps $z$ 
to $Xz$. The question is whether the permutation 
matrices 
$$
	(\pi_q(X_1),\ldots,\pi_q(X_r))|_{1^\perp}
$$
converge strongly to the regular representation
$$
	(\lambda_{\mathrm{PSL}_2(\mathbb{Z})}(X_1),\ldots,\lambda_{\mathrm{PSL}_2(\mathbb{Z})}(X_r))
$$
for any $X_1,\ldots,X_r\in\mathrm{PSL}_2(\mathbb{Z})$. Numerical evidence 
\cite{Buc86,LR92,LR93,RS19} supports this phenomenon,
but a mathematical understanding remains elusive.

The above convergence was conjectured by
Buck \cite{Buc86} for diffusion operators---that is, for polynomials with 
positive coefficients---and by Magee (personal communication) for 
arbitrary polynomials. Note, however, that these two conjectures are 
actually equivalent by the positivization trick 
(cf.\ Lemma~\ref{lem:pos} and Remark~\ref{rem:posnonfree}) since
$\mathrm{PSL}_2(\mathbb{Z})$ satisfies the rapid decay and unique
trace properties \cite{Cha17,BCD94}.

\iffalse
\footnote{%
The reader should beware of a subtlety: Buck works interchangeably
with $\mathrm{PSL}_2(\mathbb{Z})$ and $\mathrm{SL}_2(\mathbb{Z})$ as the
norms of diffusions operators on these two groups coincide
\cite[p.\ 292]{Buc86}. However, full strong convergence only follows for
$\mathrm{PSL}_2(\mathbb{Z})$ since
$\mathrm{SL}_2(\mathbb{Z})$ has a nontrivial center and hence 
does not have the unique trace property. I thank Michael Magee for 
pointing this out.
}
\fi

\subsection{Ramanujan constructions}
\label{sec:ramanujan}

Recall that if $A^N$ is the adjacency matrix of a $d$-regular graph with 
$N$ vertices, the lower bound of Lemma \ref{lem:alonboppana} states that
$$
	\|A^N|_{1^\perp}\| \ge 2\sqrt{d-1} - o(1)
$$
as $N$ to infinity. In this survey, we have said that a sequence 
of graphs has an \emph{optimal spectral gap} if satisfies this bound in 
reverse, that is, if
$$
	\|A^N|_{1^\perp}\| \le 2\sqrt{d-1} + o(1).
$$
However, a more precise question that has attracted significant attention 
in the literature is whether it is possible for graphs to have their
nontrivial eigenvalues be \emph{strictly} bounded by the spectral radius 
of the universal cover, without an error term: that is, can one have
$d$-regular graphs with $N$ vertices such that
$$
	\|A^N|_{1^\perp}\| \le 2\sqrt{d-1}
$$
for arbitrarily large $N$? Graphs satisfying this property are called 
\emph{Ramanujan graphs}. Ramanujan graphs do indeed exist and 
can be obtained by several remarkable constructions; we refer to the 
breakthrough papers \cite{LPS88,Mar88,MSS15,HMY25}.

Whether there can be a stronger notion of strong convergence 
that generalizes Ramanujan graphs is unclear. For example, one may ask 
whether there exist $N\times N$ permutation matrices 
$\boldsymbol{U}^N=(U_1^N,\ldots,U_r^N)$ so that
\begin{equation}
\label{eq:ramanotjan}
	\|P(\boldsymbol{U}^N,\boldsymbol{U}^{N*})|_{1^\perp}\|
	\le
	\|P(\boldsymbol{u},\boldsymbol{u}^*)\|
\end{equation}
for every polynomial $P$, where $\boldsymbol{u}$ are as defined in Theorem 
\ref{thm:bc}. It cannot be the case that \eqref{eq:ramanotjan} holds 
simultaneously for all $P$ in fixed dimension $N$, since that would imply 
by Lemma \ref{lem:upperlower} that $C^*_{\rm red}(\mathbf{F}_r)$ embeds in 
$\mathrm{M}_N(\mathbb{C})$. However, we are not aware of an obstruction to 
the existence of $\boldsymbol{U}^N$ so that \eqref{eq:ramanotjan} holds 
for all polynomials $P$ with a bound on the degree $\mathrm{deg}(P)\le 
q(N)$ that diverges sufficiently slowly with $N$. A weaker form of this 
question is whether for each $P$, there exist $\boldsymbol{U}^N$ (which 
may now depend on $P$) for arbitrarily large $N$ so that 
\eqref{eq:ramanotjan} holds.

The interest in Ramanujan graphs stems in part from an analogy with number 
theory: the Ramanujan property of a graph is equivalent to the validity of 
the Riemann hypothesis for its Ihara zeta function \cite[Theorem 
7.4]{Ter11}. In the setting of hyperbolic surfaces, the analogous 
``Ramanujan property'' that a hyperbolic surface $X$ has 
$\lambda_1(X)\ge\frac{1}{4}$ is equivalent to the validity of the Riemann 
hypothesis for its Selberg zeta function \cite[\S 6]{Mur87}. An important 
conjecture of Selberg \cite{Sar95} predicts that a specific family of 
hyperbolic surfaces has this property. However, no such 
surfaces have yet been proved to exist. The results in section 
\ref{sec:buser} therefore provide additional motivation for studying 
``Ramanujan forms'' of strong convergence.

\subsection{The optimal dimension of matrix coefficients}
\label{sec:optdim}

The strong convergence problem for polynomials of $N$-dimensional random 
matrices with matrix coefficients of dimension $D_N\to\infty$ was 
discussed in section \ref{sec:subexpon} in the context of the 
Peterson-Thom conjecture. While only the case $D_N=N$ is needed for that 
purpose, the optimal range of $D_N$ for which strong convergence holds 
remains open: for both Gaussian and Haar distributed matrices, it is known 
that strong convergence holds when $D_N=e^{o(N)}$ and can fail when 
$D_N\ge e^{CN^2}$ \cite{CGV25}. Understanding what lies in between is 
related to questions in operator space theory \cite[\S 4]{Pis14}.

From the random matrix perspective, an interesting feature of this 
question is that there is a basic obstacle to going beyond subexponential 
dimension that is explained in \cite[\S 1.3.1]{CGV25}. While strong 
convergence is concerned with understanding extreme eigenvalues of a 
random matrix $X^N$, essentially all known proofs of strong convergence 
are based on spectral statistics such as $\mathbf{E}[\ntr h(X^N)]$ which 
\emph{count} eigenvalues. However, when $D_N\ge e^{CN}$ the expected 
number of eigenvalues of $P(\boldsymbol{U}^N,\boldsymbol{U}^{N*})$ away 
from the support of the spectrum of $P(\boldsymbol{u},\boldsymbol{u}^*)$ 
may not go to zero even in situations where strong convergence holds, 
because polynomials with matrix coefficients can have outlier eigenvalues 
with very large multiplicity. Thus going beyond coefficients of 
exponential dimension appears to present a basic obstacle to any method
of proof that uses trace statistics.

\subsection{The optimal scale of fluctuations}

The largest eigenvalue of a GUE matrix has 
fluctuations of order $N^{-2/3}$, and the exact (Tracy-Widom) limit 
distribution is known. The universality of this phenomenon has been the 
subject of a major research program in mathematical physics \cite{EY17}, 
and corresponding results are known for many classical models of random 
matrix theory. In a major breakthrough, Huang--McKenzie--Yau \cite{HMY25} 
recently showed that the largest nontrivial eigenvalue of a random regular 
graph has the same behavior, which implies the remarkable result that 
about $69\%$ of random regular graphs are Ramanujan.

It is natural to expect that except in degenerate situations, the same 
scale and edge statistics should arise in strong convergence problems. 
However, to date the optimal scale of fluctuations $N^{-2/3}$ has only 
been established for the norm of quadratic polynomials of Wigner matrices 
\cite{FKN24}. For polynomials of arbitrary degree and for a broader
class of models, the best known rate $N^{-1/2}$ is achieved both 
by the interpolation \cite{Par23,Par23b} and polynomial \cite{CGV25} 
methods.

There is, in fact, a good reason why this is the case. The model 
considered by Parraud in \cite{Par23,Par23b} is somewhat more general in 
that it considers polynomials of both random and deterministic matrices 
(see section \ref{sec:randet} below). In this setting, however, one can 
readily construct examples where $N^{-1/2}$ is the true order of the 
fluctuations: for example, one may take the sum of a GUE matrix and a 
deterministic matrix of rank one \cite{Pec06}. The random matrix scale
$N^{-2/3}$ can therefore only be expected to appear for polynomials
of random matrices alone.

\subsection{Random and deterministic matrices}
\label{sec:randet}

Let $\boldsymbol{G}^N=(G_1^N,\ldots,G_r^N)$ be i.i.d.\ GUE matrices, and 
let $\boldsymbol{B}^N=(B_1^N,\ldots,B_s^N)$ be deterministic matrices of 
the same dimension. It was realized by Male \cite{Mal12} that the 
Haagerup--Thorbj{\o}rnsen theorem admits the following extension: if it is 
assumed that $\boldsymbol{B}^N$ converges strongly to some limiting family 
of operators $\boldsymbol{b}$, then $(\boldsymbol{G}^N,\boldsymbol{B}^N)$ 
converges strongly to $(\boldsymbol{s},\boldsymbol{b})$ where the free 
semicircular family $\boldsymbol{s}$ is taken to be freely independent of 
$\boldsymbol{b}$ in the sense of Voiculescu. This joint strong convergence 
property of random and deterministic matrices was extended to Haar 
unitaries by Collins and Male \cite{CM14}, and was developed in 
a nonasymptotic form by Collins, Guionnet, and Parraud 
\cite{CGP22,Par21,Par23,Par23b}. 
The advantage of this formulation is that it encodes a variety of 
applications that cannot be achieved without the inclusion of 
deterministic matrices.

To date, however, strong convergence of random and deterministic 
matrices has only been amenable to analytic methods, such those of 
Haagerup--Thorbj{\o}rnsen~\cite{HT05} or the interpolation methods of 
\cite{CGP22,BBV23}. Thus a counterpart of this form of strong convergence 
for random permutation matrices remains open. The development of such a 
result is motivated by various applications \cite{ACD21,CCM24,CMP25}.

\subsection{Complex eigenvalues}
\label{sec:complex}

In contrast to the real eigenvalues of self-adjoint polynomials, 
complex eigenvalues of non-self-adjoint polynomials are much more poorly 
understood. While an upper bound on the spectral radius follows directly 
from strong convergence by Lemma~\ref{lem:specradius}, a lower bound 
on the spectral radius and convergence of the empirical distribution of 
the complex eigenvalues remain largely open. The difficulty here is 
reversed from the study of strong convergence, where an upper bound on 
the norm is typically the main difficulty and both a lower bound on the 
norm and weak convergence follow automatically by Lemma 
\ref{lem:upperlower}.

It is not even entirely obvious at first sight how the complex eigenvalue 
distribution of a non-self-adjoint operator in a $C^*$-probability space 
should be defined. The natural object of this kind, whose definition 
reduces to the complex eigenvalue distribution in the case of matrices, is 
called the \emph{Brown measure} \cite[Chapter~11]{MS17}. It is tempting to 
conjecture that if a family of random matrices $\boldsymbol{X}^N$ strongly 
converges to a family of limiting operators $\boldsymbol{x}$, then the 
empirical distribution of the complex eigenvalues of any noncommutative 
polynomial $P(\boldsymbol{X}^N,\boldsymbol{X}^{N*})$ should converge to 
$P(\boldsymbol{x},\boldsymbol{x}^*)$. To date, this has only been proved 
in the special case of quadratic polynomials of independent complex 
Ginibre matrices \cite{CGH22}.

One may similarly ask whether the intrinsic freeness principle extends to 
complex eigenvalues of non-self-adjoint random matrices. For example, 
is there a counterpart of Theorem \ref{thm:intrfree} for complex 
eigenvalues, and if so what are the objects that should appear in it?
No results of this kind have been obtained to date.

% \subsection{Beyond intrinsic freeness}
% Kikuchi matrices + free tensor model for matrix chaos

\addtocontents{toc}{\protect\setcounter{tocdepth}{0}}
\subsection*{Acknowledgments}

I first learned about strong convergence a decade or so ago from Gilles 
Pisier, who asked me about its connection with the study of nonhomogeneous 
random matrices. Only many years later did I come to fully appreciate the 
significance of this question. An informal $C^*$-seminar organized by 
Peter Sarnak at Princeton during Fall 2023 further led to many fruitful 
interactions.

I am grateful to Michael Magee and Mikael 
de la Salle who taught me various things about this topic that could not 
easily be found in the literature, and to Ben Hayes and Antoine Song for
explaining the material in sections \ref{sec:hayes}--\ref{sec:minsurf} to
me.
It is a great pleasure to thank all my 
collaborators, acknowledged throughout this survey, with whom I have 
thought about these problems.

Last but not least, I thank the organizers of Current Developments in 
Mathematics for the invitation to present this survey.

The author was supported in part by NSF grant DMS-2347954. This survey was 
written while the author was at the Institute for Advanced Study in 
Princeton, NJ, which is thanked for providing a fantastic mathematical 
environment.

\addtocontents{toc}{\protect\setcounter{tocdepth}{2}}

\bibliographystyle{abbrv}
\bibliography{ref}

\begin{thebibliography}{100}

\bibitem{AR94}
N.~Alon and Y.~Roichman.
\newblock Random {C}ayley graphs and expanders.
\newblock {\em Random Structures Algorithms}, 5(2):271--284, 1994.

\bibitem{AEK20}
J.~Alt, L.~Erd\H{o}s, and T.~Kr\"{u}ger.
\newblock The {D}yson equation with linear self-energy: spectral bands, edges
  and cusps.
\newblock {\em Doc. Math.}, 25:1421--1539, 2020.

\bibitem{AL02}
A.~Amit, N.~Linial, J.~Matou\v{s}ek, and E.~Rozenman.
\newblock Random lifts of graphs.
\newblock In {\em Proceedings of the Twelfth Annual ACM-SIAM Symposium on
  Discrete Algorithms}, SODA '01, pages 883--894, 2001.

\bibitem{AM25i}
N.~Anantharaman and L.~Monk.
\newblock Friedman-{R}amanujan functions in random hyperbolic geometry and
  application to spectral gaps, 2023.
\newblock arXiv:2304.02678.

\bibitem{AM25ii}
N.~Anantharaman and L.~Monk.
\newblock Friedman-{R}amanujan functions in random hyperbolic geometry and
  application to spectral gaps {II}, 2025.
\newblock arXiv:2502.12268.

\bibitem{Ana07}
C.~Anantharaman-{D}elaroche.
\newblock Amenability and exactness for groups, group actions and operator
  algebras, 2007.
\newblock ESI lecture notes, HAL:cel-00360390.

\bibitem{And13}
G.~W. Anderson.
\newblock Convergence of the largest singular value of a polynomial in
  independent {W}igner matrices.
\newblock {\em Ann. Probab.}, 41(3B):2103--2181, 2013.

\bibitem{ACD21}
B.~Au, G.~C\'ebron, A.~Dahlqvist, F.~Gabriel, and C.~Male.
\newblock Freeness over the diagonal for large random matrices.
\newblock {\em Ann. Probab.}, 49(1):157--179, 2021.

\bibitem{AG25}
N.~Avni and I.~Glazer.
\newblock On the {F}ourier coefficients of word maps on unitary groups.
\newblock {\em Compos. Math.}, 161(4):681--713, 2025.

\bibitem{BBV23}
A.~S. Bandeira, M.~T. Boedihardjo, and R.~{van Handel}.
\newblock Matrix concentration inequalities and free probability.
\newblock {\em Invent. Math.}, 234(1):419--487, 2023.

\bibitem{BCSV23}
A.~S. Bandeira, G.~Cipolloni, D.~Schr\"oder, and R.~{van Handel}.
\newblock Matrix concentration inequalities and free probability {II}.
  {T}wo-sided bounds and applications, 2024.
\newblock Preprint arxiv:2406.11453.

\bibitem{Bea95}
A.~F. Beardon.
\newblock {\em The geometry of discrete groups}, volume~91 of {\em Graduate
  Texts in Mathematics}.
\newblock Springer-Verlag, New York, 1995.
\newblock Corrected reprint of the 1983 original.

\bibitem{BHV08}
B.~Bekka, P.~de~la Harpe, and A.~Valette.
\newblock {\em Kazhdan's property ({T})}, volume~11 of {\em New Mathematical
  Monographs}.
\newblock Cambridge University Press, Cambridge, 2008.

\bibitem{BCD94}
M.~Bekka, M.~Cowling, and P.~de~la Harpe.
\newblock Some groups whose reduced {$C^*$}-algebra is simple.
\newblock {\em Inst. Hautes \'Etudes Sci. Publ. Math.}, (80):117--134, 1994.

\bibitem{BC22}
S.~Belinschi and M.~Capitaine.
\newblock Strong convergence of tensor products of independent {GUE} matrices,
  2022.
\newblock Preprint arXiv:2205.07695.

\bibitem{BS88}
C.~Bennett and R.~Sharpley.
\newblock {\em Interpolation of operators}, volume 129 of {\em Pure and Applied
  Mathematics}.
\newblock Academic Press, Inc., Boston, MA, 1988.

\bibitem{Ber16}
N.~Bergeron.
\newblock {\em The spectrum of hyperbolic surfaces}.
\newblock Universitext. Springer, Cham; EDP Sciences, Les Ulis, 2016.
\newblock Appendix C by Valentin Blomer and Farrell Brumley, Translated from
  the 2011 French original by Brumley [2857626].

\bibitem{Bor20}
C.~Bordenave.
\newblock A new proof of {F}riedman's second eigenvalue theorem and its
  extension to random lifts.
\newblock {\em Ann. Sci. \'{E}c. Norm. Sup\'{e}r. (4)}, 53(6):1393--1439, 2020.

\bibitem{BC19}
C.~Bordenave and B.~Collins.
\newblock Eigenvalues of random lifts and polynomials of random permutation
  matrices.
\newblock {\em Ann. of Math. (2)}, 190(3):811--875, 2019.

\bibitem{BC24}
C.~Bordenave and B.~Collins.
\newblock Norm of matrix-valued polynomials in random unitaries and
  permutations, 2024.
\newblock Preprint arxiv:2304.05714v2.

\bibitem{BC20}
C.~Bordenave and B.~Collins.
\newblock Strong asymptotic freeness for independent uniform variables on
  compact groups associated to nontrivial representations.
\newblock {\em Invent. Math.}, 237(1):221--273, 2024.

\bibitem{BG08}
J.~Bourgain and A.~Gamburd.
\newblock Uniform expansion bounds for {C}ayley graphs of {${\rm SL}_2(\Bbb
  F_p)$}.
\newblock {\em Ann. of Math. (2)}, 167(2):625--642, 2008.

\bibitem{BvH24}
T.~Brailovskaya and R.~{van Handel}.
\newblock Universality and sharp matrix concentration inequalities.
\newblock {\em Geom. Funct. Anal.}, 34(6):1734--1838, 2024.

\bibitem{BKKO17}
E.~Breuillard, M.~Kalantar, M.~Kennedy, and N.~Ozawa.
\newblock {$C^*$}-simplicity and the unique trace property for discrete groups.
\newblock {\em Publ. Math. Inst. Hautes \'Etudes Sci.}, 126:35--71, 2017.

\bibitem{BO08}
N.~P. Brown and N.~Ozawa.
\newblock {\em {$C^*$}-algebras and finite-dimensional approximations},
  volume~88 of {\em Graduate Studies in Mathematics}.
\newblock American Mathematical Society, Providence, RI, 2008.

\bibitem{Bry85}
R.~L. Bryant.
\newblock Minimal surfaces of constant curvature in {$S^n$}.
\newblock {\em Trans. Amer. Math. Soc.}, 290(1):259--271, 1985.

\bibitem{Buc86}
M.~W. Buck.
\newblock Expanders and diffusers.
\newblock {\em SIAM J. Algebraic Discrete Methods}, 7(2):282--304, 1986.

\bibitem{Bus78}
P.~Buser.
\newblock Cubic graphs and the first eigenvalue of a {R}iemann surface.
\newblock {\em Math. Z.}, 162(1):87--99, 1978.

\bibitem{Bus84}
P.~Buser.
\newblock On the bipartition of graphs.
\newblock {\em Discrete Appl. Math.}, 9(1):105--109, 1984.

\bibitem{Cas24}
E.~Cassidy.
\newblock Random permutations acting on $k$-tuples have near-optimal spectral
  gap for $k=\mathrm{poly}(n)$, 2024.
\newblock Preprint arxiv:2412.13941v2.

\bibitem{Cha17}
I.~Chatterji.
\newblock Introduction to the rapid decay property.
\newblock In {\em Around {L}anglands correspondences}, volume 691 of {\em
  Contemp. Math.}, pages 53--72. Amer. Math. Soc., Providence, RI, 2017.

\bibitem{CGTV25}
C.-F. Chen, J.~Garza-Vargas, J.~A. Tropp, and R.~{van Handel}.
\newblock A new approach to strong convergence.
\newblock {\em Ann. of Math.}, 2025.
\newblock To appear.

\bibitem{CGV25}
C.-F. Chen, J.~Garza-Vargas, and R.~{van Handel}.
\newblock A new approach to strong convergence {II}. {T}he classical ensembles,
  2025.
\newblock Preprint arxiv:2412.00593.

\bibitem{Che98}
E.~W. Cheney.
\newblock {\em Introduction to approximation theory}.
\newblock AMS, Providence, RI, 1998.

\bibitem{Che75}
S.~Y. Cheng.
\newblock Eigenvalue comparison theorems and its geometric applications.
\newblock {\em Math. Z.}, 143(3):289--297, 1975.

\bibitem{CCM24}
G.~Cohen, I.~Cohen, and G.~Maor.
\newblock Tight bounds for the {Z}ig-{Z}ag product.
\newblock In {\em 2024 {IEEE} 65th {A}nnual {S}ymposium on {F}oundations of
  {C}omputer {S}cience---{FOCS} 2024}, pages 1470--1499. IEEE Computer Soc.,
  Los Alamitos, CA, [2024] \copyright 2024.

\bibitem{CMP25}
G.~Cohen, I.~Cohen, G.~Maor, and Y.~Peled.
\newblock Derandomized squaring: an analytical insight into its true behavior.
\newblock In {\em 16th {I}nnovations in {T}heoretical {C}omputer {S}cience
  {C}onference}, volume 325 of {\em LIPIcs. Leibniz Int. Proc. Inform.}, pages
  Art. No. 40, 24. Schloss Dagstuhl. Leibniz-Zent. Inform., Wadern, 2025.

\bibitem{CM11}
T.~H. Colding and W.~P. Minicozzi, II.
\newblock {\em A course in minimal surfaces}, volume 121 of {\em Graduate
  Studies in Mathematics}.
\newblock American Mathematical Society, Providence, RI, 2011.

\bibitem{Col23}
B.~Collins.
\newblock Moment methods on compact groups: {W}eingarten calculus and its
  applications.
\newblock In {\em I{CM}---{I}nternational {C}ongress of {M}athematicians.
  {V}ol. 4. {S}ections 5--8}, pages 3142--3164. EMS Press, Berlin, [2023]
  \copyright 2023.

\bibitem{CGP22}
B.~Collins, A.~Guionnet, and F.~Parraud.
\newblock On the operator norm of non-commutative polynomials in deterministic
  matrices and iid {GUE} matrices.
\newblock {\em Camb. J. Math.}, 10(1):195--260, 2022.

\bibitem{CM14}
B.~Collins and C.~Male.
\newblock The strong asymptotic freeness of {H}aar and deterministic matrices.
\newblock {\em Ann. Sci. \'Ec. Norm. Sup\'er. (4)}, 47(1):147--163, 2014.

\bibitem{Col22}
B.~Collins, S.~Matsumoto, and J.~Novak.
\newblock The {W}eingarten calculus.
\newblock {\em Notices Amer. Math. Soc.}, 69(5):734--745, 2022.

\bibitem{Con76}
A.~Connes.
\newblock Classification of injective factors. {C}ases {$II\sb{1},$}
  {$II\sb{\infty },$} {$III\sb{\lambda },$} {$\lambda \not=1$}.
\newblock {\em Ann. of Math. (2)}, 104(1):73--115, 1976.

\bibitem{CGH22}
N.~A. Cook, A.~Guionnet, and J.~Husson.
\newblock Spectrum and pseudospectrum for quadratic polynomials in {G}inibre
  matrices.
\newblock {\em Ann. Inst. Henri Poincar\'e{} Probab. Stat.}, 58(4):2284--2320,
  2022.

\bibitem{dlS10}
M.~de~la Salle.
\newblock Complete isometries between subspaces of noncommutative
  {$L_p$}-spaces.
\newblock {\em J. Operator Theory}, 64(2):265--298, 2010.

\bibitem{EY17}
L.~Erd\H{o}s and H.-T. Yau.
\newblock {\em A dynamical approach to random matrix theory}, volume~28 of {\em
  Courant Lecture Notes in Mathematics}.
\newblock Courant Institute of Mathematical Sciences, New York; American
  Mathematical Society, Providence, RI, 2017.

\bibitem{Eti14}
P.~Etingof.
\newblock Representation theory in complex rank, {I}.
\newblock {\em Transform. Groups}, 19(2):359--381, 2014.

\bibitem{Far14}
B.~Farb.
\newblock Representation stability.
\newblock In {\em Proceedings of the {I}nternational {C}ongress of
  {M}athematicians---{S}eoul 2014. {V}ol. {II}}, pages 1173--1196. Kyung Moon
  Sa, Seoul, 2014.

\bibitem{Fri03}
J.~Friedman.
\newblock Relative expanders or weakly relatively {R}amanujan graphs.
\newblock {\em Duke Math. J.}, 118(1):19--35, 2003.

\bibitem{Fri08}
J.~Friedman.
\newblock A proof of {A}lon's second eigenvalue conjecture and related
  problems.
\newblock {\em Mem. Amer. Math. Soc.}, 195(910):viii+100, 2008.

\bibitem{FJRST96}
J.~Friedman, A.~Joux, Y.~Roichman, J.~Stern, and J.-P. Tillich.
\newblock The action of a few random permutations on {$r$}-tuples and an
  application to cryptography.
\newblock In {\em S{TACS} 96 ({G}renoble, 1996)}, volume 1046 of {\em Lecture
  Notes in Comput. Sci.}, pages 375--386. Springer, Berlin, 1996.

\bibitem{FKN24}
J.~Fronk, T.~Kr\"uger, and Y.~Nemish.
\newblock Norm convergence rate for multivariate quadratic polynomials of
  {W}igner matrices.
\newblock {\em J. Funct. Anal.}, 287(12):Paper No. 110647, 59, 2024.

\bibitem{Ful95}
W.~Fulton.
\newblock {\em Algebraic topology}, volume 153 of {\em Graduate Texts in
  Mathematics}.
\newblock Springer-Verlag, New York, 1995.
\newblock A first course.

\bibitem{GJS99}
A.~Gamburd, D.~Jakobson, and P.~Sarnak.
\newblock Spectra of elements in the group ring of {${\rm SU}(2)$}.
\newblock {\em J. Eur. Math. Soc. (JEMS)}, 1(1):51--85, 1999.

\bibitem{Gui19}
A.~Guionnet.
\newblock {\em Asymptotics of random matrices and related models}, volume 130
  of {\em CBMS Regional Conference Series in Mathematics}.
\newblock American Mathematical Society, Providence, RI, 2019.
\newblock The uses of Dyson-Schwinger equations, Published for the Conference
  Board of the Mathematical Sciences.

\bibitem{GOR73}
R.~D. Gulliver, II, R.~Osserman, and H.~L. Royden.
\newblock A theory of branched immersions of surfaces.
\newblock {\em Amer. J. Math.}, 95:750--812, 1973.

\bibitem{Haa78}
U.~Haagerup.
\newblock An example of a nonnuclear {$C\sp{\ast} $}-algebra, which has the
  metric approximation property.
\newblock {\em Invent. Math.}, 50(3):279--293, 1978/79.

\bibitem{Haa85}
U.~Haagerup.
\newblock Injectivity and decomposition of completely bounded maps.
\newblock In {\em Operator algebras and their connections with topology and
  ergodic theory ({B}u\c steni, 1983)}, volume 1132 of {\em Lecture Notes in
  Math.}, pages 170--222. Springer, Berlin, 1985.

\bibitem{HST06}
U.~Haagerup, H.~Schultz, and S.~Thorbj{\o}rnsen.
\newblock A random matrix approach to the lack of projections in {$C^*_{\rm
  red}(\Bbb F_2)$}.
\newblock {\em Adv. Math.}, 204(1):1--83, 2006.

\bibitem{HT05}
U.~Haagerup and S.~Thorbj{\o}rnsen.
\newblock A new application of random matrices: {${\rm Ext}(C^*_{\rm
  red}(F_2))$} is not a group.
\newblock {\em Ann. of Math. (2)}, 162(2):711--775, 2005.

\bibitem{Hal54}
P.~R. Halmos.
\newblock Commutators of operators. {II}.
\newblock {\em Amer. J. Math.}, 76:191--198, 1954.

\bibitem{Hal82}
P.~R. Halmos.
\newblock {\em A {H}ilbert space problem book}, volume~17 of {\em Encyclopedia
  of Mathematics and its Applications}.
\newblock Springer-Verlag, New York-Berlin, second edition, 1982.
\newblock Graduate Texts in Mathematics, 19.

\bibitem{HP23}
L.~Hanany and D.~Puder.
\newblock Word measures on symmetric groups.
\newblock {\em Int. Math. Res. Not. IMRN}, (11):9221--9297, 2023.

\bibitem{Hat02}
A.~Hatcher.
\newblock {\em Algebraic topology}.
\newblock Cambridge University Press, Cambridge, 2002.

\bibitem{Hay22}
B.~Hayes.
\newblock A random matrix approach to the {P}eterson-{T}hom conjecture.
\newblock {\em Indiana Univ. Math. J.}, 71(3):1243--1297, 2022.

\bibitem{HJK25}
B.~Hayes, D.~Jekel, and S.~{Kunnawalkam Elayavalli}.
\newblock Consequences of the random matrix solution to the {P}eterson-{T}hom
  conjecture.
\newblock {\em Anal. PDE}, 18(7):1805--1834, 2025.

\bibitem{HRS07}
J.~W. Helton, R.~Rashidi~Far, and R.~Speicher.
\newblock Operator-valued semicircular elements: solving a quadratic matrix
  equation with positivity constraints.
\newblock {\em Int. Math. Res. Not. IMRN}, (22):Art. ID rnm086, 15, 2007.

\bibitem{HMT25}
W.~Hide, D.~Macera, and J.~Thomas.
\newblock Spectral gap with polynomial rate for random covering surfaces, 2025.
\newblock Preprint arxiv:2505.08479.

\bibitem{HMT25b}
W.~Hide, D.~Macera, and J.~Thomas.
\newblock Spectral gap with polynomial rate for {W}eil-{P}etersson random
  surfaces, 2025.
\newblock arxiv:2508.14874.

\bibitem{HM23}
W.~Hide and M.~Magee.
\newblock Near optimal spectral gaps for hyperbolic surfaces.
\newblock {\em Ann. of Math. (2)}, 198(2):791--824, 2023.

\bibitem{HMN25}
W.~Hide, J.~Moy, and F.~Naud.
\newblock On the spectral gap of negatively curved surface covers.
\newblock {\em Int. Math. Res. Not. IMRN}, (24):Paper No. rnaf357, 24, 2025.

\bibitem{HLW06}
S.~Hoory, N.~Linial, and A.~Wigderson.
\newblock Expander graphs and their applications.
\newblock {\em Bull. Amer. Math. Soc. (N.S.)}, 43(4):439--561, 2006.

\bibitem{Hor03}
L.~H\"{o}rmander.
\newblock {\em The analysis of linear partial differential operators. {I}}.
\newblock Classics in Mathematics. Springer-Verlag, Berlin, 2003.
\newblock Distribution theory and Fourier analysis.

\bibitem{HR19}
B.~Huang and M.~Rahman.
\newblock On the local geometry of graphs in terms of their spectra.
\newblock {\em European J. Combin.}, 81:378--393, 2019.

\bibitem{HMY24}
J.~Huang, T.~McKenzie, and H.-T. Yau.
\newblock Optimal eigenvalue rigidity of random regular graphs, 2024.
\newblock Preprint arxiv:2405.12161.

\bibitem{HMY25}
J.~Huang, T.~McKenzie, and H.-T. Yau.
\newblock Ramanujan property and edge universality of random regular graphs,
  2024.
\newblock Preprint arxiv:2412.20263.

\bibitem{HY24}
J.~Huang and H.-T. Yau.
\newblock Spectrum of random {$d$}-regular graphs up to the edge.
\newblock {\em Comm. Pure Appl. Math.}, 77(3):1635--1723, 2024.

\bibitem{Hub74}
H.~Huber.
\newblock {\"U}ber den ersten {E}igenwert des {L}aplace-{O}perators auf
  kompakten {R}iemannschen {F}l\"achen.
\newblock {\em Comment. Math. Helv.}, 49:251--259, 1974.

\bibitem{Jam78}
G.~D. James.
\newblock {\em The representation theory of the symmetric groups}, volume 682
  of {\em Lecture Notes in Mathematics}.
\newblock Springer, Berlin, 1978.

\bibitem{Kas07}
M.~Kassabov.
\newblock Symmetric groups and expander graphs.
\newblock {\em Invent. Math.}, 170(2):327--354, 2007.

\bibitem{Kes59}
H.~Kesten.
\newblock Symmetric random walks on groups.
\newblock {\em Trans. Amer. Math. Soc.}, 92:336--354, 1959.

\bibitem{LR93}
J.~Lafferty and D.~Rockmore.
\newblock Numerical investigation of the spectrum for certain families of
  {C}ayley graphs.
\newblock In {\em Expanding graphs ({P}rinceton, {NJ}, 1992)}, volume~10 of
  {\em DIMACS Ser. Discrete Math. Theoret. Comput. Sci.}, pages 63--73. Amer.
  Math. Soc., Providence, RI, 1993.

\bibitem{LR92}
J.~D. Lafferty and D.~Rockmore.
\newblock Fast {F}ourier analysis for {${\rm SL}_2$} over a finite field and
  related numerical experiments.
\newblock {\em Experiment. Math.}, 1(2):115--139, 1992.

\bibitem{Led01}
M.~Ledoux.
\newblock {\em The concentration of measure phenomenon}, volume~89 of {\em
  Mathematical Surveys and Monographs}.
\newblock American Mathematical Society, Providence, RI, 2001.

\bibitem{Leh99}
F.~Lehner.
\newblock Computing norms of free operators with matrix coefficients.
\newblock {\em Amer. J. Math.}, 121(3):453--486, 1999.

\bibitem{PL10}
N.~Linial and D.~Puder.
\newblock Word maps and spectra of random graph lifts.
\newblock {\em Random Structures Algorithms}, 37(1):100--135, 2010.

\bibitem{LM25}
L.~Louder, M.~Magee, and W.~Hide.
\newblock Strongly convergent unitary representations of limit groups.
\newblock {\em J. Funct. Anal.}, 288(6):Paper No. 110803, 2025.

\bibitem{LPS88}
A.~Lubotzky, R.~Phillips, and P.~Sarnak.
\newblock Ramanujan graphs.
\newblock {\em Combinatorica}, 8(3):261--277, 1988.

\bibitem{Mag24}
M.~Magee.
\newblock Strong convergence of unitary and permutation representations of
  discrete groups, 2024.
\newblock Proceedings of the ECM, to appear.

\bibitem{MdlS23}
M.~Magee and M.~de~la Salle.
\newblock {${\rm SL}_4(\bold Z)$} is not purely matricial field.
\newblock {\em C. R. Math. Acad. Sci. Paris}, 362:903--910, 2024.

\bibitem{MdlS24}
M.~Magee and M.~de~la Salle.
\newblock Strong asymptotic freeness of {H}aar unitaries in quasi-exponential
  dimensional representations.
\newblock {\em Geom. Funct. Anal.}, 2026.
\newblock To appear.

\bibitem{MNP22}
M.~Magee, F.~Naud, and D.~Puder.
\newblock A random cover of a compact hyperbolic surface has relative spectral
  gap {$\frac{3}{16}-\varepsilon$}.
\newblock {\em Geom. Funct. Anal.}, 32(3):595--661, 2022.

\bibitem{MP23}
M.~Magee and D.~Puder.
\newblock The asymptotic statistics of random covering surfaces.
\newblock {\em Forum Math. Pi}, 11:Paper No. e15, 51, 2023.

\bibitem{MPV25}
M.~Magee, D.~Puder, and R.~{van Handel}.
\newblock Strong convergence of uniformly random permutation representations of
  surface groups, 2025.
\newblock Preprint arxiv:2504.08988.

\bibitem{MT23}
M.~Magee and J.~Thomas.
\newblock Strongly convergent unitary representations of right-angled {A}rtin
  groups.
\newblock {\em Duke Math. J.}, 2026.
\newblock To appear.

\bibitem{Mag74}
W.~Magnus.
\newblock {\em Noneuclidean tesselations and their groups}, volume Vol. 61 of
  {\em Pure and Applied Mathematics}.
\newblock Academic Press [Harcourt Brace Jovanovich, Publishers], New
  York-London, 1974.

\bibitem{Mal12}
C.~Male.
\newblock The norm of polynomials in large random and deterministic matrices.
\newblock {\em Probab. Theory Related Fields}, 154(3-4):477--532, 2012.
\newblock With an appendix by Dimitri Shlyakhtenko.

\bibitem{MSS15}
A.~W. Marcus, D.~A. Spielman, and N.~Srivastava.
\newblock Interlacing families {I}: {B}ipartite {R}amanujan graphs of all
  degrees.
\newblock {\em Ann. of Math. (2)}, 182(1):307--325, 2015.

\bibitem{Mar88}
G.~A. Margulis.
\newblock Explicit group-theoretic constructions of combinatorial schemes and
  their applications in the construction of expanders and concentrators.
\newblock {\em Problemy Peredachi Informatsii}, 24(1):51--60, 1988.

\bibitem{MJ11}
W.~H. Meeks, III and J.~P\'erez.
\newblock The classical theory of minimal surfaces.
\newblock {\em Bull. Amer. Math. Soc. (N.S.)}, 48(3):325--407, 2011.

\bibitem{MS17}
J.~A. Mingo and R.~Speicher.
\newblock {\em Free probability and random matrices}, volume~35 of {\em Fields
  Institute Monographs}.
\newblock Springer, New York; Fields Institute for Research in Mathematical
  Sciences, Toronto, ON, 2017.

\bibitem{Miy23}
A.~Miyagawa.
\newblock A short note on strong convergence of {$q$}-{G}aussians.
\newblock {\em Internat. J. Math.}, 34(14):Paper No. 2350087, 8, 2023.

\bibitem{MOP19}
S.~Mohanty, R.~{O'Donnell}, and P.~Paredes.
\newblock Explicit near-{R}amanujan graphs of every degree, 2019.
\newblock Preprint arXiv:1909.06988v3.

\bibitem{Moo17}
J.~D. Moore.
\newblock {\em Introduction to global analysis}, volume 187 of {\em Graduate
  Studies in Mathematics}.
\newblock American Mathematical Society, Providence, RI, 2017.
\newblock Minimal surfaces in Riemannian manifolds.

\bibitem{Moy25}
J.~Moy.
\newblock Spectral gap of random covers of negatively curved noncompact
  surfaces, 2025.
\newblock Preprint arxiv:2505.07056.

\bibitem{Mur87}
M.~R. Murty.
\newblock An introduction to {S}elberg's trace formula.
\newblock {\em J. Indian Math. Soc. (N.S.)}, 52:91--126, 1987.

\bibitem{Nic94}
A.~Nica.
\newblock On the number of cycles of given length of a free word in several
  random permutations.
\newblock {\em Random Structures Algorithms}, 5(5):703--730, 1994.

\bibitem{NS06}
A.~Nica and R.~Speicher.
\newblock {\em Lectures on the combinatorics of free probability}.
\newblock Cambridge, 2006.

\bibitem{OW20}
R.~{O'Donnell} and X.~Wu.
\newblock Explicit near-fully {X}-{R}amanujan graphs, 2020.
\newblock Preprint arXiv:2009.02595.

\bibitem{Par21}
F.~Parraud.
\newblock On the operator norm of non-commutative polynomials in deterministic
  matrices and iid {H}aar unitary matrices.
\newblock {\em Probab. Theory Related Fields}, 182(3-4):751--806, 2022.

\bibitem{Par23}
F.~Parraud.
\newblock Asymptotic expansion of smooth functions in polynomials in
  deterministic matrices and iid {GUE} matrices.
\newblock {\em Comm. Math. Phys.}, 399(1):249--294, 2023.

\bibitem{Par24}
F.~Parraud.
\newblock The spectrum of a tensor of random and deterministic matrices, 2024.
\newblock Preprint arXiv:2410.04481.

\bibitem{Par23b}
F.~Parraud.
\newblock Asymptotic expansion of smooth functions in deterministic and iid
  {H}aar unitary matrices, and application to tensor products of matrices.
\newblock {\em Ann. Probab.}, 2026.
\newblock To appear.

\bibitem{Pec06}
S.~P\'ech\'e.
\newblock The largest eigenvalue of small rank perturbations of {H}ermitian
  random matrices.
\newblock {\em Probab. Theory Related Fields}, 134(1):127--173, 2006.

\bibitem{PT11}
J.~Peterson and A.~Thom.
\newblock Group cocycles and the ring of affiliated operators.
\newblock {\em Invent. Math.}, 185(3):561--592, 2011.

\bibitem{Pis96}
G.~Pisier.
\newblock A simple proof of a theorem of {K}irchberg and related results on
  {$C^*$}-norms.
\newblock {\em J. Operator Theory}, 35(2):317--335, 1996.

\bibitem{Pis03}
G.~Pisier.
\newblock {\em Introduction to operator space theory}, volume 294 of {\em
  London Mathematical Society Lecture Note Series}.
\newblock Cambridge University Press, Cambridge, 2003.

\bibitem{Pis14}
G.~Pisier.
\newblock Random matrices and subexponential operator spaces.
\newblock {\em Israel J. Math.}, 203(1):223--273, 2014.

\bibitem{Pis18}
G.~Pisier.
\newblock On a linearization trick.
\newblock {\em Enseign. Math.}, 64(3-4):315--326, 2018.

\bibitem{Pow75}
R.~T. Powers.
\newblock Simplicity of the {$C\sp{\ast} $}-algebra associated with the free
  group on two generators.
\newblock {\em Duke Math. J.}, 42:151--156, 1975.

\bibitem{RS19}
I.~Rivin and N.~T. Sardari.
\newblock Quantum chaos on random {C}ayley graphs of {${\rm SL}_2[\Bbb Z/p\Bbb
  Z]$}.
\newblock {\em Exp. Math.}, 28(3):328--341, 2019.

\bibitem{Sar95}
P.~Sarnak.
\newblock Selberg's eigenvalue conjecture.
\newblock {\em Notices Amer. Math. Soc.}, 42(11):1272--1277, 1995.

\bibitem{Sch05}
H.~Schultz.
\newblock Non-commutative polynomials of independent {G}aussian random
  matrices. {T}he real and symplectic cases.
\newblock {\em Probab. Theory Related Fields}, 131(2):261--309, 2005.

\bibitem{Son24}
A.~Song.
\newblock Random harmonic maps into spheres, 2025.
\newblock Preprint arxiv:2402.10287v2.

\bibitem{Tao15}
T.~Tao.
\newblock {\em Expansion in finite simple groups of {L}ie type}, volume 164 of
  {\em Graduate Studies in Mathematics}.
\newblock American Mathematical Society, Providence, RI, 2015.

\bibitem{Ter11}
A.~Terras.
\newblock {\em Zeta functions of graphs}, volume 128 of {\em Cambridge Studies
  in Advanced Mathematics}.
\newblock Cambridge University Press, Cambridge, 2011.
\newblock A stroll through the garden.

\bibitem{Tro18}
J.~A. Tropp.
\newblock Second-order matrix concentration inequalities.
\newblock {\em Appl. Comput. Harmon. Anal.}, 44(3):700--736, 2018.

\bibitem{Voi91}
D.~Voiculescu.
\newblock Limit laws for random matrices and free products.
\newblock {\em Invent. Math.}, 104(1):201--220, 1991.

\bibitem{Voi93}
D.~Voiculescu.
\newblock Around quasidiagonal operators.
\newblock {\em Integral Equations Operator Theory}, 17(1):137--149, 1993.

\bibitem{VSW16}
D.-V. Voiculescu, N.~Stammeier, and M.~Weber, editors.
\newblock {\em Free probability and operator algebras}.
\newblock M\"unster Lectures in Mathematics. European Mathematical Society
  (EMS), Z\"urich, 2016.
\newblock Lecture notes from the masterclass held in M\"unster, September 2--6,
  2013.

\bibitem{Wig67}
E.~P. Wigner.
\newblock Random matrices in physics.
\newblock {\em SIAM Review}, 9(1):1--23, 1967.

\end{thebibliography}

\end{document}